\documentclass[11pt,a4paper]{article}

\usepackage{setspace}
\usepackage{fullpage,amssymb,amsmath,amsfonts,rotate,amsthm}
\usepackage{graphicx,algorithmic,algorithm} %coucou1
\usepackage[]{datetime}
\usepackage{hyperref,todonotes}
\usepackage{ifthen,eurosym}
\usepackage{subcaption}
\usepackage{cleveref}

\hypersetup{
 pdftitle = {Combing a Linkage in an Annulus},
 colorlinks = true,
 urlcolor = black,
 linkcolor = blue!70!black,
 citecolor = green!70!black,
 menucolor = {red}
 filecolor = {cyan},
 anchorcolor = {black}
 urlcolor = {magenta},
 runcolor = {cyan},
 linkbordercolor = {blue}
}

\newcommand{\sshow}[2]{\ifthenelse{\equal{#1}{0}}{#2}{}}

\newcommand{\yes}{{\sf yes}}
\newcommand{\bd}{{\bf bor}}
\newcommand{\ann}{{\sf ann}}
\newcommand{\inter}{{\sf int}}
\newcommand{\remove}[1]{}

\newcommand{\bigmid}{\;\big|\;}

\newcommand{\cupall}{\pmb{\pmb{\bigcup}}}
\newcommand{\seg}{partially ${\rm\Delta}$-embedded graph}

\newcommand{\disk}{{\sf disk}}
\newcommand{\dehe}{{\sf dehe}}
\newcommand{\cae}{{\sf cae}}

\newcommand{\perim}{{\sf perim}}

\newcounter{func}
\newcommand{\newfun}[1]{f_{\refstepcounter{func}\label{#1}\thefunc}}
\newcommand{\funref}[1]{\hyperref[#1]{f_{\ref*{#1}}}} % print a
\newcounter{con}

\newcommand{\conref}[1]{\hyperref[#1]{c_{\ref*{#1}}}} % print a
\usepackage[spanish,polutonikogreek,english,]{babel}
\usepackage[utf8]{inputenc}

\usepackage[T1]{fontenc} 
\usepackage{alphabeta}

\newtheorem{observation}{Observation}
\newtheorem{proposition}{Proposition}
\newtheorem{lemma}{Lemma}
\newtheorem{corollary}{Corollary}

\newtheorem{theorem}{Theorem}

\newcommand{\tw}{{\mathbf{tw}}}

\newcommand{\hh}{\end{document}}

\usepackage{float}
\usepackage{tikz,pgf}
\usetikzlibrary{backgrounds}
\usetikzlibrary{calc,3d}
\usetikzlibrary{intersections,decorations.pathmorphing,shapes,decorations.pathreplacing,fit} 
\usepackage{color}

\usetikzlibrary{arrows.meta}

\makeatletter

\pgfdeclarearrow{
  name = ipe _linear,
  defaults = {
    length = +1bp,
    width  = +.666bp,
    line width = +0pt 1,
  },
  setup code = {
    \pgfarrowssetbackend{0pt}
    \pgfarrowssetvisualbackend{
      \pgfarrowlength\advance\pgf@x by-.5\pgfarrowlinewidth}
    \pgfarrowssetlineend{\pgfarrowlength}
    \ifpgfarrowreversed
      \pgfarrowssetlineend{\pgfarrowlength\advance\pgf@x by-.5\pgfarrowlinewidth}
    \fi
    \pgfarrowssettipend{\pgfarrowlength}
    \pgfarrowshullpoint{\pgfarrowlength}{0pt}
    \pgfarrowsupperhullpoint{0pt}{.5\pgfarrowwidth}
    \pgfarrowssavethe\pgfarrowlinewidth
    \pgfarrowssavethe\pgfarrowlength
    \pgfarrowssavethe\pgfarrowwidth
  },
  drawing code = {
    \pgfsetdash{}{+0pt}
    \ifdim\pgfarrowlinewidth=\pgflinewidth\else\pgfsetlinewidth{+\pgfarrowlinewidth}\fi
    \pgfpathmoveto{\pgfqpoint{0pt}{.5\pgfarrowwidth}}
    \pgfpathlineto{\pgfqpoint{\pgfarrowlength}{0pt}}
    \pgfpathlineto{\pgfqpoint{0pt}{-.5\pgfarrowwidth}}
    \pgfusepathqstroke
  },
  parameters = {
    \the\pgfarrowlinewidth,%
    \the\pgfarrowlength,%
    \the\pgfarrowwidth,%
  },
}

\pgfdeclarearrow{
  name = ipe _pointed,
  defaults = {
    length = +1bp,
    width  = +.666bp,
    inset  = +.2bp,
    line width = +0pt 1,
  },
  setup code = {
    \pgfarrowssetbackend{0pt}
    \pgfarrowssetvisualbackend{\pgfarrowinset}
    \pgfarrowssetlineend{\pgfarrowinset}
    \ifpgfarrowreversed
      \pgfarrowssetlineend{\pgfarrowlength}
    \fi
    \pgfarrowssettipend{\pgfarrowlength}
    \pgfarrowshullpoint{\pgfarrowlength}{0pt}
    \pgfarrowsupperhullpoint{0pt}{.5\pgfarrowwidth}
    \pgfarrowshullpoint{\pgfarrowinset}{0pt}
    \pgfarrowssavethe\pgfarrowinset
    \pgfarrowssavethe\pgfarrowlinewidth
    \pgfarrowssavethe\pgfarrowlength
    \pgfarrowssavethe\pgfarrowwidth
  },
  drawing code = {
    \pgfsetdash{}{+0pt}
    \ifdim\pgfarrowlinewidth=\pgflinewidth\else\pgfsetlinewidth{+\pgfarrowlinewidth}\fi
    \pgfpathmoveto{\pgfqpoint{\pgfarrowlength}{0pt}}
    \pgfpathlineto{\pgfqpoint{0pt}{.5\pgfarrowwidth}}
    \pgfpathlineto{\pgfqpoint{\pgfarrowinset}{0pt}}
    \pgfpathlineto{\pgfqpoint{0pt}{-.5\pgfarrowwidth}}
    \pgfpathclose
    \ifpgfarrowopen
      \pgfusepathqstroke
    \else
      \ifdim\pgfarrowlinewidth>0pt\pgfusepathqfillstroke\else\pgfusepathqfill\fi
    \fi
  },
  parameters = {
    \the\pgfarrowlinewidth,%
    \the\pgfarrowlength,%
    \the\pgfarrowwidth,%
    \the\pgfarrowinset,%
    \ifpgfarrowopen o\fi%
  },
}

\newdimen\ipeminipagewidth

\tikzstyle{ipe import} = [
  x=1bp, y=1bp,
  ipe node stretch/.store in=\ipenodestretch,
  ipe stretch normal/.style={ipe node stretch=1},
  ipe stretch normal,
  ipe node/.style={
    anchor=base west, inner sep=0, outer sep=0, scale=\ipenodestretch
  },
  ipe mark scale/.store in=\ipemarkscale,
  ipe mark tiny/.style={ipe mark scale=1.5},
  ipe mark small/.style={ipe mark scale=2},
  ipe mark normal/.style={ipe mark scale=3},
  ipe mark large/.style={ipe mark scale=5},
  ipe mark normal, % Set default
  ipe circle/.pic={
    \draw[line width=0.2*\ipemarkscale]
      (0,0) circle[radius=0.5*\ipemarkscale];
    \coordinate () at (0,0);
  },
  ipe disk/.pic={
    \fill (0,0) circle[radius=0.6*\ipemarkscale];
    \coordinate () at (0,0);
  },
  ipe fdisk/.pic={
    \filldraw[line width=0.2*\ipemarkscale]
      (0,0) circle[radius=0.5*\ipemarkscale];
    \coordinate () at (0,0);
  },
  ipe box/.pic={
    \draw[line width=0.2*\ipemarkscale, line join=miter]
      (-.5*\ipemarkscale,-.5*\ipemarkscale) rectangle
      ( .5*\ipemarkscale, .5*\ipemarkscale);
    \coordinate () at (0,0);
  },
  ipe square/.pic={
    \fill
      (-.6*\ipemarkscale,-.6*\ipemarkscale) rectangle
      ( .6*\ipemarkscale, .6*\ipemarkscale);
    \coordinate () at (0,0);
  },
  ipe fsquare/.pic={
    \filldraw[line width=0.2*\ipemarkscale, line join=miter]
      (-.5*\ipemarkscale,-.5*\ipemarkscale) rectangle
      ( .5*\ipemarkscale, .5*\ipemarkscale);
    \coordinate () at (0,0);
  },
  ipe cross/.pic={
    \draw[line width=0.2*\ipemarkscale, line cap=butt]
      (-.5*\ipemarkscale,-.5*\ipemarkscale) --
      ( .5*\ipemarkscale, .5*\ipemarkscale)
      (-.5*\ipemarkscale, .5*\ipemarkscale) --
      ( .5*\ipemarkscale,-.5*\ipemarkscale);
    \coordinate () at (0,0);
  },
  /pgf/arrow keys/.cd,
  ipe arrow normal/.style={scale=1},
  ipe arrow tiny/.style={scale=.4},
  ipe arrow small/.style={scale=.7},
  ipe arrow large/.style={scale=1.4},
  ipe arrow normal,
  /tikz/.cd,
  ipe arrows/.style={
    ipe normal/.tip={
      ipe _pointed[length=1bp, width=.666bp, inset=0bp,
                   quick, ipe arrow normal]},
    ipe pointed/.tip={
      ipe _pointed[length=1bp, width=.666bp, inset=0.2bp,
                   quick, ipe arrow normal]},
    ipe linear/.tip={
      ipe _linear[length = 1bp, width=.666bp,
                  ipe arrow normal, quick]},
    ipe fnormal/.tip={ipe normal[fill=white]},
    ipe fpointed/.tip={ipe pointed[fill=white]},
    ipe double/.tip={ipe normal[] ipe normal},
    ipe fdouble/.tip={ipe fnormal[] ipe fnormal},
    ipe arc/.tip={ipe normal},
    ipe farc/.tip={ipe fnormal},
    ipe ptarc/.tip={ipe pointed},
    ipe fptarc/.tip={ipe fpointed},
  },
  ipe arrows, % Set default sizes
]

\tikzset{
  rgb color/.code args={#1=#2}{%
    \definecolor{tempcolor-#1}{rgb}{#2}%
    \tikzset{#1=tempcolor-#1}%
  },
}

\makeatother

\definecolor{Black}{rgb}{0,0, 0}
\definecolor{Blue}{rgb}{0, 0 ,1}
\definecolor{Red}{rgb}{1, 0 ,0}
\definecolor{White}{rgb}{1, 1, 1}
\definecolor{Grey}{rgb}{.6, .6, .6}
\definecolor{Mygreen}{rgb}{.0, .7, .0}
\definecolor{Yellow}{rgb}{.55,.55,0}
\definecolor{mustard}{rgb}{1.0, 0.86, 0.35}
\definecolor{applegreen}{rgb}{0.55, 0.71, 0.0}
\definecolor{darkturquoise}{rgb}{0.0, 0.51, 0.62}
\definecolor{celestialblue}{rgb}{0.29, 0.59, 0.82}
\definecolor{green-yellow}{rgb}{0.68, 1.0, 0.18}
\definecolor{crimsonglory}{rgb}{0.75, 0.0, 0.2}
\definecolor{darkmagenta}{rgb}{0.55, 0.0, 0.55}
\definecolor{lightgreen}{rgb}{0.565,0.933,0.565}
\definecolor{darkred}{rgb}{0.545,0,0}
\definecolor{darkgray}{rgb}{0.663,0.663,0.663}
\definecolor{darkorange}{rgb}{1,0.549,0}
\definecolor{darkgreen}{rgb}{0,0.392,0}
\definecolor{firebrick1}{rgb}{1,0.188,0.188}
\definecolor{deepskyblue}{rgb}{0,0.749,1}
\definecolor{papayawhip}{rgb}{1,0.937,0.835}
\definecolor{palegreen}{rgb}{0.596,0.984,0.596}

\tikzset{red node/.style={draw=red, circle, fill = red, minimum size = 4pt, inner sep = 0pt}}
\tikzset{yellow node/.style={draw=yellow, circle, fill = yellow, minimum size = 4pt, inner sep = 0pt}}
\tikzset{blue node/.style={draw=celestialblue, circle, fill =celestialblue, minimum size = 4pt, inner sep = 0pt}}
\tikzset{triangle/.style = { regular polygon, regular polygon sides=3, rotate=180}}
\tikzset{small red/.style={draw=red, triangle, fill = red, minimum size = 2pt, inner sep = 0pt}}

\tikzset{black node/.style={draw, circle, fill = black, minimum size = 4pt, inner sep = 0pt}}
\tikzset{small black node/.style={draw, circle, fill = black, minimum size = 3pt, inner sep = 0pt}}
\tikzset{tiny black node/.style={draw, circle, fill = black, minimum size = 2pt, inner sep = 0pt}}
\tikzset{tiny white node/.style={draw, circle, fill = white, minimum size = 3pt, inner sep = 0pt}}
\tikzset{model node/.style={draw=celestialblue, circle, fill = celestialblue, minimum size = 5pt, inner sep = 0pt}}
\tikzset{model node small/.style={draw=celestialblue, circle, fill = celestialblue, minimum size = 3pt, inner sep = 0pt}}
\tikzset{rep node/.style={draw=red, circle, fill = red, minimum size = 3pt, inner sep = 0pt}}
\tikzset{track node 1/.style={draw, circle, fill = black, minimum size = 2pt, inner sep = 0pt}}
\tikzset{track node 2/.style={draw=black!30!white, circle, fill = black!30!white, minimum size = 2pt, inner sep = 0pt}}
\tikzset{track node 3/.style={draw=black!10!white, circle, fill = black!10!white, minimum size = 2pt, inner sep = 0pt}}

\tikzstyle{ipe stylesheet} = [
  ipe import,
  even odd rule,
  line join=round,
  line cap=butt,
  ipe pen normal/.style={line width=0.4},
  ipe pen heavier/.style={line width=0.8},
  ipe pen fat/.style={line width=1.2},
  ipe pen ultrafat/.style={line width=2},
  ipe pen normal,
  ipe mark normal/.style={ipe mark scale=3},
  ipe mark large/.style={ipe mark scale=5},
  ipe mark small/.style={ipe mark scale=2},
  ipe mark tiny/.style={ipe mark scale=1.1},
  ipe mark normal,
  /pgf/arrow keys/.cd,
  ipe arrow normal/.style={scale=7},
  ipe arrow large/.style={scale=10},
  ipe arrow small/.style={scale=5},
  ipe arrow tiny/.style={scale=3},
  ipe arrow normal,
  /tikz/.cd,
  ipe arrows, % update arrows
  <->/.tip = ipe normal,
  ipe dash normal/.style={dash pattern=},
  ipe dash dotted/.style={dash pattern=on 1bp off 3bp},
  ipe dash dashed/.style={dash pattern=on 4bp off 4bp},
  ipe dash dash dotted/.style={dash pattern=on 4bp off 2bp on 1bp off 2bp},
  ipe dash dash dot dotted/.style={dash pattern=on 4bp off 2bp on 1bp off 2bp on 1bp off 2bp},
  ipe dash normal,
  ipe node/.append style={font=\normalsize},
  ipe stretch normal/.style={ipe node stretch=1},
  ipe stretch normal,
  ipe opacity 10/.style={opacity=0.1},
  ipe opacity 30/.style={opacity=0.3},
  ipe opacity 50/.style={opacity=0.5},
  ipe opacity 75/.style={opacity=0.75},
  ipe opacity opaque/.style={opacity=1},
  ipe opacity opaque,
]
\usepackage{enumerate}
\usepackage{cite}
\usepackage{enumitem}

\newcommand{\NP}{{\sf NP}}

\newcommand*\samethanks[1][\value{footnote}]{\footnotemark[#1]}

\begin{document}

\title{\bf\Large  Combing a Linkage in an Annulus}

\author{\bigskip\large Petr A. Golovach\thanks{Department of Informatics, University of Bergen, Norway. Supported by the Research Council of Norway  via the project MULTIVAL.\
Emails:  \texttt{petr.golovach@uib.no}.}\and\large
	Giannos Stamoulis\thanks{LIRMM, Université de Montpellier, CNRS, Montpellier, France. {Supported}  by   the ANR projects ESIGMA (ANR-17-CE23-0010), and the French-German Collaboration ANR/DFG Project UTMA (ANR-20-CE92-0027).\  Emails: \texttt{giannos.stamoulis@lirmm.fr}, \texttt{sedthilk@thilikos.info}.}
	\and\large
	Dimitrios  M. Thilikos\samethanks[2]}

\date{}

\maketitle

\begin{abstract}\noindent 
A {\em linkage} in a graph $G$  of {\em size} $k$ is a subgraph $L$ of $G$ whose connected components are $k$   paths.  The {\em pattern} of a linkage of size $k$ is the set of $k$ pairs formed  by the endpoints of these paths. A consequence of  the \textsl{Unique Linkage Theorem}  is the following:  there exists a function $f:\mathbb{N}\to\mathbb{N}$ such that if a plane graph $G$ contains a sequence $\mathcal{C}$ of at least $f(k)$ nested cycles and  a linkage of size at most $k$ whose pattern vertices  lay \textsl{outside} the outer cycle of $\mathcal{C},$ then $G$ contains a linkage with the same pattern avoiding the inner cycle of $\mathcal{C}$.
In this paper we prove the following  variant of this result: Assume that all the cycles in $\mathcal{C}$ are ``orthogonally'' traversed by a linkage $P$ and $L$ is a linkage whose pattern vertices may lay either outside the outer cycle or inside the inner cycle of $\mathcal{C}:=[C_{1},\ldots,C_{p},\ldots,C_{2p-1}]$. We prove that there are two functions $g,f:\mathbb{N}\to\mathbb{N}$, such that if $L$ has size at most $k$, $P$ has size  at least $f(k),$ and $|\mathcal{C}|\geq g(k)$, then there is 
a linkage with the same pattern as $L$ that is ``internally combed'' by $P$, in the sense that 
$L\cap C_{p}\subseteq P\cap C_{p}$. In fact, we prove this result in the most general version where the linkage $L$ is $s$-scattered: no two vertices of distinct paths of $L$ are within distance at most $s$. We deduce several variants of this result in the cases where  $s=0$ and $s>0$. These variants permit the application of the unique linkage 
theorem on several path routing problems on embedded graphs.
\end{abstract}
\medskip


\newpage
\tableofcontents

%
%\newpage

%\tableofcontents
\newpage
\section{Introduction}

\subsection{The Disjoint Paths Problem}
One of the most central problems in algorithmic graph theory and combinatorial optimization is the {\sc Disjoint Paths Problem}.
An instance of the {\sc Disjoint Paths Problem} (in short DPP) is a graph $G$ and 
a collection $\mathcal{T}=\{(s_{1},t_{1}),\ldots,(s_{k},t_{k})\}$ of pairs of terminals.
Given such an instance $(G,\mathcal{T})$, the problem asks whether 
there are $k$ vertex-disjoint  paths $P_1,\ldots,P_k$ in $G$,
where, for every $i\in\{1,\ldots,k\}$, $P_i$ is a path between $s_i$ and $t_i$.
This problem, as well as its directed, edge-disjoint, and half-integral variants, has been extensively studied (see, for example,~\cite{KawarabayashiR09anea,Kleinberg98deci,CyganMPP13thep,Schrijver03comb,NavesS09mult,Frank90pack,RoSe90})
and has numerous
applications in network routing, transportation, and VLSI design.
From the viewpoint of computational complexity, this problem is known to be \NP-complete \cite{Karp72redu}, even on planar graphs~\cite{Lynch75thee} (see~\cite{Vygen95npco,MiddendorfP93onth,KramerL84thec} for \NP-completeness of other variants of the problem).
However, for fixed values of $k$, Robertson and Seymour, in Volume XIII of their seminal Graph Minors series~\cite{RobertsonS95GMXIII}, proved that the problem can be solved in time $\mathcal{O}(n^3)$.
Moreover, the problem is solvable in linear time when the input graph is planar~\cite{Reed95root,ReedRSS91find}, or embeddable on a surface of fixed Euler genus~\cite{Kawarabayashi09algo,Reed95root}.

\subsection{The irrelevant vertex technique}

In order to design their polynomial time algorithm for DPP, Robertson and Seymour~\cite{RobertsonS95GMXIII} introduced the celebrated {\em irrelevant vertex technique}.
This technique focuses on structural characteristics of the input that may permit the detection of a vertex of the graph whose removal does not change the answer to the instance.
More formally,
Robertson and Seymour~\cite{RobertsonS95GMXIII} proved that there is a function $f:\mathbb{N}\to\mathbb{N}$ such that, for every $k\in\mathbb{N}$,
if a given an instance $(G,\mathcal{T})$, where $G$ has treewidth bigger than $f(k)$,
then it is possible,
in linear time, to find in $G$ a (non-terminal) vertex $x$
 that is {\em irrelevant} in the sense that $(G,\mathcal{T})$ is a \yes-instance of DPP if and only if 
 $(G\setminus x,\mathcal{T})$ is. The guiding idea behind the application of this technique is, as long as the treewidth of the graph is bigger than $f(k)$,
 to detect irrelevant vertices and remove them from the graph, in order to obtain an equivalent instance of bounded treewidth. Then, one can design a dynamic programming algorithm to solve the bounded treewidth instance.
 
In fact, the main combinatorial condition that allows to characterize a vertex (or some part of the graph) to be irrelevant is the existence of sufficiently many ``insulation layers'' surrounding the potentially irrelevant part of the graph.
These insulation layers permit ``rerouting'' any possible solution ``away'' from the insulated part, which then can be safely deleted. 
To prove this rerouting 
argument is a quite technical task and the proofs span Volumes XXI and XII of the Graph Minors series~\cite{RobertsonS09XXI,RobertsonS12XXII}.
This result, known as the Unique Linkage Theorem, has been further studied and improved -- see~\cite{KawarabayashiW2010asho,AdlerKKLST17irre,Mazoit13asin}.
Furthermore, in~\cite{KawarabayashiKR12thedis},
Kawarabayashi, Kobayashi, and Reed proved that the insulation layers
 can be found in linear time, implying an (improved) quadratic time algorithm for DPP.
 
Adapting the above arguments for other problems that involve the identification of paths or collections of paths in graphs is a challenging task that has attracted a lot of attention from researchers, elevating the {\sl  irrelevant vertex technique}  to a standard algorithmic paradigm for solving such problems~\cite{FominGT19modif,Reed95root,JansenK021verte,DLindermayrSV20elimi,KawarabayashiMR08asim,MarxS07obta,Kawarabayashi09algo,GroheKMW11find,KawarabayashiK10impr,Kawarabayashi07half,KawarabayashiR09hadw,KawarabayashiLR10reco,GolovachKPT09indu,FominLST12line,TakehiroKPT09para,Thilikos12grap,KawarabayashiKM10link,KaminskiN12find,KaminskiT12cont,FominLRS11sube,HeggernesHLP11obta,KawarabayashiR10oddc,AdlerGK08comp,CattellDDFL00onco,Kawarabayashi09plan,FominLP0Z20hittin,BasteST20acomp,SauST20anftp,SauST21amor,SauST21kapiI,SauST21kapiII,GolovachST20hitti,GolovachKMT17thep,AgrawalKLPRSZ22dele,FominGSST21acomp,FominGST20analgo}.

\paragraph{Linkages.}
In~\cite{RobertsonS95GMXIII}, Robertson and Seymour defined the notion of {\sl linkages}.
A {\em linkage} $L$ of a graph is a collection of pairwise vertex-disjoint  paths, called the {\em paths} of $L$.
The {\em size} of $L$ is the number of paths of $L$.
The endpoints of the paths of a linkage are called {\em terminals}. The {\em pattern} of a linkage $L$ is the set of pairs of endpoints of the paths in $L$. Two linkages of a graph are {\em equivalent} if they have the same pattern.
Using this terminology, to declare a vertex $v\in V(G)$ {\sl irrelevant} for DPP, one has to prove that for every linkage between the terminals in $\mathcal{T}$, there is a linkage $L'$ of $G$ that is equivalent to $L$ and does not contain the vertex $v$.

\paragraph{Insulation of vertices.}
The combinatorial structure in~\cite{RobertsonS95GMXIII} that
allows the ``rerouting'' of linkages away from a vertex is a sequence of  ~``insulation layers''.
In plane graphs, this is interpreted as a sequence of {\sl nested} cycles $\mathcal{C}=[C_1,\ldots,C_p]$, each $C_i$ ``cropping'' an open disk $D_i$ such that, for every $i\in\{1,\ldots, p-1\}$, $D_i$ is contained in $D_{i+1}$. The cycle $C_1$ is called the {\em inner} cycle of $\mathcal{C}$ and the cycle $C_p$ is called the {\em outer} cycle of $\mathcal{C}$.
Under this setting, the consequence of the Unique Linkage Theorem that we are interested in is the following:

\begin{proposition}\label{propo_informal}
There is a function $f:\mathbb{N}\to\mathbb{N}$ such that, for every $k\in\mathbb{N}$, if a plane graph $G$ contains a sequence of at least $f(k)$ nested cycles $\mathcal{C}=[C_1,\ldots,C_p]$, then for every linkage $L$ of size at most $k$ whose terminals are outside $D_p$, there is an equivalent linkage $L'$ that avoids $D_1$.
\end{proposition}

\subsection{Our results}
In this paper, we prove a variant of~\autoref{propo_informal}, that allows to handle linkages in a graph in a more ``disciplined'' way. This variant permits the application of the Unique Linkage Theorem on several routing problems on embedded graphs and has already served as the combinatorial base of~\cite{BasteST20acomp,SauST20anftp,SauST21amor,SauST21kapiI,SauST21kapiII,GolovachST20hitti,FominGSST21acomp}.
In order to present it, we introduce some additional definitions.

\paragraph{Railed annuli.}
In our result, we demand a ``richer'' structure that this of nested cycles.
In fact, we assume that we are given a sequence of nested cycles $\mathcal{C}$ that is ``orthogonally'' traversed by a linkage $P$, meaning that the intersection of every path of $P$ with every cycle of $\mathcal{C}$ is a (possibly trivial) path.
We call this graph a {\em railed annulus}, we denote it by $\mathcal{A}=(\mathcal{C},P)$,
and we refer to $P$ as the {\em rails} of $\mathcal{A}$.

Given a plane graph $G$ and a railed annulus $\mathcal{A}=(\mathcal{C},P)$ of $G$,
where $\mathcal{C} = [C_1,\ldots,C_{2p-1}]$ for some $p\in\mathbb{N}$,
we say that a  linkage $L$ of $G$ is {\em combed} in $\mathcal{A}$, if  $L$ ``crosses'' $C_p$ through the rails of $\mathcal{A}$, i.e., $L\cap C_p \subseteq P\cap C_p$.

\begin{theorem}[Informal]\label{thm_infor}
There are two functions $g,f:\mathbb{N}\to\mathbb{N}$, such that if
$G$ is a plane graph, ${\cal A} =(\mathcal{C},{P})$ is a railed annulus where $|\mathcal{C}|\geq f(k)$ and $P$ has size at least $g(k)$, and $L$ is an linkage of size at most $k$ whose terminals are either outside the outer cycle of $\mathcal{C}$ or inside the inner cycle of $\mathcal{C}$,
then there is a linkage $L'$ equivalent to $L$ that is combed in $\mathcal{A}$.
\end{theorem}

In fact, we prove~\autoref{thm_infor} in a more general setting.
The planarity condition is only necessary for the part of the graph bounded by the inner and outer cycle of $\mathcal{C}$.
Therefore, it suffices to demand for $G$ to be {\em partially annulus-embedded}, which intuitively means that, given a closed annulus $\Delta$ (i.e., a set homeomorphic to $\{(x,y)\in \mathbb{R}^2 \mid 1\leq x^2+y^2\leq 2\}$),
there is a subgraph $K$ of $G$ embedded in $\Delta$ and there is neither an edge in $G$ from the interior of $\Delta$ to the part of $G$ outside $\Delta$, nor an edge connecting the two ``parts'' of $G$ ``cropped out'' from $\Delta$.
Analogously, we define $\Delta$-embedded railed annuli ${\cal A} =(\mathcal{C},{P})$, where $\Delta$ is the closed annulus bounded by the inner and the outer cycle of $\mathcal{C}$.

\begin{theorem}[Informal]\label{thm_partiallyinfor}
There are two functions $g,f:\mathbb{N}\to\mathbb{N}$, such that if $\Delta$ is a closed annulus, 
$G$ is a partially $\Delta$-embedded graph, ${\cal A} =(\mathcal{C},{P})$ is a $\Delta$-embedded railed annulus where $|\mathcal{C}|\geq f(k)$ and $P$ has size at least $g(k)$, and $L$ is an linkage of size at most $k$ whose terminals are outside $\Delta$,
then there is a linkage $L'$ equivalent to $L$ that is combed in $\mathcal{A}$.
\end{theorem}

\paragraph{Linkage reducible graph classes.}
The functions $f$ and $g$ in~\autoref{thm_partiallyinfor} are, in general,  ``immense''.
In fact, the function $g$ is asymptotically equal to the function $f_{\sf ul}$ of the Unique Linkage Theorem~\cite{KawarabayashiW10asho,RobertsonS09XXI} and $f(k)$ is asymptotically equal to $(g(k))^2$.
This function $f_{\sf ul}$ was improved to a single-exponential function on $k$ in the case of planar graphs~\cite{AdlerKKLST17irre}.
In~\cite{Mazoit13asin}, Mazoit improved $f_{\sf ul}$ to a single exponential function on $k+g$ for graphs embedded in a surface of Euler genus at most $g$.

We say that a graph class $\mathcal{G}$ is {\em linkage reducible} if it is hereditary (i.e., if $G\in\mathcal{G}$ then for every $S\subseteq V(G)$, the subgraph of $G$ induced by the vertices in $S$ belongs to $\mathcal{G}$) and
if there is a function $f_{{\cal G}}:\mathbb{N}\to\mathbb{N}$ such that for every $k\in\mathbb{N}$ and every $G\in {\cal G}$,
if the treewidth of $G$ is at least $f_{{\cal G}}(k)$ and $G$ contains a linkage $L$ of size at most $k$, then
there is a vertex $v\in V(G)$ such that $G\setminus v$ contains a linkage $L'$ that is equivalent to $L$.
\autoref{propo_informal} implies that the class of all graphs is linkage reducible.

Our results are modulable for every linkage reducible graph class (see the statement of~\autoref{afsfsdfdsdsafsadffasdasfd2} in~\autoref{subsec_main}).
In this sense, the result of Mazoit~\cite{Mazoit13asin} implies single-exponential upper bounds on $k+g$ for the functions $f$ and $g$ of~\autoref{thm_partiallyinfor}, if we restrict ourselves to graphs embedded in surfaces of Euler genus at most $g$.

\paragraph{Scattered linkages.}
A interesting variant of DPP is the one where we demand the vertex-disjoint paths  $P_1,\ldots, P_k$ between the terminals to be {\em induced}, i.e., for every $i\in[k]$ and every $v\in V(P_i)$, $v$ has no neighbors in $\bigcup_{j\neq i}V(P_j)$.
This problem has been extensively studied~\cite{Kawarabayashi09algo,KawarabayashiK12alin,GolovachKP13,KawarabayashiK08thei} and is \NP-complete even for $k=2$~\cite{KawarabayashiK08thei}.

In general, given an integer $r\geq 0$, we say that a linkage is {\em $r$-scattered} if for every $i\in[k]$ and every $v\in V(P_i)$,  there is no vertex in $\bigcup_{j\neq i}V(P_j)$ that is in distance at most $r$ from $v$.
Therefore, in a $0$-scattered linkage we ask its paths to be vertex-disjont, while in a $1$-scattered linkage we ask its paths to be induced.

We prove our results in terms of $r$-scattered linkages (\autoref{thm_partiallyinfor} is the special case where $r=0$)
for classes that are $r$-linkage reducibile, i.e., defining linkage reducibility by demanding that the equivalent linkages in this definition are also $r$-scattered.

\begin{theorem}\label{thm_informal_moregeneral}
There are two functions $g,f:\mathbb{N}^2\to\mathbb{N}$, such that for every $r,k\in\mathbb{N}$, if
${\cal G}$ is an $r$-linkage reducible graph class,
$\Delta$ is a closed annulus, 
$G$ is a partially $\Delta$-embedded graph that belongs to $\mathcal{G}$,
${\cal A} =(\mathcal{C},{P})$ is a $\Delta$-embedded railed annulus where $|\mathcal{C}|\geq f(k,r)$ and $P$ has size at least $g(k,r)$, and $L$ is an $r$-scattered linkage of size at most $k$ whose terminals are outside $\Delta$,
then there is an $r$-scattered linkage $L'$ equivalent to $L$ that is combed in $\mathcal{A}$.
\end{theorem}

For general positive values of $r$, the only known $r$-linkage reducible graph class is the class of planar graphs for $r=1$ and this is implied by the results of Kawarabayashi and Kobayashi~\cite{KawarabayashiK08thei}.
Using the ideas of~\cite{KawarabayashiK08thei} and the version of the Unique Linkage Theorem for surfaces of Mazoit~\cite{Mazoit13asin}, we complement our results by proving the following:

\begin{theorem}\label{thm_informalreducibility}
For every $r\in\mathbb{N}$ and every $g\in\mathbb{N}$,
the class of graphs embeddable on a surface of Euler genus $g$ is $r$-linkage reducible.
\end{theorem}

\paragraph{Organization of the paper.}
In~\autoref{sec_overview}, we present an overview of the proofs of~\autoref{thm_informal_moregeneral} and~\autoref{thm_informalreducibility}.
Then, in~\autoref{sec_preliminaries}, we provide formal definitions and statements of our main results.
The proof of~\autoref{thm_informal_moregeneral} is presented in~\autoref{sec_tamelinkage} and the proof of~\autoref{thm_informalreducibility} is presented in~\autoref{sec_irrfriendlybdgenus}.
\section{Overview of the two main proofs}\label{sec_overview}
In this section, we present a high-level description of the proofs of~\autoref{thm_informal_moregeneral} and~\autoref{thm_informalreducibility}.

\subsection{Linkage Combing Lemma}
In~\autoref{thm_informal_moregeneral} we are given an $r$-linkage reducible graph class $\mathcal{G}$,
a graph $G\in\mathcal{G}$ that is partially $\Delta$-embedded, for some closed annulus $\Delta$, and a $\Delta$-embedded railed annulus $\mathcal{A}=(\mathcal{C},P)$, where $\mathcal{C}=[C_1,\ldots, C_{2p-1}]$, and an $r$-scattered linkage $L$ of $G$ of size at most $k$, whose terminals avoid $\Delta$ (from now on, we call such a linkage {\em $\Delta$-avoiding}).
Our objective is to reroute the different paths of $L$ in a way that,
if they intersect the ``central'' cycle $C_p$ of $\mathcal{A}$, then this intersection should be part of the rails $P$.

\paragraph{Minimal linkages.}
A crucial notion in our arguments is the one of {\sl minimal linkages}.
Intuitively, a {\em minimal linkage} with respect to $L$ and $\mathcal{C}$ is a linkage $L'$ that is equivalent to $L$ and the number of edges of $L'$ that are not edges of cycles in $\mathcal{C}$ is minimal among all linkages equivalent to $L$ (see~\autoref{subsec_minimal} for a formal definition).
Such a minimal linkage ``diverges'' the least possible from $\mathcal{C}$.
Its minimality of $L'$ implies that the treewidth of the graph $H_{L',\mathcal{C}}$ obtained by the union of $L'$ and the cycles in $\mathcal{C}$ is bounded by $f_{{\cal G},r}(k)$, where $f_{{\cal G},r}$ is the function bounding the treewidth of the graphs in the definition of $r$-linkage reducibility of $\mathcal{G}$(see~\autoref{ap43k9s}).
\medskip

In the rest of this subsection we distinguish the paths of a linkage into two types. First, we have {\em rivers} and {\em streams} that are parts of the linkage that intersect both the inner and the outer cycle of $\mathcal{C}$. The difference between rivers and streams is that a river is the maximal subpath of a path of $L$ that is inside $\Delta$.
Also, the parts of the paths of the linkage that ``enter'' and ``exit'' the annulus from the same side and never intersect the other side are called {\em mountains} and {\em valleys}, depending whether the terminals of the corresponding path of $L$ are from the ``side'' of the outer cycle or of the inner cycle, respectively.
In what follows, we sketch how to prove that a minimal linkage $L'$ has at most $f_{{\cal G},r}(k)$ rivers and its mountains (resp. valleys) have ``height'' (resp. ``depth'') at most $f_{{\cal G},r}(k)$, using the fact that the treewidth of  $H_{L',\mathcal{C}}$ is at most $f_{{\cal G},r}(k)$
and then how to reroute the (few) rivers of $L'$ in order to comb $L'$ through the central cycle of $\mathcal{A}$.

\paragraph{Minimal linkages have few rivers.}
Assuming that the size of $\mathcal{C}$ is larger than $f_{{\cal G},r}(k)$,
one can observe that if $L'$ has more than $f_{{\cal G},r}(k)$ rivers,
then $H_{L',\mathcal{C}}$ contains a bramble of size larger than $f_{{\cal G},r}(k)$, a structure that forces the treewidth to be larger than $f_{{\cal G},r}(k)$, implying a contradiction to the fact that the treewidth of  $H_{L',\mathcal{C}}$ is at most $f_{{\cal G},r}(k)$ (see~\autoref{subsec_rivers} and in particular~\autoref{u4l49rop0r} and~\autoref{fskfsl}).

\paragraph{Minimal linkages do not have high mountains or deep valleys.}
The {\em height} (resp. {\em depth}) of a mountain (resp. valley) of $L'$ measures the ``intrusion'' of the mountain (resp. valley) in $\mathcal{C}$, i.e., how many cycles it intersects (see~\autoref{subsec_mountains} for formal definitions of mountains and valleys).
To prove that mountains (resp. valleys) have height (resp. depth) at most $f_{{\cal G},r}(k)$,
we first show that all mountains and valleys of $L'$ are {\em tight}, i.e., they cannot be ``compressed'' so as to intersect less cycles of $\mathcal{C}$ (see~\autoref{ao7ui4jkhq0f}).
Then, since the existence of a {\sl tight} mountain (resp. valley) of ``big enough'' height (resp. depth) also implies the existence of a ``large enough'' bramble, we obtain an upper bound  on the height (resp. depth) of mountains (resp. valleys) of $L'$ (see~\autoref{ahks5llozn} and~\autoref{ato954jgd}).
This also implies that $L'$ cannot intersect the inner part of ``sufficiently insulated'' sequence of nested closed disks (\autoref{aop4icl}).

\paragraph{Combing the linkage.}
Having all above tools, the proof of~\autoref{thm_informal_moregeneral} works as follows:
We consider a minimal linkage $L'$ with respect to $L$ and $\mathcal{C}$.
We also consider the sequence of nested cycles obtained from $\mathcal{A}$ as in~\autoref{asdfsdfgdhfgsdhffsdhnsfdhgsa}, which bound a sequence of disks that are subsets of the closed annulus $\Delta$. The size of $\mathcal{C}$ and $P$ allows to take a sufficiently large such sequence so as the inner disk defined in this sequence, denoted by $D$, is not intersected by any mountain or valley of $L'$. The latter follows from the fact that the mountains and the valleys of $L'$ have ``small'' height and depth, respectively.
The disk $D$ is situated in the ``center'' of $\mathcal{A}$, in the sense that it intersects only some ``central'' cycles of $\mathcal{C}$.
These cycles are also intersected by the (few) rivers of $L'$.
Using the railed annulus infrastructure,
we can reroute the rivers of $L'$ inside $D$ and then prove that they can be ``combed'' through the rails of $\mathcal{A}$ (see~\autoref{alo3qx}).

\subsection{Irrelevant vertices for scattered linkages}
After proving~\autoref{thm_informal_moregeneral}, our task is to provide a proof for~\autoref{thm_informalreducibility}.
The proof of~\autoref{thm_informalreducibility} is based on ideas of used in~\cite{KawarabayashiK08thei} to deal with the Induced Disjoint Paths Problem in planar graphs.
We generalize these ideas for $r$-scattered linkages and for graphs embedded in more general surfaces.
For this, we make use of the version of the Unique Linkage Theorem for graphs embedded in surfaces proved in~\cite{Mazoit13asin}.

The proof of~\autoref{thm_informalreducibility} boils down to the proof of the following result:
there is a function $f:\mathbb{N}^3\to\mathbb{N}$ such that for every $r,g,k\in\mathbb{N}$, if $G$ is a graph embedded on some surface $\Sigma$ of Euler genus $g$, $L$ is a $\Delta$-avoiding $r$-scattered linkage of $G$ of size at most $k$, where $\Delta$ is an open disk of $\Sigma$, and $\mathcal{C}$ is a $\Delta$-nested sequence of $f(r,k,g)$ cycles of $G$, and $v$ is a vertex of $G$ that is inside the disk ``cropped'' by the inner cycle of $\mathcal{C}$, then there is an $r$-scattered linkage $L'$ of $G\setminus v$ that is equivalent to $L$. The case where $L$ does not contain $v$ directly implies the result so we focus in the case where $L$ contains $v$.

In order to prove the above result, we proceed to define a notion of minimal $r$-scattered linkages, minimizing its number of bridges and crossings with respect to $\mathcal{C}$ and we show that such a linkage can be rerouted in order to avoid $v$.

\paragraph{Bridges and crossings.}
Given an open disk $\Delta$ of a surface $\Sigma$, a graph $G$ embedded in $\Sigma$ and a $\Delta$-avoiding linkage $L$ of $G$, we define the {\em bridges} of $L$ to be the maximal subpaths of $L$ that do not contain any vertex inside $\Delta$.
A {\em crossing} of  $L$ with a $\Delta$-nested sequence of cycles $\mathcal{C}$ of $G$ is a subpath of $L$ that ``transverses'' a cycle of $\mathcal{C}$.
We consider a $\Delta$-nested sequence of cycles $\mathcal{C}$ of $G$ such that every two consecutive cycles in the sequence have distance at least $r$ such that $v$ is inside the disk ``cropped'' by the inner cycle of $\mathcal{C}$ that is inside $\Delta$.
Also, we consider {\em BC-minimal linkages around $v$},
 that are linkages that are equivalent to $L$ that contain $v$ and have minimum number of bridges and crossings with respect to the $\Delta$-nested sequence of cycles $\mathcal{C}$.
 We observe that, if the number of bridges of a BC-minimal linkage $L'$ is less than $|\mathcal{C}|-\lfloor k/2\rfloor-1$, then $L'$ has at most $|\mathcal{C}-1|$ components in $\Delta$ and therefore $L'$ can be rerouted in the graph obtained from the union of $L'$ and the cycles in $\mathcal{C}$ in order to obtained an $r$-scattered linkage equivalent to $L'$ that also avoids $v'$ (\autoref{fewbridges}).
 
\paragraph{Rainbows.}
We aim to shoe that any given BC-minimal linkage around $v$ has at most $|\mathcal{C}|-\lfloor k/2\rfloor-1$ bridges.
In order to do this, we classify bridges in {\sl rainbows} around $\Delta$.
A {\em $\Delta$-rainbow} of $L$ is a sequence of $\Delta$-bridges of $L$ that are {\sl homotopic}, in the sense that they pairwise bound closed disk of $\Sigma$ (there is not handle of the surface interfering between them). Note that the number of homotopy classes of cycles in a surface is bounded by a (linear) function on the Euler genus (see~\autoref{jfopdsn} and~\autoref{rainbow}).
Suppose, towards a contradiction, that there are ``many enough'' $\Delta$-bridges of $L$.
Then ``sufficiently large'' collection ${\cal B}$ of them should have the following property:
they belong to the same homotopy class and
the two ``marginal'' bridges in this homotopy class bound an open disk that contains the rest and does not contain any terminal of $L'$.
The $\Delta$-bridges in $\mathcal{B}$ together with parts of the cycles of $\mathcal{C}$ can define a sequence of nested cycles around the ``central'' bridge in $\mathcal{B}$.
Then, using the result of~\cite{Mazoit13asin}, we can prove that there is a linkage equivalent to $L'$ that avoids this central bridge, implying a contradiction to the BC-minimality of $L'$.
In fact, to impose $r$-scatteredness, we contract some edges of$L$ and we apply the result of~\cite{Mazoit13asin}, using a trick appeared in~\cite{KawarabayashiK08thei} for the case where $r=1$, which we generalize to arbitrary $r\in\mathbb{N}$.

\section{Preliminaries}\label{sec_preliminaries}
In this section, we provide formal definitions and statements of our results.
In~\autoref{subsec_def}, we provide some basic definitions on graphs, (partial) embeddability of graphs in subsets of the plane, and railed annuli.
Then, in~\autoref{subsec_main}, we introduce some definitions concerning linkages and we state our main result (\autoref{afsfsdfdsdsafsadffasdasfd2}).
In~\autoref{afsfsdfdsdsafsadffasdasfd2}, the size perquisites on the underlying combinatorial structure of the given graph $G$ are governed by a function $f_{{\cal G},r}$, where $\mathcal{G}$ is a class where $\mathcal{G}$ belongs that enjoys an irrelevant-vertex-type reduction of its $r$-scattered linkages if the treewidth of $G$ is at least $f_{{\cal G},r}(|L|)$.
By the Unique Linkage Theorem~\cite{KawarabayashiW10asho,RobertsonS09XXI} (see also~\autoref{i94opq}), when considering linkages that are $0$-scattered, there is such a function $f$ for the class of all graphs.
This function is huge but it can become single exponential for the class of graphs embedded in some fixed surface; and this holds for linkages of any given scatteredness.
We present this in more detail in~\autoref{subsec_implic}, where we state the analogous versions of~\autoref{afsfsdfdsdsafsadffasdasfd2} for the two aforementioned compromises between graph classes, scatteredness, and size dependence.

\subsection{Definitions}\label{subsec_def}

We denote by $\mathbb{N}$ the set of non-negative integers.
Given two integers $p$ and
$q,$ the set $[p,q]$ refers to the set of every integer $r$ such that $p \leq r \leq q.$
For an integer $p\geq 1,$ we set $[p]=[1,p]$ and $\mathbb{N}_{\geq p}=\mathbb{N}\setminus [0,p-1].$

\paragraph{Basic concepts on graphs.}
All graphs in this paper are undirected, finite, and they do not have loops or multiple edges.
If $G_1 = (V_1, E_1)$ and $G_2 = (V_2, E_2)$ are graphs,
then we denote $G_1 \cap G_2 = (V_1 \cap V_2,E_1\cap E_2)$ and
$G_1\cup G_2 = (V_1\cup V_2,E_1 \cup E_2)$.
Also, given a graph $G$ and a set $S\subseteq V(G)$,
we denote by $G\setminus S$ the graph obtained if we remove from $G$ the vertices in $S$,
along with their incident edges.
A {\em path} (resp. {\em cycle}) in a graph $G$ is a
connected subgraph with all vertices of degree at most (resp. exactly) 2.
A path is {\em trivial} if it has only one vertex and it is {\em empty} if it is the empty graph (i.e., the graph with empty vertex set).
Given a graph $G$, an $r\in\mathbb{N}$, and a $v\in V(G)$, we define the {\em $r$-neighborhood} $N_G^{(\leq r)}(v)$ of $v$ in $G$ to be the set of all vertices $u\in V(G)$ such that there is a path of length at most $r$ between $u$ and $v$ in $G$.

\paragraph{Disk and annuli on the plane.}
A {\em cycle} is a set homeomorphic to the set $\{(x,y)\in \mathbb{R}^2 \mid x^2+y^2 = 1\}$.
We define a {\em closed disk} (resp. {\em open disk}) to be a set homeomorphic to the set $\{(x,y)\in \mathbb{R}^2\mid x^2+y^2\leq 1\}$
(resp. $\{(x,y)\in \mathbb{R}^2\mid x^2+y^2< 1\}$) and
a {\em closed annulus} (resp. {\em open annulus}) to be a set homeomorphic to the set $\{(x,y)\in \mathbb{R}^2 \mid 1\leq x^2+y^2\leq 2\}$ (resp. $\{(x,y)\in \mathbb{R}^2 \mid 1< x^2+y^2< 2\}$).
Given a closed disk or a closed annulus $X$, we use $\bd(X)$ to denote the boundary of $X$ (i.e., the set of points of $X$ for which every neighborhood around them contains some point not in $X$).
Notice that if $X$ is a closed disk then $\bd(X)$ is a cycle, while if $X$ is a closed annulus then $\bd(X)=C_{1}\cup C_{2}$ where $C_{1}, C_{2}$ are the two unique connected components of $\bd(X)$ and $C_1,C_2$ are two disjoint cycles.
We call $C_1$ and $C_2$ {\em boundaries} of $X.$
We call $C_1$ the {\em left boundary} of $X$ and $C_2$ the {\em right boundary} of $X$.
Also given a closed disk (resp. closed annulus) $X$, we use $\inter(X)$ to denote the open disk (resp. open annulus) $X\setminus \bd(X)$.
When we embed a graph $G$ in the plane, in a closed disk, or in a closed annulus, we treat G as a set of points.
This permits us to make set operations between graphs and sets of points.

\paragraph{Partially embedded graphs.}
Given a graph $G$, we say that a pair $(L,R)\in 2^{V(G)}\times 2^{V(G)}$ is a {\em separation} of $G$
if $L\cup R=V(G)$ and there is no edge in $G$ between a vertex in $L\setminus R$ and a vertex in $R\setminus L.$
We say that two separations  $(X_1,Y_1)$ and $(X_2,Y_2)$ of a graph $G$ are {\em laminar} if $Y_1\subseteq Y_2$ and $X_2\subseteq X_1$.

Given a closed disk $\Delta$, we say that a graph $G$ is {\em partially $\Delta$-embedded}, 
if there is some subgraph $K$ of $G$ that is embedded in $\Delta$
such that $\bd(\Delta)$ is a cycle of $K$ and $(V(G)\cap \Delta,V(G)\setminus\inter(\Delta))$
is a separation of $G$.
Similarly, given a closed annulus $\Delta$, we say that a graph $G$ is {\em partially $\Delta$-embedded}, 
if there is some subgraph $K$ of $G$ that is embedded in $\Delta$
such that $\bd(\Delta)$ is the disjoint union of two cycles of $K$ and there are two laminar separations $(X_1,Y_1)$ and $(X_2,Y_2)$ of $G$ such that $X_1\cap Y_2 =V(G)\cap \Delta$.
In both above cases, we call the graph $K$
{\em compass}
of the partially $\Delta$-embedded graph $G$ and we always assume that we accompany
a partially $\Delta$-embedded graph $G$ together with an embedding of its compass in $\Delta$ that is the set $G\cap \Delta$.

\paragraph{Parallel cycles.}
Let $\Delta$ be a closed annulus with boundaries $B_1, B_2$, and let $G$ be a partially $\Delta$-embedded graph.
Also, let $\mathcal{C} = [C_1,\ldots, C_p]$, $p\geq 2$ be a collection of vertex disjoint cycles of $G$ that are embedded in $\Delta$.
We say that $\mathcal{C}$ is a {\em $\Delta$-parallel sequence of cycles} of $G$ if
$C_1 = B_1$, $C_p = B_2$ and, for every $i\in [r-1]$, $C_i$ and $C_p$ are the boundaries of a closed annulus, denoted by $D_i$, that is a subset of $\Delta$ such that $\Delta = D_1\supseteq \cdots \supseteq D_{p-1}$.
%We call $C_1$ the {\em leftmost} and $C_r$ the {\em rightmost} cycle of $\mathcal{C}$.
From now on, each $\Delta$-parallel sequence $\mathcal{C}= [C_1,\ldots, C_p]$, $p\geq 2$ of cycles will be accompanied with the sequence $[D_1, \ldots, D_{p-1}]$ of the corresponding closed annuli.
We call $|{\cal C}|$ the {\em size} of ${\cal C}$.

If $\Delta$ is a closed disk, then a sequence $\mathcal{C}$ of vertex disjoint cycles of a partially $\Delta$-embedded graph $G$ is called {\em $\Delta$-nested}, if every $C_{i}$ is the boundary of a closed  disk $D_{i}$ of $\Delta$ such that  $\Delta\supseteq D_{1}\supseteq\cdots \supseteq D_{r}$.
Each $\Delta$-nested sequence $\mathcal{C}$ will be accompanied 
with the sequence $[D_{1},\ldots,D_{r}]$ of the corresponding closed disks.

In both above cases (i.e., either $\Delta$ is a closed annulus or a closed disk),
given $i,j\in[p-1]$, where $i\leq j$, we call the set $D_i\setminus {\sf int}(D_j)$ {\em $(i,j)$-annulus} of $\mathcal{C}$ and we denote it by $\ann(\mathcal{C},i,j)$. Also, for every $i\in[p-1]$, we set $D_i$ to be the {\em $(i,p)$-annulus} of $\mathcal{C}$ and we also denote it by $\ann(\mathcal{C},i,p)$.

%From now on, each ∆-nested sequence C will be accompanied with the sequence [D1,...,Dr] of the corresponding open disks as well as the sequence [D1,...,Dr] of their closures. Given x,y ∈ [r] where x ≤ y, we call the set Dx\Dy (x,y)-annulus of C and we denote it by ann(C,x,y). Finally, we say that ann(C,1,r) is the annulus of C and we denote it by ann(C).

\begin{figure}[ht]
	\centering
	
	\sshow{0}{\begin{tikzpicture}[scale=.5]
	\foreach \x in {4,...,12}{
		\draw[line width =0.6pt, name path=cycle\x] (0,0) circle (0.5*\x cm);
	}
	\begin{scope}[on background layer]
	\fill[black!10!white] (0,0) circle (6 cm);
	\fill[white] (0,0) circle (2 cm);
	\end{scope}
	
\node (P3) at (45:7) {$P_{3}$};

	\node[small black node] (P11) at (45:6) {};
	\node[small black node] (m21) at (40:5.5) {};
	\node[small black node] (P21a) at (30:5) {};
	\node[small black node] (P21b) at (40:5) {};
	\node[small black node] (m31) at (40:4.5) {};
	\node[small black node] (P31a) at (35:4) {};
	\node[small black node] (P31b) at (50:4) {};
	\node[small black node] (m41) at (45:3.5) {};
	\node[small black node] (P41a) at (45:3) {};
	\node[small black node] (P41b) at (25:3) {};
	\node[small black node] (m51) at (25:2.5) {};
	\node[small black node] (P51) at (40:2) {};
	
	\draw[line width=1pt] (P11) -- (m21) -- (P21a) -- (P21b) -- (m31) --(P31a) -- (P31b) -- (m41) -- (P41a) -- (P41b) -- (m51) -- (P51);

\node (P4) at (70:7) {$P_{4}$};
	
	\node[small black node] (P12) at (70:6) {};
	\node[small black node] (m22) at (80:5.5) {};
	\node[small black node] (P22a) at (80:5) {};
	\node[small black node] (P22b) at (75:5) {};
		\node[small black node] (m32) at (80:4.5) {};
	\node[small black node] (P32a) at (90:4) {};
	\node[small black node] (P32b) at (75:4) {};
		\node[small black node] (m42) at (80:3.5) {};
	\node[small black node] (P42a) at (85:3) {};
	\node[small black node] (P42b) at (70:3) {};
		\node[small black node] (m52) at (70:2.5) {};
	\node[small black node] (P52) at (75:2) {};
	
	\draw[line width=1pt] (P12) -- (m22) -- (P22a) -- (P22b) -- (m32)-- (P32a) -- (P32b) -- (m42) -- (P42a) -- (P42b) -- (m52) -- (P52);

\node (P5) at (115:7) {$P_{5}$};
\node[small black node] (P13a) at (120:6) {};
	\node[small black node] (P13b) at (110:6) {};
		\node[small black node] (m23) at (105:5.5) {};
	\node[small black node] (P23a) at (110:5) {};
	\node[small black node] (P23b) at (115:5) {};
		\node[small black node] (m33) at (120:4.5) {};
	\node[small black node] (P33a) at (120:4) {};
	\node[small black node] (P33b) at (130:4) {};
		\node[small black node] (m43) at (130:3.5) {};
	\node[small black node] (P43a) at (135:3) {};
	\node[small black node] (P43b) at (120:3) {};
		\node[small black node] (m53) at (120:2.5) {};
	\node[small black node] (P53) at (125:2) {};
	
	\draw[line width=1pt] (P13a) -- (P13b) -- (m23) -- (P23a) -- (P23b) --  (m33) -- (P33a) -- (P33b) -- (m43) -- (P43a) -- (P43b) -- (m53) -- (P53);

\node (P6) at (165:7) {$P_{6}$};
	\node[small black node] (P14a) at (170:6) {};
	\node[small black node] (P14b) at (160:6) {};
			\node[small black node] (m24) at (160:5.5) {};
	\node[small black node] (P24) at (155:5) {};
			\node[small black node] (m34) at (153:4.5) {};
	\node[small black node] (P34) at (160:4) {};
			\node[small black node] (m44) at (163:3.5) {};
	\node[small black node] (P44a) at (165:3) {};
	\node[small black node] (P44b) at (180:3) {};
			\node[small black node] (m54) at (175:2.5) {};
	\node[small black node] (P54) at (170:2) {};
	
	\draw[line width=1pt] (P14a) -- (P14b) -- (m24) -- (P24) -- (m34) --  (P34) -- (m44) -- (P44a) -- (P44b) -- (m54) -- (P54);
\node (P7) at (190:7) {$P_{7}$};
	\node[small black node] (P18a) at (200:6) {};			
		\node[small black node] (m28) at (197:5.5) {};
	\node[small black node] (P28a) at (195:5) {};
	\node[small black node] (P28b) at (220:5) {};
		\node[small black node] (m38) at (215:4.5) {};
	\node[small black node] (P38a) at (210:4) {};
	\node[small black node] (P38b) at (225:4) {};
		\node[small black node] (m48) at (215:3.5) {};
	\node[small black node] (P48a) at (205:3) {};
	\node[small black node] (P48b) at (220:3) {};
		\node[small black node] (m58) at (210:2.5) {};
	\node[small black node] (P58) at (200:2) {};
	
	\draw[line width=1pt] (P18a) -- (m28) -- (P28a) to [bend right=10]  (P28b) -- (m38) --  (P38a) -- (P38b) -- (m48) -- (P48a) -- (P48b) -- (m58) -- (P58);

\node (P8) at (235:7) {$P_{8}$};	
	\node[small black node] (P15) at (235:6) {};
		\node[small black node] (m25) at (235:5.5) {};
	\node[small black node] (P25a) at (240:5) {};
	\node[small black node] (P25b) at (250:5) {};
		\node[small black node] (m35) at (245:4.5) {};
	\node[small black node] (P35b) at (245:4) {};
		\node[small black node] (m45) at (247:3.5) {};
	\node[small black node] (P45a) at (250:3) {};
	\node[small black node] (P45b) at (240:3) {};
		\node[small black node] (m55) at (236:2.5) {};
	\node[small black node] (P55) at (235:2) {};
	
	\draw[line width=1pt] (P15) -- (m25)-- (P25a)  -- (P25b) -- (m35) --  (P35b) -- (m45) -- (P45a) -- (P45b) -- (m55) -- (P55);

\node (P1) at (295:7) {$P_{1}$};	
	\node[small black node] (P16a) at (290:6) {};
	\node[small black node] (P16b) at (300:6) {};
		\node[small black node] (m26) at (300:5.5) {};
	\node[small black node] (P26a) at (295:5) {};
	\node[small black node] (P26b) at (310:5) {};
		\node[small black node] (m36) at (305:4.5) {};
	\node[small black node] (P36a) at (300:4) {};
	\node[small black node] (P36b) at (290:4) {};
		\node[small black node] (m46) at (300:3.5) {};
	\node[small black node] (P46a) at (310:3) {};
	\node[small black node] (P46b) at (325:3) {};
		\node[small black node] (m56) at (325:2.5) {};
	\node[small black node] (P56) at (320:2) {};
	
	\draw[line width=1pt] (P16a) -- (P16b) -- (m26) -- (P26a) -- (P26b) --  (m36) -- (P36a) -- (P36b) -- (m46) -- (P46a) -- (P46b) -- (m56) -- (P56);

\node (P2) at (0:7) {$P_{2}$};
	\node[small black node] (P17a) at (5:6) {};
	\node[small black node] (P17b) at (-5:6) {};
		\node[small black node] (m27) at (-5:5.5) {};
	\node[small black node] (P27a) at (0:5) {};
	\node[small black node] (P27b) at (-10:5) {};
		\node[small black node] (m37) at (-10:4.5) {};
	\node[small black node] (P37a) at (-15:4) {};
	\node[small black node] (P37b) at (0:4) {};
		\node[small black node] (m47) at (5:3.5) {};
	\node[small black node] (P47a) at (10:3) {};
	\node[small black node] (P47b) at (-5:3) {};
		\node[small black node] (m57) at (0:2.5) {};
	\node[small black node] (P57) at (-5:2) {};
	
	\draw[line width=1pt] (P17a) -- (P17b) -- (m27) --  (P27b)  -- (P27a) -- (m37) --  (P37a) -- (P37b) -- (m47) --  (P47a) -- (P47b) -- (m57) -- (P57);
					
\end{tikzpicture}}
\caption{An example of an annulus $\Delta$ depicted in grey and a $\Delta$-embedded $(9,8)$-railed annulus $\mathcal{A} = (\mathcal{C}, \mathcal{P})$.}
\label{fig_railedann}	
\end{figure}
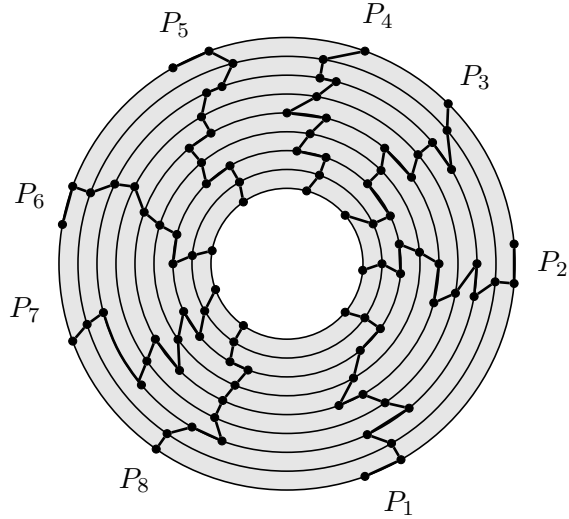

\paragraph{Railed annuli.}
Let $\Delta$ be a closed annulus
and let $G$ be a partially $\Delta$-embedded graph.
Also, let $p\in\mathbb{N}_{\geq 3}$ and $q\in \mathbb{N}_{\geq 3}$ and assume that $p$ is an odd number.
A {\em $\Delta$-embedded $(p,q)$-railed annulus} of $G$ is a pair $\mathcal{A}=(\mathcal{C},\mathcal{P})$ where 
$\mathcal{C}=[C_{1},\ldots,C_{p}]$  is a $\Delta$-parallel sequence of cycles of $G$ and $\mathcal{P}=[P_{1},\ldots,P_{q}]$ is a  collection of pairwise vertex-disjoint 
 paths in $G$ such that 
 \begin{itemize}
\item  
For every $j\in[q],$\ $P_{j}\subseteq \Delta$.
 
\item For every $(i,j)\in[p]\times[q],$   $C_{i}\cap P_{j}$ is  a non-empty path, that we denote $P_{i,j}$.%\marg{ $P_{i,j}$.}
\end{itemize}
We refer to the paths of $\mathcal{P}$ as the {\em rails} of $\mathcal{A}$ and to the cycles of $\mathcal{C}$ as the {\em cycles} of $\mathcal{A}$. We use $\ann(\mathcal{A})$ to denote $\ann(\mathcal{C},1,p)$. See~\autoref{fig_railedann} for an example.

\paragraph{Treewidth.}
A \emph{tree decomposition} of a graph~$G$
is a pair~$(T,\chi)$ where $T$ is a tree and $\chi: V(T)\to 2^{V(G)}$
such that
\begin{itemize}
	\item $\bigcup_{t \in V(T)} \chi(t) = V(G),$
	\item for every edge~$e$ of~$G$ there is a $t\in V(T)$ such that
	      $\chi(t)$
	      contains both endpoints of~$e,$ and
	\item for every~$v \in V(G),$ the subgraph of~${T}$
	      induced by $\{t \in V(T)\mid {v \in \chi(t)}\}$ is connected.
\end{itemize}
The {\em width} of $(T,\chi)$ is equal to $\max\big\{\left|\chi(t)\right|-1 \bigmid t\in V(T)\big\}$ and the {\em treewidth} of $G$ is the minimum width over all tree decompositions of $G.$

\subsection{Main result}\label{subsec_main}
Our goal is to prove \autoref{afsfsdfdsdsafsadffasdasfd2}.
Intuitively, we show that, given a graph $G$ that is partially embedded on an annulus, and a ``large enough'' railed annulus $\mathcal{A}$, any linkage of $G$ can be ``combed'' through some tracks of $\mathcal{A}$.
We start with some definitions on linkages.

\paragraph{Linkages.} 
A {\em linkage} in a graph $G$ is a subgraph $L$ of $G$ whose connected components are non-trivial paths.
The {\em paths} of a linkage are its connected components and we denote them by $\mathcal{P}(L)$.
We call $|\mathcal{P}|$ the {\em size} of $\mathcal{P}(L)$.
The {\em terminals} of a linkage $L$, denoted by $T(L)$, are the endpoints of the paths of $L$,
and the {\em pattern} of $L$ is the set $\{\{s,t\}\mid \mathcal{P}(L)\mbox{  contains some  }(s,t)\mbox{-path}\}$
(in case $L$ contains a trivial path we may see its unique endpoint as a singleton).
Two linkages $L_{1}, L_{2}$ of $G$ are {\em equivalent} if they have the same pattern and we denote this fact by $L_{1}\equiv L_{2}$.
%A linkage $L$ of $G$ is {\em $r$-scattered} if for every two connected components $C,C'$ of $L$ it holds that ${\bf dist}_{G}(C,C')\geq r$.
Let $\Delta$ be a closed annulus or a closed disk,
let $G$ be a partially $\Delta$-embedded graph, $L$ be a linkage of $G$, and $D$ be a subset of $\Delta$.
We say that $L$ is {\em $D$-avoiding} if $T(L)\cap D=\emptyset$ and we say that $L$ is {\em  $D$-free} if 
$D\cap L=\emptyset$ (see \autoref{asdfdsghdfhgdfdfsvdfgdfdsafdsf}). 

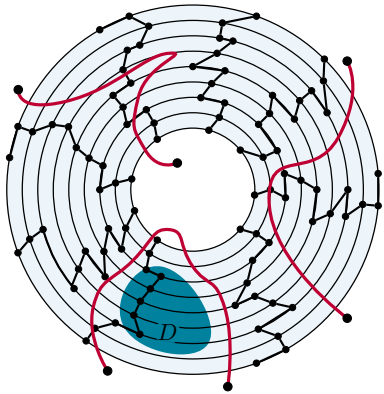
\begin{figure}[H]
	\centering
	\scalebox{0.8}{
	\sshow{0}{\begin{tikzpicture}[scale=.51]

	\fill[darkturquoise] plot [smooth cycle, tension=1.2] coordinates { (-2.2,-4.5) (.5,-5) (-1,-2.5)};
	
	\foreach \x in {2,2.5,...,6}{
		\draw[line width =0.6pt] (0,0) circle (\x cm);
	}
	\node[circle, fill=darkturquoise,opacity=.9,text opacity=1, inner sep=0pt] () at (260:4.7) {$D$};
	
	\begin{scope}[on background layer]
	\fill[celestialblue!10!white] (0,0) circle (6 cm);
	\fill[white] (0,0) circle (2 cm);
	
	\end{scope}

	\node[small black node] (P11) at (45:6) {};
	\node[small black node] (m21) at (40:5.5) {};
	\node[small black node] (P21a) at (30:5) {};
	\node[small black node] (P21b) at (40:5) {};
	\node[small black node] (m31) at (40:4.5) {};
	\node[small black node] (P31a) at (35:4) {};
	\node[small black node] (P31b) at (50:4) {};
	\node[small black node] (m41) at (45:3.5) {};
	\node[small black node] (P41a) at (45:3) {};
	\node[small black node] (P41b) at (25:3) {};
	\node[small black node] (m51) at (25:2.5) {};
	\node[small black node] (P51) at (40:2) {};
	
	\draw[line width=1pt] (P11) -- (m21) -- (P21a) -- (P21b) -- (m31) --(P31a) -- (P31b) -- (m41) -- (P41a) -- (P41b) -- (m51) -- (P51);

	\node[small black node] (P12) at (70:6) {};
	\node[small black node] (m22) at (80:5.5) {};
	\node[small black node] (P22a) at (80:5) {};
	\node[small black node] (P22b) at (75:5) {};
		\node[small black node] (m32) at (80:4.5) {};
	\node[small black node] (P32a) at (90:4) {};
	\node[small black node] (P32b) at (75:4) {};
		\node[small black node] (m42) at (80:3.5) {};
	\node[small black node] (P42a) at (85:3) {};
	\node[small black node] (P42b) at (70:3) {};
		\node[small black node] (m52) at (70:2.5) {};
	\node[small black node] (P52) at (75:2) {};
	
	\draw[line width=1pt] (P12) -- (m22) -- (P22a) -- (P22b) -- (m32)-- (P32a) -- (P32b) -- (m42) -- (P42a) -- (P42b) -- (m52) -- (P52);

\node[small black node] (P13a) at (120:6) {};
	\node[small black node] (P13b) at (110:6) {};
		\node[small black node] (m23) at (105:5.5) {};
	\node[small black node] (P23a) at (110:5) {};
	\node[small black node] (P23b) at (115:5) {};
		\node[small black node] (m33) at (120:4.5) {};
	\node[small black node] (P33a) at (120:4) {};
	\node[small black node] (P33b) at (130:4) {};
		\node[small black node] (m43) at (130:3.5) {};
	\node[small black node] (P43a) at (135:3) {};
	\node[small black node] (P43b) at (120:3) {};
		\node[small black node] (m53) at (120:2.5) {};
	\node[small black node] (P53) at (125:2) {};
	
	\draw[line width=1pt] (P13a) -- (P13b) -- (m23) -- (P23a) -- (P23b) --  (m33) -- (P33a) -- (P33b) -- (m43) -- (P43a) -- (P43b) -- (m53) -- (P53);

	\node[small black node] (P14a) at (170:6) {};
	\node[small black node] (P14b) at (160:6) {};
			\node[small black node] (m24) at (160:5.5) {};
	\node[small black node] (P24) at (155:5) {};
			\node[small black node] (m34) at (153:4.5) {};
	\node[small black node] (P34) at (160:4) {};
			\node[small black node] (m44) at (163:3.5) {};
	\node[small black node] (P44a) at (165:3) {};
	\node[small black node] (P44b) at (180:3) {};
			\node[small black node] (m54) at (175:2.5) {};
	\node[small black node] (P54) at (170:2) {};
	
	\draw[line width=1pt] (P14a) -- (P14b) -- (m24) -- (P24) -- (m34) --  (P34) -- (m44) -- (P44a) -- (P44b) -- (m54) -- (P54);

	\node[small black node] (P18a) at (200:6) {};			
		\node[small black node] (m28) at (197:5.5) {};
	\node[small black node] (P28a) at (195:5) {};
	\node[small black node] (P28b) at (220:5) {};
		\node[small black node] (m38) at (215:4.5) {};
	\node[small black node] (P38a) at (210:4) {};
	\node[small black node] (P38b) at (225:4) {};
		\node[small black node] (m48) at (215:3.5) {};
	\node[small black node] (P48a) at (205:3) {};
	\node[small black node] (P48b) at (220:3) {};
		\node[small black node] (m58) at (210:2.5) {};
	\node[small black node] (P58) at (200:2) {};
	
	\draw[line width=1pt] (P18a) -- (m28) -- (P28a) to [bend right=10]  (P28b) -- (m38) --  (P38a) -- (P38b) -- (m48) -- (P48a) -- (P48b) -- (m58) -- (P58);

	\node[small black node] (P15) at (235:6) {};
		\node[small black node] (m25) at (235:5.5) {};
	\node[small black node] (P25a) at (240:5) {};
	\node[small black node] (P25b) at (250:5) {};
		\node[small black node] (m35) at (245:4.5) {};
	\node[small black node] (P35b) at (245:4) {};
		\node[small black node] (m45) at (247:3.5) {};
	\node[small black node] (P45a) at (250:3) {};
	\node[small black node] (P45b) at (240:3) {};
		\node[small black node] (m55) at (236:2.5) {};
	\node[small black node] (P55) at (235:2) {};
	
	\draw[line width=1pt] (P15) -- (m25)-- (P25a)  -- (P25b) -- (m35) --  (P35b) -- (m45) -- (P45a) -- (P45b) -- (m55) -- (P55);

	\node[small black node] (P16a) at (290:6) {};
	\node[small black node] (P16b) at (300:6) {};
		\node[small black node] (m26) at (300:5.5) {};
	\node[small black node] (P26a) at (295:5) {};
	\node[small black node] (P26b) at (310:5) {};
		\node[small black node] (m36) at (305:4.5) {};
	\node[small black node] (P36a) at (300:4) {};
	\node[small black node] (P36b) at (290:4) {};
		\node[small black node] (m46) at (300:3.5) {};
	\node[small black node] (P46a) at (310:3) {};
	\node[small black node] (P46b) at (325:3) {};
		\node[small black node] (m56) at (325:2.5) {};
	\node[small black node] (P56) at (320:2) {};
	
	\draw[line width=1pt] (P16a) -- (P16b) -- (m26) -- (P26a) -- (P26b) --  (m36) -- (P36a) -- (P36b) -- (m46) -- (P46a) -- (P46b) -- (m56) -- (P56);

	\node[small black node] (P17a) at (5:6) {};
	\node[small black node] (P17b) at (-5:6) {};
		\node[small black node] (m27) at (-5:5.5) {};
	\node[small black node] (P27a) at (0:5) {};
	\node[small black node] (P27b) at (-10:5) {};
		\node[small black node] (m37) at (-10:4.5) {};
	\node[small black node] (P37a) at (-15:4) {};
	\node[small black node] (P37b) at (0:4) {};
		\node[small black node] (m47) at (5:3.5) {};
	\node[small black node] (P47a) at (10:3) {};
	\node[small black node] (P47b) at (-5:3) {};
		\node[small black node] (m57) at (0:2.5) {};
	\node[small black node] (P57) at (-5:2) {};
	
	\draw[line width=1pt] (P17a) -- (P17b) -- (m27) --  (P27b)  -- (P27a) -- (m37) --  (P37a) -- (P37b) -- (m47) --  (P47a) -- (P47b) -- (m57) -- (P57);

	%Linkage
	\draw[crimsonglory, line width=1.5pt] plot [smooth, tension=1.2] coordinates { (120:1) (140:2) (120:3.5) (100:4.5) (145:5) (150:6.5)};
	\draw[crimsonglory, line width=1.5pt] plot [smooth, tension=1] coordinates { (245:6.5) (230:5) (230:3) (240:1.5) (270:2.5) (285:4) (280:6.5)};
	\draw[crimsonglory, line width=1.5pt] plot [smooth, tension=1.2] coordinates { (40:6.5) (20:5) (-10:2.5) (-40:6.5)};
	\node[black node]  () at (120:1) {};
	\node[black node]  () at (150:6.5) {};
	\node[black node]  () at (245:6.5) {};
	\node[black node]  () at (280:6.5) {};
	\node[black node]  () at (40:6.5) {};
	\node[black node]  () at (-40:6.5) {};

	\end{tikzpicture}}}
	\caption{An example of a railed annulus $\mathcal{A}$, a closed disk $D$ (depicted in blue)
		and a linkage $L$ (depicted in red) that is $D$-free and $\ann({\cal A})$-avoiding.}
	\label{asdfdsghdfhgdfdfsvdfgdfdsafdsf}
\end{figure}

Given an $r\in\mathbb{N}$,
we say that a linkage $L$ of $G$ is {\em $r$-scattered} if for every $v\in V(L)$ it holds that
$N_G^{(\leq r)}(v)\cap (V(L)\setminus V(C_v)) = \emptyset,$
where $C_v$ is the connected component of $L$ that contains $v$.
Observe that, since $N_G^{(\leq 0)}(v)=\{v\}$, every linkage is $0$-scattered.

\paragraph{Linkages confined in annuli.}
Let $t\in\mathbb{N}_{\geq 1}$, let  $p=2t+1$, and let $s\in[p]$ where $s=2t'+1$.
Also, let $\Delta$ be a closed annulus and $\mathcal{A}=(\mathcal{C},\mathcal{P})$ be a  $\Delta$-embedded $(p,q)$-railed annulus of a partially $\Delta$-embedded graph $G$.
Given some $I\subseteq [q]$, we say that a linkage  $L$ of $G$ is {\em $(s,I)$-confined in $\mathcal{A}$} if 
$$L\cap \ann(\mathcal{C},t+1-t',t+1+t')\subseteq \bigcup_{i\in I}P_{i}.$$

\paragraph{Linkage reducible graph classes.}
Let ${\cal G}$ be a graph class and $r\in\mathbb{N}$.
We say that ${\cal G}$ is {\em $r$-linkage reducible} if it is hereditary (i.e., if $G\in\mathcal{G}$ then for every $S\subseteq V(G)$, $G[S]\in\mathcal{G}$) and
if there is a function $f_{{\cal G},r}:\mathbb{N}\to\mathbb{N}$ such that for every $k\in\mathbb{N}$ and every $G\in {\cal G}$,
if $\tw(G)\geq f_{{\cal G},r}(k)$ and $G$ contains an $r$-scattered linkage $L$ of size at most $k$, then
there is a vertex $v\in V(G)$ such that $G\setminus v$ contains an $r$-scattered linkage $L'$ that is equivalent to $L$.
We say that $f_{{\cal G},r}$ {\em certifies} that $\mathcal{G}$ is $r$-linkage reducible.

We are now ready to state the main result of our paper.

\begin{theorem}\label{afsfsdfdsdsafsadffasdasfd2}
There exists a function  $\newfun{axfsfsd}:\mathbb{N}^2\to\mathbb{N}$,
 where
the images of  $\funref{axfsfsd}$ are even,
 such that 
for every odd $s\in \mathbb{N}_{\geq  1}$ and every $r,k\in\mathbb{N}$,
if 
\begin{itemize}
\item ${\cal G}$ is an $r$-linkage reducible graph class certified by a function $f_{{\cal G},r}$,
\item $\Delta$ is a closed annulus,
\item $G$ is a partially $\Delta$-embedded graph that belongs to ${\cal G}$,
\item  ${\cal A}=({\cal C},{\cal P})$ is a $\Delta$-embedded $(p,q)$-railed annulus of $G$, where  $p= \funref{axfsfsd}(m,r)+s$ and $q\geq  \frac{2r+5}{2}\cdot m$, where $m=f_{{\cal G},r}(k)$,
\item $L$ is a $\Delta$-avoiding $r$-scattered linkage of size at most $k$, and
\item $I\subseteq [q]$, where $|I|> m\cdot (r+1)$,
\end{itemize}
then
$G$ contains an $r$-scattered linkage
$\tilde{L}$ where $\tilde{L}\equiv L$, $\tilde{L}\setminus \Delta \subseteq L \setminus \Delta$, and  $\tilde{L}$ is $(s,I)$-confined in ${\cal A}$.
Moreover, $\funref{axfsfsd}(m,r) = {\cal O}(m^2+mr+r)$.
\end{theorem}

The proof of~\autoref{afsfsdfdsdsafsadffasdasfd2} is presented in~\autoref{sec_tamelinkage}.

\subsection{Implications of~\autoref{afsfsdfdsdsafsadffasdasfd2}}
\label{subsec_implic}

As~\autoref{afsfsdfdsdsafsadffasdasfd2} is stated for graphs that belong to $r$-linkage reducible graph classes.
Therefore, our principal goal is to inspect which graph classes are $r$-linkage reducible and what is the function  $f_{{\cal G},r}$ certifying this.

We say that a function is {\em even}
if its images are even numbers. We state the following result.

\begin{proposition}[\cite{KawarabayashiW10asho,RobertsonS09XXI}]
	\label{i94opq}
	There exists an even function  $\newfun{fun_ulinka}:\mathbb{N}\to\mathbb{N}$ such that for every $k\in \mathbb{N}$
	if $G$ is a graph and $L$ is a linkage of $G$ of size at most $k$ and $\tw(G)\geq \funref{fun_ulinka}(k)$,
	then there is a vertex $v\in V(G)$ such that $G\setminus v$ contains a linkage $L'$ that is equivalent to $L$.
\end{proposition}
%In the above proposition, $\funref{fun_ulinka}$ is a non-decreasing 
%function that is important for the statement 
%of many of the results of this paper. For this reason, for now on,
%$\funref{fun_ulinka}$ will always denote the function of~\autoref{i94opq}. 

In terms of linkage reducible graph classes, the above proposition can be interpreted as:
\begin{observation}
The class $\mathcal{G}_{\sf all}$ of all graphs is $0$-linkage reducible, certified by the function $\funref{fun_ulinka}$.
\end{observation}

Therefore, from~\autoref{afsfsdfdsdsafsadffasdasfd2} and~\autoref{i94opq}, we derive the following:

\begin{corollary}\label{main1}
There exist two functions  $\funref{fun_ulinka}, \newfun{fun_somelinkage}:\mathbb{N}\to\mathbb{N}$ such that 
for every odd $s\in \mathbb{N}_{\geq  1}$ and every $k\in\mathbb{N}$, if
\begin{itemize}
\item $\Delta$ is a closed annulus,
\item $G$ is a graph that is partially $\Delta$-embedded,
\item $\mathcal{A}=(\mathcal{C},\mathcal{P})$ is a  $\Delta$-embedded $(p,q)$-railed 
annulus of $G$, where  $p\geq \funref{fun_somelinkage}(k)+s$ and $q\geq  5/2\cdot\funref{fun_ulinka}(k)$,
\item $L$ is a $\Delta$-avoiding linkage of size at most $k$, and
\item $I\subseteq [q]$, where $|I|> \funref{fun_ulinka}(k)$,
\end{itemize}
then
$G$ contains a linkage
$\tilde{L}$ where $\tilde{L}\equiv L$, $\tilde{L}\setminus \Delta \subseteq L \setminus \Delta$, and  $\tilde{L}$ is $(s,I)$-confined in $\mathcal{A}$.
Moreover, $ \funref{fun_somelinkage}(k)= {\cal O}((\funref{fun_ulinka}(k))^2)$.
\end{corollary}

In~\cite{KawarabayashiK12alin}, Kawarabayashi and Kobayashi proved that planar graphs are $1$-linkage reducible.
The function that certifies that planar graphs are $1$-linkage reducible of planar graphs can be made single-exponential using the results of~\cite{AdlerKKLST17irre}.
Based on the techniques of~\cite{KawarabayashiK12alin} and the single-exponential bound of~\cite{Mazoit13asin}, we also prove that the class of all graphs embedded on a surface of genus $g$ is $r$-linkage reducible, certified by a single-exponential function.

\begin{theorem}\label{amultheah}
There is a function $\newfun{wtfun}:\mathbb{N}^{3}\to \mathbb{N}$ such that for every $r,k,g\in \mathbb{N}$ if $G$ is a graph embedded on a surface $\Sigma$ of genus $g$,
$L$ is an $r$-scattered linkage of $G$ of size at most $k$,
and $\tw(G)\geq \funref{wtfun}(r,k,g)$,
then there is a vertex $v\in V(G)$ such that
$G\setminus v$ contains a linkage $L'$ that is equivalent to $L$.
Moreover, it holds that $\funref{wtfun}(r,k,g)=r\cdot 2^{{\cal O}(k+g)}$.
\end{theorem}

The proof of~\autoref{amultheah} is presented in~\autoref{sec_irrfriendlybdgenus}. Using~\autoref{afsfsdfdsdsafsadffasdasfd2}
and~\autoref{amultheah}, we can derive the following.

\begin{corollary}\label{cor_bdg}
There exist two functions  $\newfun{fun_bdg}, \funref{wtfun}:\mathbb{N}^3\to\mathbb{N}$ such that 
for every odd $s\in \mathbb{N}_{\geq  1}$ and every $r,k,g\in\mathbb{N}$,
if 
\begin{itemize}
\item $\Sigma$ is a surface of Euler genus $g$,
\item $\Delta$ is a closed annulus of $\Sigma$,
\item $G$ is a graph embedded in $\Sigma$
\item  ${\cal A}=({\cal C},{\cal P})$ is a $\Delta$-embedded $(p,q)$-railed annulus of $G$, where
$p\geq \funref{fun_bdg}(r,g,k)+s$ and $q\geq  \frac{2r+5}{2}\cdot\funref{wtfun}(r,k,g)$
\item $L$ is a $\Delta$-avoiding $r$-scattered linkage of size at most $k$, and
\item $I\subseteq [q]$, where $|I|> \funref{wtfun}(r,k,g)\cdot (r+1)$,
\end{itemize}
then
$G$ contains an $r$-scattered linkage
$\tilde{L}$ where $\tilde{L}\equiv L$, $\tilde{L}\setminus \Delta \subseteq L \setminus \Delta$, and  $\tilde{L}$ is $(s,I)$-confined in ${\cal A}$.
Moreover, $ \funref{fun_bdg}(r,k,g)= {\cal O}((\funref{wtfun}(r,k,g))^2)$.
\end{corollary}

\section{How to comb a linkage}%\sed{Comb? I suggest we remove tame...}
\label{sec_tamelinkage}
This section is dedicated to the proof of~\autoref{afsfsdfdsdsafsadffasdasfd2}.
We develop a series of definitions that allow us to ``comb'' a given linkage through the tracks of a railed annulus.
First, in~\autoref{subsec_minimal}, we define a notion of minimal linkages.
Then, in \autoref{subsec_rivers}, we prove that minimal linkages have few rivers, that means that they traverse the annulus few times and in~\autoref{subsec_mountains} we prove that minimal linkages do not have high mountains or deep valleys, which means that there are not ``deep enough'' same-side laminar intrusions of the linkage in the annulus.
Also, in~\autoref{subsec_rerouting} we present how to route an $r$-scattered linkage through a railed annulus in a ``combed'' way, given that the terminals of the linkage are vertices in the intersection of the rails and the cycles of the railed annulus and are scattered enough. This constitutes the core of the rerouting arguments that we develop, which are all tied together in~\autoref{subsec_prooftame}, in order to prove~\autoref{afsfsdfdsdsafsadffasdasfd2}.

\subsection{Minimal linkages}\label{subsec_minimal}
In this subsection, we aim to define the notion of minimal linkages  with respect to a given linkage $L$ and a collection of cycles ${\cal C}$, that intuitively corresponds to linkages that
are equivalent to $L$ and are ``diverging'' from ${\cal C}$ in the minimum (in number of edges) possible way.
Then, we prove that, intuitively, given a graph $G$, a linkage $L$ of $G$, and a collection of cycles ${\cal C}$ of $G$, if $L$ is a minimal linkage with respect to $L$ and ${\cal C}$, then the union of $L'$ and ${\cal C}$ is a graph of bounded treewidth.
The bound on the treewidth is given by the function $f_{\mathcal{G},r}$ of the $r$-linkage reducible class ${\cal G}$ in which $G$ belongs.

\paragraph{LB-pairs.}
Given a graph $G$,  a {\em LB-pair} of $G$ is a pair $(L,B)$ where $B$ is a  subgraph of $G$ with maximum degree 2 and  $L$ is a  linkage of $G$. We define $\cae(L,B)=|E(L)\setminus E(B)|$ (i.e., the number of linkage edges that are not edges of $B$).%\marg{$\cae(L,B)$}

%\paragraph{$H$-vitality.}
%Let $G$ be a graph, $H$ be a subgraph of $G$ and $L$ be an $r$-scattered linkage $L$ of $G$.
%We say that $L$ is {\em $H$-vital} if there is no $(L,r)$-irrelevant vertex in $L\cup H$.
%%\begin{lemma}
%%Let $G$ be a graph, $H$ be a subgraph of $G$ and $L$ be an $r$-scattered linkage $L$ of $G$.
%%If $L$ is $H$-vital, then $\tw(L\cup H)\leq f(|L|)$.
%%\end{lemma}
%\medskip
%
%We can use the fact that if $\tw(G)>\funref{wtfun}(r,k,g)$, then for every $r$-scattered linkage $L$ we can find an equivalent $r$-scattered linkage $L'\subseteq L\cup\cupall {\cal C}$.
\begin{lemma}\label{alo9op}
Let $r\in\mathbb{N}$, let $\mathcal{G}$ be an $r$-linkage reducible graph class, certified by a function $f_{\mathcal{G},r}$, let $G\in \mathcal{G}$, and let $(L,B)$ be an LB-pair of $G$, 
where $L$ is $r$-scattered in $G$.
If $\tw(L\cup B)>f_{\mathcal{G},r}(|L|)$,
then $G$ contains a linkage $L'$ where 
	\begin{enumerate}
		\item  $\cae(L',B)<\cae(L,B)$,
		\item $L'\equiv L$,
		\item $L'\subseteq L\cup B$.
	\end{enumerate}
\end{lemma}

\begin{proof}
Let $H=L\cup B$.
Since  $G\in \mathcal{G}$ and  $\mathcal{G}$ is $r$-linkage reducible,
there is a vertex $v\in V(H)$ such that
$H\setminus \{v\}$ contains a linkage $L'$ that is equivalent to $L$.
Notice that $E(L')\setminus E(B)\subseteq E(L)\setminus E(B)$. It remains to 
prove that this inclusion is proper.
Let $\{x,y\}$ be a member of the common pattern of  ${L}$ and ${L}'$ such 
that the  $(x,y)$-path $P$ of $L$ is different than the $(x,y)$-path $P'$ of ${L}'$.
Clearly, $P$ and $P'$, when oriented from $x$ to $y$, 
have a common part $P^*$. Formally, this is the connected component 
of $P\cap P'$ that contains $x$. Let $e$ be the $(m+1)$-th edge of $P$, starting from $x$, where $m$ is the length of $P^*$.  Notice that $e\in E({L})\setminus E(B)$, while $e\not\in E({L}')\setminus E(B)$.
We conclude that $E(L')\setminus E(B)\subsetneq E(L)\setminus E(B)$, therefore 
$|E(L')\setminus E(B)|<|E(L)\setminus E(B)|$, as required. 
\end{proof}

\smallskip\noindent {\bf Minimal linkages.}
Let $r\in\mathbb{N}$,
let $G$ be a \seg,  $\mathcal{C}$ be a $\Delta$-parallel sequence of cycles of $G$, 
$D\subseteq \Delta$, $L$ be a  
$\Delta$-avoiding and $D$-free $r$-scattered linkage 
of $G$. We say that an $r$-scattered linkage $L'$ of $G$
is {\em $(\mathcal{C},D,L)$-minimal} if, among all the $\Delta$-avoiding $r$-scattered linkages  of $G$ that are equivalent to $L$
and are subgraphs of $L\cup(\cupall\mathcal{C}\setminus D)$,
$L'$ is one where the quantity $\cae(L',\cupall\mathcal{C}\setminus D)$ is minimized.

\begin{lemma}
\label{ap43k9s}
Let $r\in\mathbb{N}$ and let $\mathcal{G}$ be an $r$-linkage reducible graph class, certified by a function $f_{\mathcal{G},r}$.
Let $\Delta$ be a closed annulus,
let $G$ be a \seg that also belongs to $\mathcal{G}$,
let ${\cal C}$ be a $\Delta$-parallel sequence of cycles of $G$,
let $D\subseteq \Delta$,
and let $L$ be  a  $\Delta$-avoiding and $D$-free $r$-scattered linkage of $G$.
If $L'$ is a $({\cal C},D,L)$-minimal 
linkage of $G$, then
$\tw(L'\cup (\cupall {\cal C}\setminus D))\leq f_{\mathcal{G},r}(|L'|)$.
\end{lemma}

\begin{proof} Let $B=\cupall {\cal C}\setminus D$ and note that $(L',B)$ is an LB-pair of $G$.
If $\tw(L'\cup (\cupall {\cal C}\setminus D))> f_{\mathcal{G},r}(|L'|)$, then by~\autoref{alo9op},
$G$ contains a linkage $L''$ that is equivalent to $L'$ where 
$\cae(L'',B)<\cae(L',B)$
and $L''\subseteq  L'\cup B$. This contradicts 
the choice of  $L'$ as a $({\cal C},D,L)$-minimal linkage of $G$. 
\end{proof}

\subsection{Minimal linkages have few rivers}
\label{subsec_rivers}
In this subsection we deal with {\em streams} and {\em rivers} of annuli-avoiding linkages.
Intuitively, a stream is a minimal part of a given linkage $L$, that traverses a given annulus that $L$-avoids.
We show that, given a linkage $L$ and a collection of cycles ${\cal C}$,
the minimum between the size of ${\cal C}$ and the number of different streams of $L$
is a lower bound to the treewidth of the graph obtained by the union of $L$ and ${\cal C}$ (\autoref{u4l49rop0r}).
This implies an upper bound on the number of rivers of a minimal linkage (\autoref{fskfsl}).

\paragraph{Streams and rivers.}
Let $\Delta$ be a closed annulus.
Let $G$ be a \seg, $\mathcal{C}=[C_{1},\ldots,C_{p}]$ be $\Delta$-parallel sequence of cycles
of $G$, and $L$ be a $\Delta$-avoiding linkage of $G$. 
A {\em $(\mathcal{C},L)$-stream} of $G$ is a subpath of $L$ that is a 
subset $P$ of $\Delta$ and such that $V(P\cap C_{1})$  
consists of the one endpoint of $P$ and $V(P\cap C_{p})$ 
consists of the other. 
A {\em disjoint collection of $(\mathcal{C},L)$-streams} of $G$ is a 
collection $\mathcal{R}$ of $(\mathcal{C},L)$-streams
such that $\cupall\mathcal{R}$ is a linkage of $G$.
A {\em $(\mathcal{C},L)$-river} of $G$  is a $(\mathcal{C},L)$-stream that is a subpath of  
a connected component of $L\cap \Delta$ that has  one of its endpoints 
in $C_{1}$ and the other in $C_{p}$. Notice that not each $(\mathcal{C},L)$-stream of $G$ is 
a $(\mathcal{C},L)$-river and any collection of $(\mathcal{C},L)$-rivers is a disjoint collection of $(\mathcal{C},L)$-streams (see \autoref{asdfgsdfdsffdsgdfshfgfsaf}).

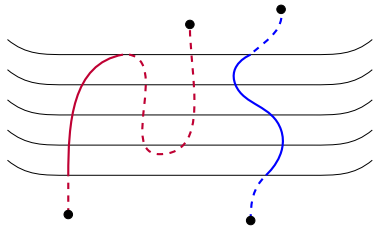
\begin{figure}[H]
	\centering\scalebox{0.8}{
	\sshow{0}{\begin{tikzpicture}[scale=0.5]
	
	%CYCLES
	\foreach \y in {0,...,4}{
		\draw[-] (-2,\y+ 0.5) to [out=-40, in= 180] (0,\y) to  (8,\y)  to [out=0, in= -140] (10,\y+0.5);
	}

	\begin{scope}%STREAM
	\draw[dashed, crimsonglory, line width=1pt] (0,-1.3) to (0,0);
	\draw[crimsonglory, line width=1pt] (0,0) to [out=90, in=190] (1.8,4);
	\draw[dashed,crimsonglory, line width=1pt] (2,4) to [out=-10, in=180] (3,0.7) to [out=0, in =270] (4,5);
	\node[black node]  (s1) at (0,-1.3) {};
	\node[black node]  (t1) at (4,5) {};
	\end{scope}

	\begin{scope}%RIVER
	\draw[blue, line width=1pt] plot [smooth, tension=1] coordinates {(6.5,0) (7,1.5) (5.5,3) (6,4)};
	\draw[dashed,blue, line width=1pt] (6,-1.5) to [out=90, in=225] (6.5,0);
	\draw[dashed, blue, line width=1pt] (6.1,4.1) to [out=40, in=280] (7,5.5);
	
	\node[black node]  (s2) at (6,-1.5) {};
	\node[black node]  (t2) at (7,5.5) {};
	\end{scope}
	
	\end{tikzpicture}}}
	\caption{An example of a %$\Delta$-nested cycle collection $\mathcal{C}$, a $\ann(\mathcal{C})$-avoiding linkage $L$, a
		$(\mathcal{C}, L)$-stream (depicted in solid red) and a $(\mathcal{C}, L)$-river (depicted in solid blue).}
		\label{asdfgsdfdsffdsgdfshfgfsaf}
\end{figure}

We now introduce a notion of {\sl ordering} of streams that will be used to consider certain collections of consecutive streams of a linkage.
 \paragraph{Orderings of streams.}
If $\mathcal{Z}$ is  a disjoint collection of $(\mathcal{C},L)$-streams of $G$ we  define 
its  {\em $D$-ordering}  as follows:
Consider the sequence $[Z_{1},\ldots,Z_{d}]$ 
such that  for each $i\in[d]$, 
one, say $D_{i}$, of the two connected components of $\Delta\setminus (Z_{i}\cup Z_{i+1})$  does not intersect $\cupall\mathcal{Z}$ (here $Z_{d+1}$ denotes $Z_{1}$). Among 
all $(d-1!)$ such sequences we insist that  $[Z_{1},\ldots,Z_{d}]$ is the unique one where $D\subseteq D_{q}$ and that the order 
of $\mathcal{Z}$ is the counter-clockwise order that its 
elements appear  around  $\Delta$ (see \autoref{Dordering}). 
We call  $[Z_{1},\ldots,Z_{d}]$ the {\em $D$-ordering} of $\mathcal{Z}$.

% and we refer to the closure of the open disk $\Delta_{q}$ as the {\em $\mathcal{Z}$-extension}
%of $D$ in $\mathcal{C}$.\sed{Correct this and add it in the definition!}

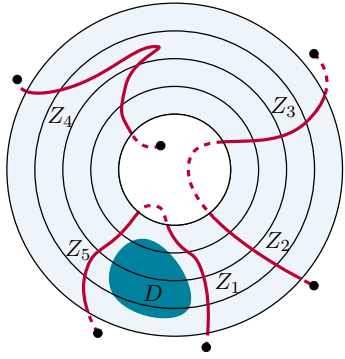
\begin{figure}[H]
	\centering
	\scalebox{0.8}{
	\sshow{0}{\begin{tikzpicture}[scale=.46]
	
	\begin{scope}
	\fill[celestialblue!10!white] (0,0) circle (6 cm);
	\fill[white] (0,0) circle (2 cm);
	\end{scope}

	\fill[darkturquoise] plot [smooth cycle, tension=1.2] coordinates { (-2.2,-4.5) (.5,-5) (-1,-2.5)};
	\node () at (260:4.5) {$D$};
	\foreach \x in {2,...,6}{
		\draw[line width =0.6pt] (0,0) circle (\x cm);
	}

	\begin{scope}[on background layer]
	\draw[crimsonglory, dashed, line width=1.5pt] plot [smooth, tension=1.2] coordinates { (120:1) (140:2) (120:3.5) (100:4.5) (145:5) (150:6.5)};
	\draw[crimsonglory, dashed, line width=1.5pt] plot [smooth, tension=1] coordinates { (245:6.5) (230:5) (230:3) (240:1.5) (270:2.5) (285:4) (280:6.5)};
	\draw[crimsonglory, dashed, line width=1.5pt] plot [smooth, tension=1.2] coordinates { (40:6.5) (20:5) (-10:0.5) (-40:6.5)};
	\node[black node]  () at (120:1) {};
	\node[black node]  () at (150:6.5) {};
	\node[black node]  () at (245:6.5) {};
	\node[black node]  () at (280:6.5) {};
	\node[black node]  () at (40:6.5) {};
	\node[black node]  () at (-40:6.5) {};
	
	\end{scope}
	
	\begin{scope}
	\clip  (0,0) circle (6 cm);
	
		\draw[crimsonglory, line width=1.5pt] plot [smooth, tension=1.2] coordinates { (120:1) (140:2) (120:3.5) (100:4.5) (145:5) (150:6.5)};
	\draw[crimsonglory, line width=1.5pt] plot [smooth, tension=1] coordinates { (245:6.5) (230:5) (230:3) (240:1.5) (270:2.5) (285:4) (280:6.5)};
	\draw[crimsonglory, line width=1.5pt] plot [smooth, tension=1.2] coordinates { (40:6.5) (20:5) (-10:0.5) (-40:6.5)};
	\node[black node]  () at (120:1) {};
	\node[black node]  () at (150:6.5) {};
	\node[black node]  () at (245:6.5) {};
	\node[black node]  () at (280:6.5) {};
	\node[black node]  () at (40:6.5) {};
	\node[black node]  () at (-40:6.5) {};
	\fill[white] (0,0) circle (2 cm);
	\draw[line width =0.6pt] (0,0) circle (2 cm);
	\end{scope}

	\begin{scope}
	\clip (0,0) circle (2 cm);
	
		\draw[crimsonglory, dashed, line width=1.5pt] plot [smooth, tension=1.2] coordinates { (120:1) (140:2) (120:3.5) (100:4.5) (145:5) (150:6.5)};
	\draw[crimsonglory,dashed, line width=1.5pt] plot [smooth, tension=1] coordinates { (245:6.5) (230:5) (230:3) (240:1.5) (270:2.5) (285:4) (280:6.5)};
	\draw[crimsonglory,dashed, line width=1.5pt] plot [smooth, tension=1.2] coordinates { (40:6.5) (20:5) (-10:0.5) (-40:6.5)};
	\node[black node]  () at (120:1) {};
	\node[black node]  () at (150:6.5) {};
	\node[black node]  () at (245:6.5) {};
	\node[black node]  () at (280:6.5) {};
	\node[black node]  () at (40:6.5) {};
	\node[black node]  () at (-40:6.5) {};
	\end{scope}
	
	\node (Z1) at (295:4.5) {$Z_{1}$};
	\node (Z2) at (325:4.5) {$Z_{2}$};
	\node (Z3) at (30:4.5) {$Z_{3}$};
	\node (Z4) at (155:4.5) {$Z_{4}$};
	\node (Z5) at (220:4.5) {$Z_{5}$};
	
	\end{tikzpicture}}}
	\caption{An example of a $\Delta$-parallel sequence $\mathcal{C}$ of cycles, an open disk $D\subseteq\Delta$ (depicted in blue), a linkage $L$ (depicted in red) that is $D$-free and $\ann(\mathcal{C})$-avoiding, a disjoint collection $\mathcal{Z}$ of $(\mathcal{C}, L)$-streams, and the $D$-ordering $[Z_{1},\ldots, Z_{5}]$ of $\mathcal{Z}$.}
	\label{Dordering}
\end{figure}

In order to prove our next result (\autoref{u4l49rop0r}), we will use the equivalent definition of treewidth in terms of brambles.

\paragraph{Brambles.}
Given a graph $G$, we say that 
a subset $S$ of $V(G)$ is {\em connected} if $G[S]$ is connected. Given $S_{1},S_2\subseteq V(G)$,
we say that $S_1$ and $S_{2}$ touch if either $S_{1}\cap S_{2}\neq\emptyset$
or there is an edge $e\in E(G)$ where $e\cap S_{1}\neq\emptyset$ and $e\cap S_{2}\neq\emptyset$.
A {\em bramble} in $G$ is a collection $\mathcal{B}$ is pairwise touching connected 
subsets of $V(G)$. The {\em order} of a bramble $\mathcal{B}$ is the minimum
number of vertices that intersect all of its elements.

\begin{proposition}[~\cite{SeymourT93grap}]
	\label{brlps3}
	Let $k\in\mathbb{N}$. 
	A graph $G$ has a bramble of order $k+1$ if and only if $\tw(G)\geq k $.
\end{proposition}

We are now ready to prove the following result.

\begin{lemma}
\label{u4l49rop0r}
Let $\Delta$ be a closed annulus.
Let $G$ be a \seg, $\mathcal{C}$ be a $\Delta$-parallel sequence of cycles of $G$,
$D$ be an open disk where $D\subseteq \Delta$, 
$L$ be a $\Delta$-avoiding and $D$-free linkage of $G$,
and $\mathcal{Z}$  be a disjoint collection of $(\mathcal{C},L)$-streams of $G$.
	Then $\tw(L\cup(\cupall \mathcal{C}\setminus D))\geq \min\{|\mathcal{C}|,|\mathcal{Z}|\}$.
\end{lemma}

\begin{proof} Let $[Z_{1},\ldots,Z_{d}]$ be the $D$-ordering of $\mathcal{Z}$ and let $D'$ be the connected component of $\ann(\mathcal{C})\setminus (Z_{d}\cup Z_{1})$  that contains $D$.
	Let $r= \min\{|\mathcal{C}|,|\mathcal{Z}|\}$, and let $[Z_1,\ldots,Z_{r}]$ 
	be the sequence consisting of the first $r$ elements of the $D$-ordering of $\mathcal{Z}$.
	Let also 
	$\mathcal{C}'$ be the  
	sequence consisting of the first $r$ elements of $\mathcal{C}$. Notice that there is a disjoint collection
	$\mathcal{Z}'=[Z_1',\ldots,Z_{r}']$ of $(\mathcal{C}',L)$-streams of $G$
	such that for each $i\in[r]$, $Z_{i}'\subseteq Z_{i}$.

	We now set $\mathcal{B}=\mathcal{C}'\setminus D'$, denote $\mathcal{B}=[B_{1},\ldots,B_{r}]$, and notice that both $\mathcal{B}$ and $\mathcal{Z}'$ are sequences 
	of paths in $G$, such that both $\cupall\mathcal{B}$ 
	and $\cupall\mathcal{Z}'$ are linkages of $G$.
	Consider now the graph  $Q=\cupall\mathcal{B}\cup\cupall\mathcal{Z}'$ and 
	notice that  $C=B_{1}\cup Z_{1}'\cup B_{r}\cup Z_{r}'$ is a cycle of $G$.

	As $Q\subseteq L\cup(\cupall \mathcal{C}\setminus D)$, it remains to prove that $\tw(Q)\geq  r$. 
	For this, because of \autoref{brlps3}, it suffices to 
	give a bramble of $Q$ of order $r+1$.
	For each $(i,j)\in[2,r-1]^2$ we define $X^{(i,j)}=(B_{i}\cup Z_{j}')\setminus V(C)$.
	It is easy to check that $\mathcal{X}=\{X^{(i,j)}\mid (i,j)\in[2, r-1]^2\}$  is a bramble of $Q$ of order $\geq r-2$.
	Let also $X^{(1)}=Z_{1}\setminus B_{1}$, $X^{(2)}=B_{1}$, and $X^{(3)}=Z_{r}'\cup B_{r}$.  Notice that $\mathcal{X}\cup\{X^{(1)},X^{(2)},X^{(3)}\}$ is also a bramble 
	of $Q$ and its order is the order of $\mathcal{X}$ incremented by  3. Therefore $Q$ contains a bramble of order at least $r+1$, as required (see \autoref{sadgfsfgdhfhhj}).

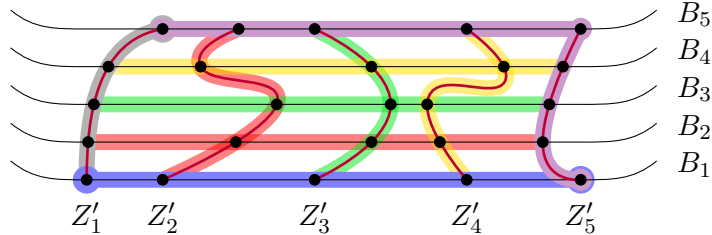
\begin{figure}[H]
\centering
\scalebox{1}{
	\begin{tikzpicture}
	%CYCLES
	\foreach \y in {0,...,4}{
	\draw[-] (-1,\y/2+ 0.25) to [out=-40, in= 180] (0,\y/2) to  (6.5,\y/2)  to [out=0, in= -140] (7.5,\y/2+0.25);}

		\draw[crimsonglory, line width=1pt,name path= stream1] (0,0) to [out=90, in=190] (1,2);
		\node () at (0,-.5) {$Z_{1}'$};
			
		\draw[crimsonglory, line width=1pt,  name path=stream2] plot [smooth, tension=1] coordinates {(1,0) (2.5,1) (1.5,1.5) (2,2)};
		\node () at (1,-.5) {$Z_{2}'$};
					
		\draw[crimsonglory, line width=1pt,  name path=stream3] plot [smooth, tension=1] coordinates {(3,0) (4,1) (3,2)};
		\node () at (3,-.5) {$Z_{3}'$};
					
		\draw[crimsonglory, line width=1pt,  name path=stream4] plot [smooth, tension=1] coordinates {(5,0) (4.5,1.1) (5.5,1.3) (5,2)};
		\node () at (5,-.5) {$Z_{4}'$};
		
		\draw[crimsonglory, line width=1pt,  name path=stream5] plot [smooth, tension=1] coordinates {(6.5,0) (6,0.5) (6.5,2)};
		\node () at (6.5,-.5) {$Z_{5}'$};
		
			\path[name path=cycle0]  (-1,0) to  (7.5,0);
			\node () at (8,0.2) {$B_{1}$};
			\path[name path=cycle1]  (-1,0.5) to  (7.5,0.5);
			\node () at (8,0.7) {$B_{2}$};
			\path[name path=cycle2]  (0,1) to  (7.5,1);
			\node () at (8,1.2) {$B_{3}$};
			\path[name path=cycle3]  (0,1.5) to  (7.5,1.5);
			\node () at (8,1.7) {$B_{4}$};
			\path[name path=cycle4]  (0,2) to  (7.5,2);
			\node () at (8,2.2) {$B_{5}$};
			
			{
				\node[black node, name intersections = {of = cycle0 and stream1}] (A1) at
				(intersection-1) {};
				\node[black node, name intersections = {of = cycle1 and stream1}] (A2) at
				(intersection-1) {};
				\node[black node, name intersections = {of = cycle2 and stream1}] (A3) at
				(intersection-1) {};
				\node[black node, name intersections = {of = cycle3 and stream1}] (A4) at
				(intersection-1) {};
				\node[black node, name intersections = {of = cycle4 and stream1}] (A5) at
				(intersection-1) {};
			}
			
			{
				\node[black node, name intersections = {of = cycle0 and stream2}] (B1) at
				(intersection-1) {};
				\node[black node, name intersections = {of = cycle1 and stream2}] (B2) at
				(intersection-1) {};
				\node[black node, name intersections = {of = cycle2 and stream2}] (B3) at
				(intersection-1) {};
				\node[black node, name intersections = {of = cycle3 and stream2}] (B4) at
				(intersection-1) {};
				\node[black node, name intersections = {of = cycle4 and stream2}] (B5) at
				(intersection-1) {};
			}
			
			{
				\node[black node, name intersections = {of = cycle0 and stream3}] (C1) at
				(intersection-1) {};
				\node[black node, name intersections = {of = cycle1 and stream3}] (C2) at
				(intersection-1) {};
				\node[black node, name intersections = {of = cycle2 and stream3}] (C3) at
				(intersection-1) {};
				\node[black node, name intersections = {of = cycle3 and stream3}] (C4) at
				(intersection-1) {};
				\node[black node, name intersections = {of = cycle4 and stream3}] (C5) at
				(intersection-1) {};
			}
			
			{
				\node[black node, name intersections = {of = cycle0 and stream4}] (D1) at
				(intersection-1) {};
				\node[black node, name intersections = {of = cycle1 and stream4}] (D2) at
				(intersection-1) {};
				\node[black node, name intersections = {of = cycle2 and stream4}] (D3) at
				(intersection-1) {};
				\node[black node, name intersections = {of = cycle3 and stream4}] (D4) at
				(intersection-1) {};
				\node[black node, name intersections = {of = cycle4 and stream4}] (D5) at
				(intersection-1) {};
			}
			{
				\node[black node, name intersections = {of = cycle0 and stream5}] (E1) at
				(intersection-1) {};
				\node[black node, name intersections = {of = cycle1 and stream5}] (E2) at
				(intersection-1) {};
				\node[black node, name intersections = {of = cycle2 and stream5}] (E3) at
				(intersection-1) {};
				\node[black node, name intersections = {of = cycle3 and stream5}] (E4) at
				(intersection-1) {};
				\node[black node, name intersections = {of = cycle4 and stream5}] (E5) at
				(intersection-1) {};
			}
			
			\begin{scope}[on background layer]

			\begin{scope}%STREAM2
			\draw[red, opacity=0.5, line width=6pt] plot [smooth, tension=1] coordinates {(1,0) (2.5,1) (1.5,1.5) (2,2)};
			\end{scope}
			\draw[red, opacity=0.5, line width=6pt] (A2) --  (E2);
			
			\begin{scope}%STREAM3
			\draw[green!90!blue, opacity=0.5, line width=6pt]  plot [smooth, tension=1] coordinates {(3,0) (4,1) (3,2)};
			\end{scope}
			\draw[green!90!blue, opacity=0.5, line width=6pt]  (A3) --  (E3);
			
			\begin{scope}%STREAM4
			\draw[yellow!90!red, opacity=0.5, line width=6pt] plot [smooth, tension=1] coordinates {(5,0) (4.5,1.1) (5.5,1.3) (5,2)};
			\end{scope}
			\draw[yellow!90!red, opacity=0.5, line width=6pt]  (A4) --  (E4);

			\begin{scope}%STREAM1
			\draw[black!30!white, line width=6pt] (0,.1) to [out=90, in=190] (1,2);
			\node[draw=black!30!white, fill=black!30!white, circle, minimum size=10pt, inner sep=0pt] ()  at (1,2) {};
			\end{scope}

			\node[draw=blue!50!white, fill=blue!50!white, circle, minimum size=10pt, inner sep=0pt] ()  at (0,0) {};
			\node[draw=blue!50!white, fill=blue!50!white, circle, minimum size=10pt, inner sep=0pt] ()  at (6.5,0) {};
			\draw[blue!50!white, line width=6pt] (A1.center) to (E1.center);
			
			\draw[violet!40!white, line width=6pt]  plot [smooth, tension=1] coordinates {(6.5,0) (6,0.5) (6.5,2)};
			\draw[violet!40!white, line width=6pt]  (1,2) --  (E5.center);
			\node[draw=violet!40!white, fill=violet!40!white, circle, minimum size=8pt, inner sep=0pt] ()  at (6.5,2) {};
			\node[draw=violet!40!white, fill=violet!40!white, circle, minimum size=8pt, inner sep=0pt] ()  at (6.5,0) {};
			\node[draw=violet!40!white, fill=violet!40!white, circle, minimum size=8pt, inner sep=0pt] ()  at (1,2) {};
			
	\end{scope}
	\end{tikzpicture}
}
\caption{An example of the construction of a bramble of $Q$, where $|\mathcal{B}|=5$ and $|\mathcal{Z}'|=5$. Here, $X^{(2,2)}, X^{(3,3)}, X^{(4,4)}$ are depicted in red, green, and yellow, respectively, while $X^{(1)}, X^{(2)}, X^{(3)}$ are depicted in grey, blue, and violet, respectively.}	\label{sadgfsfgdhfhhj}
\end{figure}
\end{proof}

Using~\autoref{u4l49rop0r}, we can prove that minimal linkages have few rivers.
\begin{lemma}
\label{fskfsl}
Let $\Delta$ be a closed annulus and let $r\in\mathbb{N}_{\geq 0}$.
Let  $G$ be a \seg,  ${\cal C}$ be a $\Delta$-parallel sequence of cycles 
of $G$,  $D\subseteq \Delta$, and $L$ be a $\Delta$-avoiding  and $D$-free $r$-scattered linkage of $G$.
If $L'$ is a $({\cal C},D,L)$-minimal linkage and $\tw(L'\cup (\cupall {\cal C}\setminus D))< |{\cal C}|$, then 
 $L'$ has at most $\tw(L'\cup (\cupall {\cal C}\setminus D))$ ${\cal A}$-rivers.
\end{lemma}

\begin{proof}
Let $m=\tw(L'\cup (\cupall {\cal C}\setminus D)).$
Assume
that  $G$ contains 
a collection ${\cal Z}$ of  $({\cal C},L')$-rivers
where $|{\cal Z}|>m$.
Recall that ${\cal Z}$ is a disjoint collection of $({\cal C},L')$-streams of $G$.
From \autoref{u4l49rop0r}, $\tw(L'\cup(\cupall {\cal C}\setminus D))\geq \min\{|{\cal C}|,|{\cal Z}|\}>m$, a contradiction.
\end{proof}

\subsection{Minimal linkages do not have high mountains or deep valleys}\label{subsec_mountains}
In this subsection, we introduce another type of structure in linkages, that are {\sl mountains} and {\sl valleys}.
Intuitively, given a linkage $L$ and a collection ${\cal C}$ of cycles,
a mountain (resp. valley) of $L$ is a subpath of a path of $L$ that crosses twice a cycle $C_i$ in ${\cal C}$ while possibly crossing only cycles of larger (resp. smaller) indices.
We then define {\sl tight} mountains and valleys, that are mountains (resp. valleys) that cannot be ``pushed away'' towards their bases, due to the existence of a sequence of laminar  mountains (resp. valleys) below (resp. above) them.
We prove that all mountains and valleys of minimal linkages are tight (\autoref{ao7ui4jkhq0f})
and that tight mountains (resp. valleys) have small height (resp. depth) (\autoref{ahks5llozn}).
This implies that minimal linkages have mountains and valleys of small height and depth, respectively (\autoref{ato954jgd}).
The latter implies that if all terminals of a linkage are outside of a disk that contains many nested cycles,
then a minimal linkage would be disjoint of an ``inner area'' of this disk (\autoref{aop4icl}).

\paragraph{Mountains and valleys.}
Let $\Delta$ be a closed annulus.
Let $G$ be a \seg,  $\mathcal{C}$ be a $\Delta$-parallel sequence of cycles of $G$ of size $p\in\mathbb{N}$,
$D$ be an open disk where $D\subseteq \Delta$,
$L$ be a $\Delta$-avoiding and $D$-free linkage of $G$. 
Let  $i\in[p]$.
An {\em $(\mathcal{C},D,L)$-mountain (resp.~{\em $(\mathcal{C},D,L)$-valley}) of $G$ based on $C_{i}$}
is a non-trivial subpath $P$ of some path of $L$ where 
\begin{enumerate}
	\item $P\subseteq \overline{D}_{i}$ (resp. $P\subseteq \Delta\setminus D_{i}$),
	\item $P\cap D_{p}=\emptyset$  (resp. $P\cap(\Delta\setminus \overline{D}_{1})=\emptyset$),
	\item $P\cap C_{i}$ has two connected components, each containing exactly one of the endpoints of $P$,
	\item if $D'$ is the closure  of the connected component of $D_{i}\setminus P$ (resp. $(\Delta\setminus\overline{D}_{i})\setminus P$) that does not contain $D_{r}$ (resp. $\Delta\setminus\overline{D}_{1}$), then 
	$D'\cap T(L)=\emptyset$ and $D' \cap D=\emptyset$. 
\end{enumerate}
Clearly, in (4), $D'$ is a closed disk. We call it,  the {\em disk} of the $(\mathcal{C},D,L)$-mountain (resp. valley) $P$ and we denote it by ${\disk}(P)$.
Notice that there is no $(\mathcal{C},D,L)$-mountain based on $C_{p}$ and there 
is no $(\mathcal{C},D,L)$-valley based on $C_{1}$.

A {\em $(\mathcal{C},D,L)$-mountain} (resp.~{\em $(\mathcal{C},D,L)$-valley}) of $G$ is any 
$(\mathcal{C},D,L)$-mountain (resp. {$(\mathcal{C},D,L)$-valley}) of $G$ based on some of the cycles of $\mathcal{C}$.

The {\em height} (resp. {\em depth}) of  a {\em $(\mathcal{C},D,L)$-mountain} (resp.~{\em $(\mathcal{C},D,L)$-valley}) $P$ that is based on $C_{i}$ is the maximum $j$ such that $C_{i+j-1}$ (resp. $C_{i-j+1}$) intersects $P$ and, in both cases, we denote it by $\dehe(P)$. %\marg{$\dehe(P)$} 
Moreover,  
the  {height} (resp. {depth}) of $P$ is at least $1$ and at most $p$.

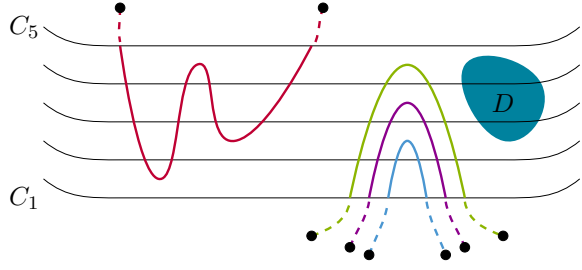
\begin{figure}%[H]
	\centering\scalebox{0.9}{
	\sshow{0}{\begin{tikzpicture}[scale=0.56]
	
	%CYCLES
	\foreach \y in {0,...,4}{
		\draw[-] (-2,\y+ 0.5) to [out=-40, in= 180] (0,\y) to  (10,\y)  to [out=0, in= -140] (12,\y+0.5);
	}
	\node (C1) at (-2.5, 0) {$C_1$};
	
	\node (C5) at (-2.5, 4.5) {$C_5$};
	
	\begin{scope}[on background layer]
	\fill[darkturquoise] plot [smooth cycle, tension=1.2] coordinates { (9,3.5) (11,3) (10,1.5)};
	\node () at (10,2.5) {$D$};
	\end{scope}

	\begin{scope}%Valley
	\draw[crimsonglory, line width=1pt] plot [smooth, tension=1] coordinates {(0,4) (1,.5) (2,3.5) (3,1.5) (5,4)};
	\draw[dashed, crimsonglory, line width=1pt] (0,5) to [out=270, in=100] (0,4.1);
	\draw[dashed, crimsonglory, line width=1pt] (5.1,4.1) to [out=70, in=250] (5.3,5);
	
	\node[black node] (s1) at (0,5) {};
	\node[black node] (t1) at (5.3,5) {};
	\end{scope}

	\begin{scope}[xshift=6cm]%MOUNTAIN
	\draw[applegreen, line width=1pt] plot [smooth, tension=1] coordinates {(0,0) (1.5,3.5) (3,0)};
	\draw[dashed, applegreen, line width=1pt] (-1,-1) to [out=20, in=260] (0,-.1);
	\draw[dashed, applegreen, line width=1pt] (3,-.1) to [out=280, in=160] (4,-1);
	
	\node[black node] (s3) at (-1,-1) {};
	\node[black node] (t3) at (4,-1) {};

	\draw[darkmagenta, line width=1pt] plot [smooth, tension=1] coordinates {(0.5,0) (1.5,2.5) (2.5,0)};
	\draw[dashed, darkmagenta, line width=1pt] (0,-1.3) to [out=70, in=260] (0.5,-.1);
	\draw[dashed, darkmagenta, line width=1pt] (2.5,-.1) to [out=280, in=110] (3,-1.3);
	
	\node[black node] (s4) at (0,-1.3) {};
	\node[black node] (t4) at (3,-1.3) {};

	\draw[celestialblue, line width=1pt] plot [smooth, tension=1] coordinates {(1,0) (1.5,1.5) (2,0)};
	\draw[dashed, celestialblue, line width=1pt] (0.5,-1.5) to [out=70, in=260] (1,-.1);
	\draw[dashed, celestialblue, line width=1pt] (2,-.1) to [out=280, in=110] (2.5,-1.5);
	
	\node[black node] (s4) at (0.5,-1.5) {};
	\node[black node] (t4) at (2.5,-1.5) {};
	\end{scope}
	\end{tikzpicture}}}
	\caption{An example of a $({\cal C}, D, L)$-valley (depicted in solid red), and some $({\cal C}, D, L)$-mountains (depicted in solid colors). Notice that the $({\cal C}, D, L)$-mountain depicted in green is tight.}
	\label{asdfsdgfsfgsfdgdghdgsfgfjhjsdg}
\end{figure}

Notice that if a $(\mathcal{C},L)$-stream $P$ of $G$ is a subpath of a $(\mathcal{C},D,L)$-mountain $P'$ or 
a $(\mathcal{C},D,L)$-valley $P'$ of $G$ then $\dehe(P')=p$. Moreover, if a $(\mathcal{C},L)$-stream $P$ of $G$ is not a subpath of some  $(\mathcal{C},D,L)$-mountain or some  $(\mathcal{C},D,L)$-valley of $G$, then $P$ is a
$(\mathcal{C},L)$-river of $G$.

\paragraph{Tight mountains and valleys.}
Let $r\in\mathbb{N}_{\geq 0}$.
Let $\Delta$ be a closed annulus.
Let $G$ be a \seg, ${\cal C}=[C_{1},\ldots,C_{p}]$ be $\Delta$-parallel sequence of cycles
of $G$, and $L$ be a $\Delta$-avoiding $r$-scattered linkage of $G$. 
Also, let $D\subseteq \Delta$.
We say that a $({\cal C},D,L)$-mountain (resp. $({\cal C},D,L)$-valley) $P$ based on $C_{i}$, is {\em tight} if  there is a $d\in\mathbb{N}_{\geq 0}$ such that $\dehe(P)=d \cdot (r+1)+2$  and there is a sequence $[P_0,\ldots,P_{d}]$
of  $({\cal C},D,L)$-mountains (resp. $({\cal C},D,L)$-valleys) based on $C_{i}$
such that 
\begin{itemize}
	\item $P=P_{d}$,
	\item $\forall j\in[0,d],\ \dehe(P_{j})=j\cdot (r+1)+2$, and 
	\item $\forall j\in[0,d-1], P_{j}\subseteq {\disk}(P_{j+1})$.
\end{itemize}

%Let $P$ be a $({\cal C},D,L)$-mountain of $G$ based on $C_{i}$ such that $\dehe(P) = d$. For every $j\in[d]$, we define $P^{(j)} = (P\setminus D_{i+j-1}) \cup (C_{i+j-1}\cap {\disk}(P))$.\marg{$P^{(j)}$}
%Observe that $\dehe(P^{(j)})=j$.

\begin{lemma}
\label{ao7ui4jkhq0f}
Let $\Delta$ be a closed annulus and let $r\in\mathbb{N}_{\geq 0}$.
Let $G$ be a \seg,  ${\cal C}$ be a $\Delta$-parallel sequence of cycles of $G$, $D\subseteq \Delta$, $L$ be a  $\Delta$-avoiding and $D$-free $r$-scattered linkage of $G$.  Let also $L'$ be a  $({\cal C},D,L)$-minimal $r$-scattered linkage of $G$.
Then all $({\cal C},D,L')$-mountains (resp.  $({\cal C},D,L')$-valleys) of $G$ are tight.
\end{lemma}

\begin{proof}
Let $B= \cupall {\cal C}\setminus D$. We present the proof for the case of $({\cal C},D,L')$-mountains
as the case of $({\cal C},D,L')$-valleys is symmetric.
% It suffices to prove that for every $j\in[2,d-1]$ there is a $({\cal C},D,L)$-mountain $P_{j}$ based on $C_{i}$ such that $P_{j}\subseteq {\disk}(P_{j+1})$ and $\dehe(P_{j})=j$.
\medskip

\noindent{\em Claim:}
Let $i\in\mathbb{N}_{\geq 1}, j\in\mathbb{N}_{\geq 1}$. If $P_{j}$ is a $({\cal C},D,L')$-mountain of $G$ based on $C_{i}$ such that $\dehe(P_{j})=j\cdot (r+1) +2$, then there exists a $({\cal C},D,L')$-mountain $P'$ based on $C_{i}$ such that $\dehe(P') = (j-1)\cdot (r+1) + 2$ and $P'\subseteq {\disk}(P_{j})$.
\medskip

\noindent{\em Proof of Claim:} Suppose to the contrary that there does not exist  a $({\cal C},D,L')$-mountain $P'$ based on $C_{i}$ such that $\dehe(P') = (j-1)\cdot (r+1)+2$ and $P'\subseteq {\disk}(P_{j})$.
Let $P_{j}^{\star}= (P_{j}\setminus D_{i+j\cdot (r+1)}) \cup (C_{i+j\cdot (r+1)}\cap {\disk}(P_{j}))$ and notice that $\dehe(P_{j}^{\star})=j\cdot (r+1) + 1$ (see \autoref{asadfdgfdhfsfdasdggsfdngfsdnhfssb}).

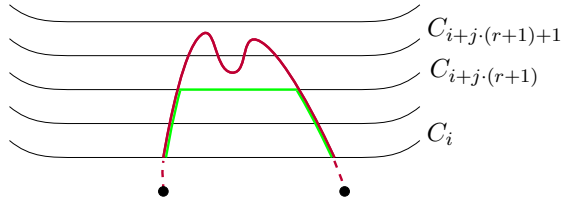
\begin{figure}[H]
	\centering\scalebox{0.9}{
	\sshow{0}{\begin{tikzpicture}[scale=0.5]
	
	%CYCLES
	\foreach \y in {0,...,4}{
		\draw[-] (-2,\y+ 0.5) to [out=-40, in= 180] (0,\y) to  (8,\y)  to [out=0, in= -140] (10,\y+0.5);
	}
	\node () at (10.6,0.7) {$C_{i}$};
	\node () at (11.9, 2.5) {$C_{i+j\cdot (r+1)}$};
	\node () at (12.2, 3.7) {$C_{i+j\cdot (r+1)+1}$};

%		\begin{scope}[on background layer]
%		\fill[darkturquoise] plot [smooth cycle, tension=1.2] coordinates { (9,3.5) (11,3) (10,1.5)};
%		\node () at (10,2.5) {$D$};
%		\end{scope}
%		
		
		\begin{scope}[xshift = 2.5cm]
		\draw[crimsonglory, line width=1pt, name path=P1] plot [smooth, tension=1] coordinates {(0,0) (1,3.5) (2,2.5) (3,3.3) (5,0)};
		\draw[dashed, crimsonglory, line width=1pt] (0,-1) to [out=90, in=260] (0,-0.1);
		\draw[dashed, crimsonglory, line width=1pt] (5,-.1) to [out=290, in=100] (5.3,-1);
		
		\path[name path=P2] (0,2) to (5,2);
		
		\path [name intersections={of=P1 and P2, by={A,B}}];
		
		\begin{scope}[scale=0.97]%Valley
		\clip (0,0) rectangle + (6,2);
		\begin{scope}[xshift=-0.02cm]
		\draw[green, line width=1pt]  plot [smooth, tension=1] coordinates {(0.1,0) (1.1,3.5) (2.1,2.5) (3.1,3.3) (5.1,0)};
		\end{scope}
		\end{scope}
				\draw[green, line width =1pt] (A.east) to (B.west);
		\draw[crimsonglory, line width=1pt] plot [smooth, tension=1] coordinates {(0,0) (1,3.5) (2,2.5) (3,3.3) (5,0)};
		
		\node[black node] (s1) at (0,-1) {};
		\node[black node] (t1) at (5.3,-1) {};
		\end{scope}
		
		\end{tikzpicture}}}
		\caption{An example of a $({\cal C},D,L')$-mountain  $P_{j}$ of $G$ based on $C_{i}$ such that $\dehe(P_{j})=j\cdot (r+1) +2$ (depicted in red) and the $({\cal C},D,L')$-mountain $P_{j}^{\star}$ (depicted in green).}
\label{asadfdgfdhfsfdasdggsfdngfsdnhfssb}
	\end{figure}

Observe that the linkage $L''= (L'\setminus P_{j})\cup P_{j}^{\star}$ is equivalent to $L$.
To see why $L''$ is also $r$-scattered, first keep in mind that $N_G^{(\leq r)} (V(C_{i+j\cdot (r+1)})) \cap V(C_{i+(j-1)\cdot (r+1)+1}))=\emptyset$.
Also, recall that, by assumption, there is no $({\cal C},D,L')$-mountain $P'$ based on $C_{i}$ such that $\dehe(P') = (j-1)\cdot (r+1)+2$ and $P'\subseteq {\disk}(P_{j})$, or, equivalently,
there is no $({\cal C},D,L')$-mountain $P'$ based on $C_{i}$
that intersects the cycle $C_{i+(j-1)\cdot (r+1)+1}$ and $P'\subseteq {\disk}(P_{j})$.
Therefore, taking into account that $P_{j}^{\star}\subseteq {\disk}(P_{j})$,
we have that, for every path $\tilde{P}$ in $L'\setminus P_{j}$,  
$N_G^{(\leq r)} (V(P_{j}^{\star}))\cap V(\tilde{P})=\emptyset$.
This implies that $L''$ is an $r$-scattered linkage.
Moreover, notice that $\cae(L'',B)<\cae(L',B)$
and $L''\subseteq  L'\cup B$.
This contradicts the choice of $L'$ as a $({\cal C},D,L)$-minimal $r$-scattered linkage of $G$. The claim follows.
\hfill$\diamond$\medskip
	
Let $P$ be a $({\cal C},D,L')$-mountain of $G$ based on $C_{i}$
such that $\dehe(P)=d\cdot (r+1) +2$, for some $d\in\mathbb{N}_{\geq 0}$.
The fact that $P$ is tight follows by recursively applying the Claim above.
\end{proof}

We now prove that the height (resp. depth) of a tight mountain (resp. valley) is ``small''.
\begin{lemma}
\label{ahks5llozn}
Let $\Delta$ be a closed annulus and let $r\in \mathbb{N}_{\geq 0}$.
Let $G$ be a \seg,  ${\cal C}$ be a $\Delta$-parallel sequence of cycles of $G$,
$D$ be a connected subset of $\Delta$, and  ${L}$ be a  $\Delta$-avoiding
and $D$-free $r$-scattered linkage of $G$. 
If $P$ is a tight $({\cal C},D,L)$-mountain (resp.  $({\cal C},D,L)$-valley) of $G$, then $\dehe(P)\leq \frac{3}{2}\cdot \tw(L\cup (\cupall {\cal C}\setminus D))$.
\end{lemma}

\begin{proof}
Let $d\in\mathbb{N}_{\geq 0}$ such that $\dehe(P) = d\cdot (r+1)+2$.
We examine the non-trivial case where $d\geq 1$.
We present the proof for the case where $P$ is a $({\cal C},D,L)$-mountain
as the case where  $P$ is a $({\cal C},D,L)$-valley is symmetric.

Let ${\cal C} = [C_1, \ldots, C_p]$.
We assume that $P$ is based on $C_{i}$, for some $i\in[p]$.
By the definition of tightness, 
there is a sequence ${\cal P}=[P_{0},\ldots,P_{d}=P]$
of  $({\cal C},D,L)$-mountains (resp. $({\cal C},D,L)$-valleys) based on $C_{i}$
such that 
\begin{itemize}
	\item $P=P_{d}$,
	\item $\forall j\in[0,d],\ \dehe(P_{j})=j\cdot (r+1)+2$, and 
	\item $\forall j\in[0,d-1], P_{j}\subseteq {\disk}(P_{j+1})$.
\end{itemize}
For every $j\in[0,d]$, we denote ${\cal C}^{(j)}=[C_{i},\ldots,C_{i+j\cdot (r+1)+1}]$ and by $\Delta^{(j)}$, the closure of the connected component of $\Delta\setminus (C_{i} \cup C_{i+j\cdot (r+1)+1})$
that is a subset of $\Delta$.
Notice that $\Delta^{(j)}$ is a closed annulus and ${\cal C}^{(j)}$ is a $\Delta^{(j)}$-parallel sequence of cycles of $G$.
Notice that for every $j\in[0,d]$,  $L$ is 
an $\Delta^{(j)}$-avoiding 
and $D$-free ($r$-scattered) linkage of $G$.\medskip

\noindent{\em Claim:} For every $j\in[0,d-1]$ there exists a disjoint 
collection ${\cal Z}_{j}$ of $({\cal C}^{(j)}, L)$-streams of $G$ where $|{\cal Z}_{j}|\geq 2(d-j)+1$.
\smallskip

\noindent{\em Proof of Claim:}	
Let  $j\in[0,d-1]$.
Observe that for each $h\in[j+1, d]$ exactly two of the 
connected components of $\Delta^{(j)}\cap P_{h}$ are $({\cal C}^{(j)}, L)$-rivers in $G$.
This implies that there is a collection ${\cal R}_{j}$
of at least 
$2(d-j)$ many $({\cal C}^{(j)}, L)$-rivers in $G$.
Recall that  ${\cal R}_{j}$ is 
a disjoint 
collection of $({\cal C}^{(j)}, L)$-streams of $G$.
Observe also that we can pick some subpath of 
$\Delta^{(j)}\cap P_{j}$ that has one endpoint 
in $C_{i}$ and 
the other in $C_{i+j\cdot(r+1)+1}$. As this path does not share 
vertices with any of the paths in ${\cal R}_{j}$ 
we can add it in ${\cal R}_{j}$ and obtain a disjoint 
collection ${\cal Z}_{j}$ of $({\cal C}^{(j)}, L)$-streams of $G$ where $|{\cal Z}_{j}|\geq 2(d-j)+1$.  Claim follows (see \autoref{asadfsdgfsdghgdhregergt4tew4rgs}).
	
\begin{figure}[H]
		\centering\scalebox{0.87}{
		\sshow{0}{\begin{tikzpicture}[scale=0.5]
		
		%CYCLES
		
		\fill[celestialblue!30!white] (-2,0.5) to [out=-40, in= 180] (0,0) to  (10,0)  to [out=0, in= -140] (12,0.5) to (12,3.5) to [out=-140, in= 0] (10,3) to (0,3) to  [out=180, in=-40]  (-2,3.5) to  (-2,0.5);
		\foreach \y in {0,1,2,3, 4,5,6}{
			\draw[-] (-2,\y+ 0.5) to [out=-40, in= 180] (0,\y) to  (10,\y)  to [out=0, in= -140] (12,\y+0.5);
		}
		\node[anchor=west] () at (12, 0.5) {$C_i$};
		\node[anchor=west] () at (12, 3.5) {$C_{i+j\cdot(r+1)+1}$};
		\node[anchor=east] () at (-2, 2) {${\cal C}^{(j)}$};
		\node[anchor=west] () at (12, 5.5) {$C_{i+d\cdot(r+1)+1}$};

		\begin{scope}[xshift=5cm]%MOUNTAIN
		
		\draw[crimsonglory, line width=1pt] plot [smooth, tension=1] coordinates {(-2,0) (1.5,4.5) (4,0)};
		\draw[dashed, crimsonglory, line width=1pt] (-2.2,-1) to [out=90, in=260] (-2,-.1);
		\draw[dashed, crimsonglory, line width=1pt] (4,-.1) to [out=280, in=90] (4.2,-1);
		\node () at (-1,5.5) {$P$};
		
		\node[black node] (s5) at (-2.2,-1) {};
		\node[black node] (t5) at (4.2,-1) {};
		
		\draw[crimsonglory, line width=1pt] plot [smooth, tension=1] coordinates {(-3,0) (1,5.5) (5,0)};
		\draw[dashed, crimsonglory, line width=1pt] (-3.2,-1) to [out=90, in=260] (-3,-.1);
		\draw[dashed, crimsonglory, line width=1pt] (5,-.1) to [out=280, in=90] (5.2,-1);
		
		\node[black node] (s4) at (-3.2,-1) {};
		\node[black node] (t4) at (5.2,-1) {};
		
		\draw[crimsonglory, line width=1pt] plot [smooth, tension=1] coordinates {(-.7,0) (1.5,3.5) (3.2,0)};
		\draw[dashed, crimsonglory, line width=1pt] (-1,-1) to [out=90, in=260] (-.7,-.1);
		\draw[dashed, crimsonglory, line width=1pt] (3.2,-.1) to [out=280, in=120] (3.5,-1);
		
		\node[black node] (s3) at (-1,-1) {};
		\node[black node] (t3) at (3.5,-1) {};

		\draw[crimsonglory, line width=1pt] plot [smooth, tension=1] coordinates {(0,0) (1.5,2.5) (2.5,0)};
		\draw[dashed, crimsonglory, line width=1pt] (0,-1) to [out=70, in=260] (0,-.1);
		\draw[dashed, crimsonglory, line width=1pt] (2.5,-.1) to [out=280, in=110] (2.7,-1);
		
		\node[black node] (s4) at (0,-1) {};
		\node[black node] (t4) at (2.7,-1) {};

		\draw[crimsonglory, line width=1pt] plot [smooth, tension=1] coordinates {(1,0) (1.5,1.5) (2,0)};
		\draw[dashed, crimsonglory, line width=1pt] (1,-1) to [out=70, in=260] (1,-.1);
		\draw[dashed, crimsonglory, line width=1pt] (2,-.1) to [out=280, in=110] (2,-1);
		
		\node[black node] (s4) at (1,-1) {};
		\node[black node] (t4) at (2,-1) {};
		\end{scope}

		\begin{scope}[xshift=5cm]%MOUNTAIN
		\begin{scope}
		\clip (-5,0) rectangle (10,3);
		
		\draw[blue, line width=1.3pt] plot [smooth, tension=1] coordinates {(-2,0) (1.5,4.5) (4,0)};
		\draw[dashed, blue, line width=1pt] (-2.2,-1) to [out=90, in=260] (-2,-.1);
		\draw[dashed, blue, line width=1pt] (4,-.1) to [out=280, in=90] (4.2,-1);
		
		\node[black node] (s5) at (-2.2,-1) {};
		\node[black node] (t5) at (4.2,-1) {};

\draw[blue, line width=1.3pt] plot [smooth, tension=1] coordinates {(-3,0) (1,5.5) (5,0)};
\draw[dashed, blue, line width=1.3pt] (5,-.1) to [out=280, in=90] (5.2,-1);
		
		\node[black node] (s4) at (-3.2,-1) {};
		\node[black node] (t4) at (5.2,-1) {};
		
		\end{scope}
		
		\begin{scope}
		\clip (-5,0) rectangle (1.5,3);

		\draw[blue, line width=1.3pt] plot [smooth, tension=1] coordinates {(-.7,0) (1.5,3.5) (3.2,0)};
		\draw[dashed, crimsonglory, line width=1pt] (3.2,-.1) to [out=280, in=120] (3.5,-1);
		
		\node[black node] (s3) at (-1,-1) {};
		\node[black node] (t3) at (3.5,-1) {};
		
		\end{scope}
		\end{scope}
		
		\end{tikzpicture}}}
		\caption{An example of a tight $({\cal C},D,L)$-mountain $P$ based on $C_{i}$ of height $d$ and the respective sequence of $({\cal C},D,L)$-mountains based on $C_{i}$ (depicted in red), an annulus ${\cal C}^{(j)}$ (depicted in cyan), for some $j\in [2,d]$,
and a disjoint 
collection ${\cal Z}_{j}$ (depicted in blue) of $({\cal C}^{(j)}, L)$-streams of $G$. }
			\label{asadfsdgfsdghgdhregergt4tew4rgs}
\end{figure}
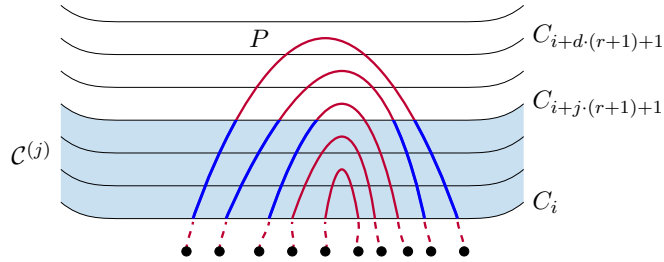

We now set $j'=\lfloor (2d+1)/3\rfloor$ and observe that $0\leq j'\leq d-1$.
The above claim implies that there exists a disjoint 
collection ${\cal Z}_{j'}$ of $({\cal C}^{(j')}, L)$-streams of $G$ such 
that $|{\cal Z}_{j'}|\geq 2(d-j')+1\geq j'=|{\cal C}^{(j')}|$. Therefore, we can apply 
\autoref{u4l49rop0r} on ${\cal C}^{(j')}$ and deduce that 
$\tw(L\cup(\cupall {\cal C}^{(j')}\setminus D))\geq j'$. 
The Lemma follows 
as $L\cup(\cupall {\cal C}^{(j')}\setminus D)\subseteq L\cup (\cupall {\cal C}\setminus D)$ and $\lfloor (2d+1)/3\rfloor\geq 2d/3$. 
\end{proof}

Using~\autoref{ao7ui4jkhq0f} and~\autoref{ahks5llozn}, we prove that the mountains (resp. valleys) of minimal linkages have ``small'' height (resp. depth).

\begin{lemma}
\label{ato954jgd}
Let $\Delta$ be a closed annulus and $r\in \mathbb{N}_{\geq 0}$.
Let $G$ be a \seg,  
${\cal C}$ be a $\Delta$-parallel sequence of cycles 
of $G$, $D$ be a connected subset of $\Delta$, $L$ be 
a  $\Delta$-avoiding and $D$-free $r$-scattered linkage of $G$, and $L'$ be a  $({\cal C},D,L)$-minimal $r$-scattered linkage of $G$. Then all $({\cal C},D,L')$-mountains (resp.  $({\cal C},D,L')$-valleys) of $G$ have height (resp. depth) at most $ \frac{3}{2}\cdot \tw(L'\cup (\cupall {\cal C}\setminus D))$.
\end{lemma}

\begin{proof} 
We set $B=\cupall {\cal C}\setminus D$.
Let ${\cal C} = [C_1, \ldots, C_p]$.
Let $P$ be a  $({\cal C},D,L')$-mountain (resp. $({\cal C},D,L')$-valley) of $G$ based on $C_{i},$
for some $i\in[p-1]$ (resp. $i\in[2,p]$). 
From \autoref{ao7ui4jkhq0f}, $P$ should be tight and, from \autoref{ahks5llozn},  
$\tw(L'\cup B)\geq \frac{2}{3}\cdot \dehe(P)$. Therefore, 
$\dehe(P)\leq \frac{3}{2}\cdot \tw(L'\cup B)$.
\end{proof}

Before concluding this section, we show that, given a closed {\sl disk} $\Delta$, a graph $G$ partially $\Delta$-embedded, a {\sl nested} sequence of cycles $\mathcal{C}$ (i.e., sequence of cycles that crop nested disks of $\Delta$), and a linkage $L$ whose terminals are outside $\Delta$,
every minimal linkage with respect to $L$ and $\mathcal{C}$ does not intersect any ``deep enough'' insulation layer of ${\cal C}$.
%We start by providing a formal definition of {\sl nested} cycles.
%
%\paragraph{Nested cycles.}
%Let $\Delta$ be a closed disk.
%Let $G$ be a  partially $\Delta$-embedded graph and let $\mathcal{C}=[C_{1},\ldots,C_{p}]$, $r\geq 2$, 
%be a collection of  vertex-disjoint cycles of the compass of $G$. We say that the sequence $\mathcal{C}$ is a {\em $\Delta$-nested sequence of cycles} of $G$
%if every $C_{i}$ is the boundary of an open  disk $D_{i}$ of $\Delta$ such that  $\Delta\supseteq D_{1}\supseteq\cdots \supseteq D_{r}$. From now on,
%each $\Delta$-nested sequence $\mathcal{C}$ will be accompanied 
%with the sequence $[D_{1},\ldots,D_{r}]$ of the corresponding open disks 
%as well as 
%the sequence $[\overline{D}_{1},\ldots,\overline{D}_{r}]$ of their closures. 
%Given $x,y\in[r]$ where $x\leq y$, 
%we call the set  $\overline{D}_{x}\setminus D_{y}$ {\em $(x,y)$-annulus} of $\mathcal{C}$
%and we denote it by $\ann(\mathcal{C},x,y)$. Finally, we say that $\ann(\mathcal{C},1,r)$
%is the {\em annulus} of $\mathcal{C}$ and we denote it by $\ann(\mathcal{C})$.  

\begin{lemma}
\label{aop4icl}
Let $\Delta$ be a closed disk and let $r\in\mathbb{N}_{\geq 0}$.
Let $G$ be  a \seg, ${\cal C}=[C_{1},\ldots,C_{p}],$  where $p\geq 3m/2+1$, be  a $\Delta$-nested 
collection of cycles
of $G$, and $L$ be  a $\Delta$-avoiding $r$-scattered linkage. Every $r$-scattered $({\cal C},\emptyset, L)$-minimal linkage $L'$ of $G$ is $\overline{D}_{ 3m/2+1}$-free, where $m=\tw(L'\cup (\cupall {\cal C}\setminus D))$.
\end{lemma}

\begin{proof}
Let $L'$ be a $({\cal C},\emptyset, L)$-minimal linkage of $G$.
Assume to the contrary that 
$L'$ is a linkage of $G$ that is intersecting $\overline{D}_{3m/2+1}$. 
As $L'$ is a $\Delta$-avoiding linkage of $G$
we obtain that 
$G$ contains some $({\cal C},\emptyset,L')$-mountain $P$, based on $C_{1}$
where $\dehe(P)> 3m/2$. 
We set $\Delta'$ to be the closed annulus $\Delta\setminus D_{p}$.
As $L$ is a $\Delta$-avoiding
linkage of $G$, 
it is also a $\Delta'$-avoiding linkage of $G$.
Threfore, we can apply
\autoref{ato954jgd}, on $G$, ${\cal C}$, $\emptyset$, $L$, and $L'$ and obtain 
that  $\dehe(P)\leq  3m/2$, a contradiction.
\end{proof}

\subsection{Routing confined paths}
\label{subsec_rerouting}
In this subsection, we aim to prove~\autoref{alo3qx}, that intuitively shows how to use the infrastructure of a railed annulus of a given graph $G$, in order to obtain an $r$-scattered linkage of $G$ whose endpoints lie in particular parts of the railed annulus and it is confined in ${\cal A}$.
The obtained routing will appear in the core of the proof of~\autoref{afsfsdfdsdsafsadffasdasfd2}.

We start by the following proposition, that can be derived from the proof of \cite[Lemma 7]{AdlerKKLST17irre}.

\begin{proposition}\label{upside}
Let $r\in\mathbb{N}$ and let $k,k',d\in\mathbb{N}$
such that  $1\leq d\cdot (r+1)\leq k'\leq k$.
Let $\Gamma$ be a $(k\times k')$-grid
and let $\{p_{1}^{\rm up},\ldots,p_{d}^{\rm up}\}$
(resp. $\{p_{1}^{\rm down},\ldots,p_{d}^{\rm down}\}$)  be vertices of the higher  (resp. lower) horizontal line arranged as they appear in it from left to right.
If for every $i\in[d]$ $N_\Gamma^{(\leq r)}(p_{i}^{\rm down})\cap \{p_{i-1}^{\rm down}, p_{i+1}^{\rm down}\} = \emptyset$ and $N_\Gamma^{(\leq r)}(p_{i}^{\rm up})\cap \{p_{i-1}^{\rm up}, p_{i+1}^{\rm up}\} = \emptyset$, then the grid $\Gamma$ contains $d$ paths $P_{1},\ldots,P_{d}$
such that, for every $i\in[d]$, the endpoints of $P_{i}$ are $p_{i}^{\rm up}$ 
and $p_{i}^{\rm down}$ and for every $i,j\in[d]$, $i\neq j$, $N_\Gamma^{(\leq r)}(V(P_{i}))\cap V(P_{j}) = \emptyset$.
\end{proposition}

Given two vertex disjoint paths $P_{1}$ and $P_{2}$ of $G$, we say that an $(P_{1},P_{2})$-path of $G$ is a path that 
whose one endpoint is a vertex of $P_{1}$ the other endpoint is a vertex of $P_{2}$ and contains all edges of $P_{1}\cup P_{2}$.
We now prove the following:

\begin{lemma}
\label{alo3qx}
Let $r\in\mathbb{N}$, $p,q,s\in\mathbb{N}_{\geq 3}$, $b,d\in \mathbb{N}_{\geq 1}$, such that $p\geq s+2b$ and $q\geq b+d\cdot (r+1)$,  where $p$ and $s$ are odd numbers.
Also, let $\Delta$ be a closed annulus.
If $G$ is a \seg, ${\cal A}$ is a $\Delta$-embedded $(p,q)$-railed annulus of $G$, $I\subseteq [q]$ where $|I|\geq d\cdot(r+1)$, then there is an $r$-scattered linkage $K$ of $G$ such that, 
\begin{itemize}
\item[(a)]  there is an ordering $[K_{1},\ldots,K_{d}]$ of ${\cal P}(K)$, where for every $i\in[d]$,
$$\text{$K_{i}$  is a $(P_{1,b+(i-1)\cdot (r+1)+1},P_{p,b+(i-1)\cdot(r+1)+1})$-path of $G$ and}$$
		\item[(b)] $K$ is $(s,I)$-confined in ${\cal A}$.
		%\item[(c)] $L\cap C_{1}=\bigcup_{i\in[m]}P_{1,c+i}$ and $L\cap C_{p}=\bigcup_{i\in[m]}P_{r,c+i}$.
	\end{itemize}
\end{lemma}

\begin{proof}
Let ${\cal A} = ({\cal C}, {\cal P})$, let $t =\lfloor p/2 \rfloor$ and $t'=\lfloor s/2 \rfloor$.
Also, let $\{i_{1}, \ldots, i_{d\cdot(r+1)}\}\subseteq I$ such that $\forall j\in[d\cdot (r+1)-1], i_{j}< i_{j+1}$.
\medskip

\noindent{\em Claim:} There is a collection of paths ${\cal P}^{\sf down} = \{P_{1}^{\sf down},\ldots, P_{d}^{\sf down}\}$  such that, for every $h\in [d]$, $P_{h}^{\sf down}$ is a $(P_{1, b+(h-1)\cdot(r+1) +1}, P_{b, i_{h\cdot (r+1)}})$-path and for every $h,j\in[d]$ where $h\neq j$,
it holds that $N_G^{(\leq r)}(V(P_{i}^{\sf down}))\cap V(P_{j}^{\sf down}) = \emptyset$.
\medskip

\noindent{\em Proof of Claim:}
For $i\in[b], j\in[q]$ let $p_{i,j}$ be the vertex obtained after contracting all edges in $P_{i,j}$. We also define  $E_{\sf horizontal} =\bigcup_{(i,j)\in[b]\times[q-1]} \{p_{i,j},p_{i,j+1}\}$ and $E_{\sf vertical} =\bigcup_{(i,j)\in[b-1]\times[q]} \{p_{i,j},p_{i+1,j}\}$.
%\begin{itemize}
%\item $E_{\sf horizontal} =\bigcup_{(i,j)\in[b]\times[q-1]} \{p_{i,j},p_{i,j+1}\}$
%% is the edge obtained
%%after contracting all but one of the edges of $L_{i,j\rightarrow j+1}, i\in[b], j\in[q-1]\}$ and  
%
%\item $E_{\sf vertical} =\bigcup_{(i,j)\in[b-1]\times[q]} \{p_{i,j},p_{i+1,j}\}$.
%% \{e=\{p_{i,j},p_{i+1,j}\} \mid e$ is the edge obtained after   contracting all but one of the edges of $R_{i\rightarrow i+1,j}, i\in[b-1], j\in[q]\}.$
%\end{itemize}

Let $H$ be the graph where $V(H)=\{p_{i,j}\mid  (i,j)\in[b]\times[q]\}$ and $E(H) = E_{\sf horizontal}\cup E_{\sf vertical}$. Observe that $H$
% is a minor of $\ann({\cal C}, 1,b)$ and 
is a minor of $G$ that is isomorphic to a $(q \times b)$-grid (see \autoref{asdgfsdsdfds}). 
For $h\in [d]$, let $p_{h}^{\sf low}$ (resp. $p_{h}^{\sf high}$ ) be the vertex $p_{1, b+(h-1)\cdot(r+1)+1}$ (resp. $p_{b, i_{h\cdot(r+1)}}$).

\begin{figure}[H]
	\centering\scalebox{1}{
	\sshow{0}{\begin{tikzpicture}[yscale=-1,scale=0.5]
	\draw[black] (0,4.0) to
	 node[tiny white node, pos=0.05] (A1) {}
	 node[tiny white node, pos=0.08] (A2) {}
	 node[tiny black node, pos=0.13] (A3) {}
	 node[tiny black node, pos=0.17] (A4) {}
	 node[tiny black node, pos=0.23] (A5) {}
	 node[tiny white node, pos=0.26] (A6) {}
	 node[tiny black node, pos=0.35] (A7) {}
	 node[tiny black node, pos=0.40] (A8) {}
	 node[tiny white node, pos=0.45] (A9) {}
	 node[tiny white node, pos=0.49] (A10) {}
	 node[tiny black node, pos=0.57] (A11) {}
	 node[tiny black node, pos=0.62] (A12) {}
	 node[tiny black node, pos=0.68] (A13) {}
	 node[tiny white node, pos=0.74] (A14) {}
	 node[tiny black node, pos=0.80] (A15) {}
	 node[tiny black node, pos=0.85] (A16) {}
	 node[tiny black node, pos=0.93] (A17) {}
	 node[tiny black node, pos=0.97] (A18) {}
	   (12,4.0);
	\draw[black] (0,3.5) to 	
	 node[tiny black node, pos=0.10] (B1) {}
	 node[tiny black node, pos=0.18] (B2) {}
	 node[tiny black node, pos=0.30] (B3) {}
	 node[tiny black node, pos=0.37] (B4) {}
	 node[tiny black node, pos=0.43] (B5) {}
	 node[tiny black node, pos=0.47] (B6) {}
	 node[tiny black node, pos=0.60] (B7) {}
	 node[tiny black node, pos=0.65] (B8) {}
	 node[tiny black node, pos=0.73] (B9) {}
	 node[tiny black node, pos=0.77] (B10) {}
	 node[tiny black node, pos=0.86] (B11) {}
	 node[tiny black node, pos=0.92] (B12) {}
	 node[tiny black node, pos=0.98] (B13) {}
	(12,3.5);
	\draw[black] (0,3.0) to 
	 node[tiny black node, pos=0.05] (C1) {}
	 node[tiny black node, pos=0.13] (C2) {}
	 node[tiny black node, pos=0.25] (C3) {}
	 node[tiny black node, pos=0.30] (C4) {}
	 node[tiny black node, pos=0.45] (C5) {}
	 node[tiny black node, pos=0.50] (C6) {}
	 node[tiny black node, pos=0.57] (C7) {}
	 node[tiny black node, pos=0.63] (C8) {}
	 node[tiny black node, pos=0.68] (C9) {}
	 node[tiny black node, pos=0.72] (C10) {}
	 node[tiny black node, pos=0.80] (C11) {}
	 node[tiny black node, pos=0.84] (C12) {}
	 node[tiny black node, pos=0.95] (C13) {}
(12,3.0);
	\draw[black] (0,2.5) to 
	 node[tiny black node, pos=0.05] (D1) {}
	 node[tiny black node, pos=0.08] (D2) {}
	 node[tiny black node, pos=0.13] (D3) {}
	 node[tiny black node, pos=0.17] (D4) {}
	 node[tiny black node, pos=0.23] (D5) {}
	 node[tiny black node, pos=0.26] (D6) {}
	 node[tiny black node, pos=0.35] (D7) {}
	 node[tiny black node, pos=0.40] (D8) {}
	 node[tiny black node, pos=0.45] (D9) {}
	 node[tiny black node, pos=0.49] (D10) {}
	 node[tiny black node, pos=0.57] (D11) {}
	 node[tiny black node, pos=0.62] (D12) {}
	 node[tiny black node, pos=0.68] (D13) {}
	 node[tiny black node, pos=0.74] (D14) {}
	 node[tiny black node, pos=0.80] (D15) {}
	 node[tiny black node, pos=0.85] (D16) {}
	 node[tiny black node, pos=0.90] (D17) {}
	 node[tiny black node, pos=0.97] (D18) {}
	  (12,2.5);
	\draw[black] (0,2.0) to 
	 node[tiny black node, pos=0.05] (E1) {}
	 node[tiny black node, pos=0.13] (E2) {}
	 node[tiny black node, pos=0.25] (E3) {}
	 node[tiny black node, pos=0.36] (E4) {}
	 node[tiny black node, pos=0.42] (E5) {}
	 node[tiny black node, pos=0.50] (E6) {}
	 node[tiny black node, pos=0.57] (E7) {}
	 node[tiny black node, pos=0.63] (E8) {}
	 node[tiny black node, pos=0.68] (E9) {}
	 node[tiny black node, pos=0.72] (E10) {}
	 node[tiny black node, pos=0.80] (E11) {}
	 node[tiny black node, pos=0.84] (E12) {}
	 node[tiny black node, pos=0.93] (E13) {}
	 (12,2.0);
	\draw[black] (0,1.5) to
	 node[tiny black node, pos=0.10] (F1) {}
	 node[tiny black node, pos=0.18] (F2) {}
	 node[tiny black node, pos=0.30] (F3) {}
	 node[tiny black node, pos=0.37] (F4) {}
	 node[tiny black node, pos=0.43] (F5) {}
	 node[tiny black node, pos=0.47] (F6) {}
	 node[tiny black node, pos=0.55] (F7) {}
	 node[tiny black node, pos=0.65] (F8) {}
	 node[tiny black node, pos=0.73] (F9) {}
	 node[tiny black node, pos=0.77] (F10) {}
	 node[tiny black node, pos=0.86] (F11) {}
	 node[tiny black node, pos=0.94] (F12) {}
	 node[tiny black node, pos=0.98] (F13) {}
	 (12,1.5);
	\draw[black] (0,1.0) to 
	 node[tiny black node, pos=0.05] (G1) {}
	 node[tiny black node, pos=0.08] (G2) {}
	 node[tiny black node, pos=0.13] (G3) {}
	 node[tiny black node, pos=0.17] (G4) {}
	 node[tiny black node, pos=0.23] (G5) {}
	 node[tiny black node, pos=0.26] (G6) {}
	 node[tiny black node, pos=0.35] (G7) {}
	 node[tiny black node, pos=0.40] (G8) {}
	 node[tiny black node, pos=0.45] (G9) {}
	 node[tiny black node, pos=0.49] (G10) {}
	 node[tiny black node, pos=0.55] (G11) {}
	 node[tiny black node, pos=0.62] (G12) {}
	 node[tiny black node, pos=0.68] (G13) {}
	 node[tiny black node, pos=0.74] (G14) {}
	 node[tiny black node, pos=0.80] (G15) {}
	 node[tiny black node, pos=0.85] (G16) {}
	 node[tiny black node, pos=0.93] (G17) {}
	 node[tiny black node, pos=0.97] (G18) {}
	(12,1.0);
	\draw[black] (0,0.5) to
	 node[tiny black node, pos=0.10] (H1) {}
	 node[tiny black node, pos=0.18] (H2) {}
	 node[tiny black node, pos=0.27] (H3) {}
	 node[tiny black node, pos=0.37] (H4) {}
	 node[tiny black node, pos=0.43] (H5) {}
	 node[tiny black node, pos=0.47] (H6) {}
	 node[tiny black node, pos=0.60] (H7) {}
	 node[tiny black node, pos=0.65] (H8) {}
	 node[tiny black node, pos=0.73] (H9) {}
	 node[tiny black node, pos=0.77] (H10) {}
	 node[tiny black node, pos=0.86] (H11) {}
	 node[tiny black node, pos=0.91] (H12) {}
	 node[tiny black node, pos=0.98] (H13) {}
	(12,0.5);
	\draw[black] (0,0.0) to 
	 node[tiny black node, pos=0.05] (I1) {}
	 node[tiny white node, pos=0.19] (I2) {}
	 node[tiny black node, pos=0.25] (I3) {}
	 node[tiny white node, pos=0.34] (I4) {}
	 node[tiny white node, pos=0.45] (I5) {}
	 node[tiny white node, pos=0.50] (I6) {}
	 node[tiny black node, pos=0.57] (I7) {}
	 node[tiny black node, pos=0.63] (I8) {}
	 node[tiny black node, pos=0.68] (I9) {}
	 node[tiny black node, pos=0.72] (I10) {}
	 node[tiny black node, pos=0.80] (I11) {}
	 node[tiny black node, pos=0.84] (I12) {}
	 node[tiny black node, pos=0.93] (I13) {}
	 (12,0.0);
	
	%P_i 's
	\draw[red, line width = 1pt]
	(A1) -- (A2) -- (B1) -- (C2) -- (D3) -- (D4) -- (E2) -- (F1) -- (G3) -- (H1) -- (I1);
	\draw[red, line width = 1pt]
	(A4) -- (B2) -- (C3) -- (D6) -- (E3) -- (F2) -- (G5) -- (G6)-- (H3) -- (H2) -- (I2);
	\draw[red, line width = 1pt]
	(A6) -- (B3) -- (C4) -- (D7) -- (E4) -- (F4) -- (G7) -- (H4) -- (I4);
	\draw[red, line width = 1pt]
	(A8) -- (B4) -- (B5) -- (C5) -- (D8) -- (E5) -- (F5) -- (G8) -- (H5) -- (I5);
	\draw[red, line width = 1pt]
	(A9) -- (A10)-- (B6) -- (C6) -- (D9) -- (D10) -- (E6) -- (F7) -- (G11) -- (H6) -- (I6);
	\draw[red, line width = 1pt]
	(A11) -- (A12) -- (B7) -- (C8) -- (D11) -- (E7) -- (F8) -- (G12) -- (H8) -- (I7);
	\draw[red, line width = 1pt]
	(A14) -- (B9) -- (B10) -- (C10) -- (D14) -- (E11) -- (F9) -- (F10) -- (G14) -- (H9) -- (I10);
	\draw[red, line width = 1pt]
	(A16) -- (B11) -- (C12) -- (D16) -- (E12) -- (F11) -- (G16) -- (H11) -- (I12) -- (I11);
	\draw[red, line width = 1pt]
	(A18) -- (A17) -- (B12) -- (C13) -- (D17) -- (E13) -- (F12) -- (G17) -- (H12) -- (I13);

	\draw[->, thick] (13.5,2) to (14,2);
	
	\begin{scope}[xshift=15.5cm]
	\foreach \i in {0,...,8}{		
			\draw[-] (0,\i/2) to (4,\i/2);
			\draw[red,line width =1.5pt] (\i/2,0) to (\i/2,4);
		\foreach \j in {0,...,8}{		
			\node[tiny black node] () at (\i/2,\j/2) {};}
		}
	\node[tiny white node] () at (0.5,0){};
	\node[tiny white node] () at (1.0,0){};
	\node[tiny white node] () at (1.5,0){};
	\node[tiny white node] () at (2.0,0){};
	\node[tiny white node] () at (0.0,4){};
	\node[tiny white node] () at (1.0,4){};
	\node[tiny white node] () at (2.0,4){};
	\node[tiny white node] () at (3.0,4){};

	\end{scope}
	\end{tikzpicture}}}
	\caption{An example showing the construction of the graph $H$. For every $h\in [d]$, the resulting vertices $p_{h}^{\sf low}$ and $p_{h}^{\sf high}$ (corresponding to the vertices of the paths $P_{1,b+(h-1)\cdot(r+1) +1}$ and $P_{b,i_{h}}$, respectively) are depicted in white.}
\label{asdgfsdsdfds}
\end{figure}
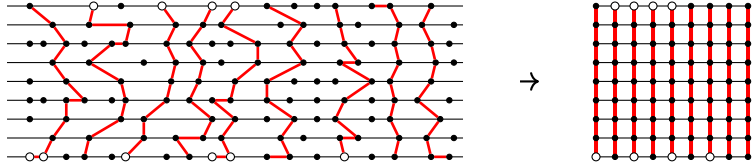

Due to \autoref{upside}, $H$ contains $d$ paths $P_{1},\ldots,P_{d}$
such that, for every $h\in[d]$, the endpoints of $P_{h}$ are $p_{h}^{\rm low}$ 
and $p_{h}^{\rm high}$ and for every $i,j\in[d]$, $i\neq j$, $N_H^{(\leq r)}(V(P_{i}))\cap V(P_{j}) = \emptyset$.
Therefore, if we substitute every vertex of each $P_{i}$ with the
edges that where contracted in $G$ in order to obtain it in $H$,
we obtain the claimed result.\hfill$\diamond$\medskip
	
By applying the previous claim symmetrically, we can find a collection
of paths ${\cal P}^{\sf up} = \{P_{1}^{\sf up},\ldots, P_{d}^{\sf up}\}$
such that, for every $h\in [d]$, $P_{h}^{\sf up}$ is a
$(P_{p-b+1, i_{h\cdot(r+1)}}, P_{p, b+(h-1)\cdot(r+1)+1})$-path and for every $h,j\in[d]$ where $h\neq j$,
it holds that $N_G^{(\leq r)}(V(P_{i}^{\sf up}))\cap V(P_{j}^{\sf up}) = \emptyset$.
	
For every $h\in[d]$, let 
$P_{h}^{\sf mid}=\ann{({\cal C}, b,p-b+1)}\cap P_{i_{h\cdot(r+1)}}$
and let $K_{h} = P_{h}^{\sf down}\cup P_{h}^{\sf mid} \cup P_{h}^{\sf up}$.
Observe that $K=\bigcup_{i\in[d]} K_i$ is an $r$-scattered linkage of $G$ and for every $i\in[d]$, $K_i$ is a $(P_{1,b+(i-1)\cdot (r+1)+1},P_{p,b+(i-1)\cdot(r+1)+1})$-path of $G$.
Since $p\geq s+2b$, $t =\lfloor p/2 \rfloor$, and $t'=\lfloor s/2 \rfloor$, $\ann{({\cal C}, t+1-t',t+1+t')}\subseteq \ann{({\cal C}, b,p-b+1)}$ and therefore $K$ is $(s,I)$-confined in $\mathcal{A}$.
\end{proof}

\subsection{Combing the linkage  - Proof of \autoref{afsfsdfdsdsafsadffasdasfd2}}
\label{subsec_prooftame}

Before we proceed with the proof of \autoref{afsfsdfdsdsafsadffasdasfd2} we need some more definitions. 
Let $\Delta$ be a closed annulus and let ${\cal A}=({\cal C},{\cal P})$ be a $\Delta$-embedded $(p,q)$-railed annulus of a partially $\Delta$-embedded graph $G$. 
\begin{figure}[H]
\centering
	\sshow{0}{\begin{tikzpicture}[scale=.4]
		\foreach \x in {2,...,6}{
			\draw[line width =0.6pt] (0,0) circle (\x cm);
		}
		%	\begin{scope}[on background layer]
		%	\fill[celestialblue!50!white] (0,0) circle (6 cm);
		%	\fill[white] (0,0) circle (2 cm);
		%	\end{scope}
		\node (P3) at (45:7) {$P_{3}$};
		\node[black node] (P11) at (45:6) {};
		\node[black node] (P21a) at (30:5) {};
		\node[black node] (P21b) at (40:5) {};
		\node[black node] (P31a) at (35:4) {};
		\node[black node] (P31b) at (50:4) {};
		\node[black node] (P41a) at (45:3) {};
		\node[black node] (P41b) at (25:3) {};
		\node[black node] (P51) at (40:2) {};
		\draw[line width=1pt] (P11) -- (P21a) -- (P21b) -- (P31a)  (P31b) -- (P41a) -- (P41b) -- (P51);
		
		\node (P4) at (70:7) {$P_{4}$};
		\node[black node] (P12) at (70:6) {};
		\node[black node] (P22a) at (80:5) {};
		\node[black node] (P22b) at (75:5) {};
		\node[black node] (P32a) at (90:4) {};
		\node[black node] (P32b) at (75:4) {};
		\node[black node] (P42a) at (85:3) {};
		\node[black node] (P42b) at (70:3) {};
		\node[black node] (P52) at (75:2) {};
		\draw[line width=1pt] (P12) -- (P22a) -- (P22b) -- (P32a) (P32b) -- (P42a) -- (P42b) -- (P52);
		
		\node (P5) at (115:7) {$P_{5}$};
		\node[black node] (P13a) at (120:6) {};
		\node[black node] (P13b) at (110:6) {};
		\node[black node] (P23a) at (110:5) {};
		\node[black node] (P23b) at (115:5) {};
		\node[black node] (P33a) at (120:4) {};
		\node[black node] (P33b) at (130:4) {};
		\node[black node] (P43a) at (135:3) {};
		\node[black node] (P43b) at (120:3) {};
		\node[black node] (P53) at (125:2) {};
		\draw[line width=1pt] (P13a) -- (P13b) -- (P23a) -- (P23b) -- (P33a) -- (P33b) -- (P43a) -- (P43b) -- (P53);
		
		\node (P6) at (165:7) {$P_{6}$};
		\node[black node] (P14a) at (170:6) {};
		\node[black node] (P14b) at (160:6) {};
		\node[black node] (P24) at (155:5) {};
		\node[black node] (P34) at (160:4) {};
		\node[black node] (P44a) at (165:3) {};
		\node[black node] (P44b) at (180:3) {};
		\node[black node] (P54) at (170:2) {};
		\draw[line width=1pt] (P14a) -- (P14b) -- (P24) -- (P34) -- (P44a) -- (P44b) -- (P54);
		
		\node (P7) at (190:7) {$P_{7}$};
		\node[black node] (P18a) at (190:6) {};
		\node[black node] (P28a) at (185:5) {};
		\node[black node] (P28b) at (220:5) {};
		\node[black node] (P38a) at (210:4) {};
		\node[black node] (P38b) at (225:4) {};
		\node[black node] (P48a) at (205:3) {};
		\node[black node] (P48b) at (220:3) {};
		\node[black node] (P58) at (200:2) {};
		\draw[line width=1pt] (P18a) -- (P28a) to [bend right=15]  (P28b);
		\draw[line width=1pt] (P28b) --  (P38a) -- (P38b) -- (P48a) -- (P48b) -- (P58);
		
		\node (P8) at (235:7) {$P_{8}$};
		\node[black node] (P15) at (235:6) {};
		\node[black node] (P25a) at (240:5) {};
		\node[black node] (P25b) at (250:5) {};
		\node[black node] (P35b) at (245:4) {};
		\node[black node] (P45a) at (250:3) {};
		\node[black node] (P45b) at (240:3) {};
		\node[black node] (P55) at (235:2) {};
		\draw[line width=1pt] (P15) -- (P25a)  -- (P25b) --  (P35b) -- (P45a) -- (P45b) -- (P55);
		
		\node (P1) at (295:7) {$P_{1}$};
		\node[black node] (P16a) at (290:6) {};
		\node[black node] (P16b) at (300:6) {};
		\node[black node] (P26a) at (295:5) {};
		\node[black node] (P26b) at (310:5) {};
		\node[black node] (P36a) at (300:4) {};
		\node[black node] (P36b) at (290:4) {};
		\node[black node] (P46a) at (310:3) {};
		\node[black node] (P46b) at (325:3) {};
		\node[black node] (P56) at (320:2) {};
		\draw[line width=1pt] (P16a) -- (P16b) -- (P26a) -- (P26b) -- (P36a) -- (P36b) -- (P46a) -- (P46b) -- (P56);
		
		\node (P2) at (0:7) {$P_{2}$};
		\node[black node] (P17a) at (5:6) {};
		\node[black node] (P17b) at (-5:6) {};
		\node[black node] (P27a) at (0:5) {};
		\node[black node] (P27b) at (-10:5) {};
		\node[black node] (P37a) at (-15:4) {};
		\node[black node] (P37b) at (0:4) {};
		\node[black node] (P47a) at (10:3) {};
		\node[black node] (P47b) at (-5:3) {};
		\node[black node] (P57) at (-5:2) {};
		\draw[line width=1pt] (P17a) -- (P17b) -- (P27b)  -- (P27a) -- (P37a) -- (P37b) -- (P47a) -- (P47b) -- (P57);
		
		%Lij
		\node[red node] (A1) at (115:5) {};
		\node[red node] () at (155:5) {};
		\node[red node] (A2) at (185:5) {};
		\draw[red, line width=1pt] (A1) to [bend right=30] (A2);
		%	\node () at (150:8) {\red{$L_{2, 5\to 7}$}};
		
		%Rij
		\node[yellow node] (B1) at (310:5) {};
		\node[yellow node] (B2) at (300:4) {};
		\node[yellow node] (B3) at (290:4) {};
		\node[yellow node] (B4) at (310:3) {};
		\draw[yellow, line width=1.5pt] (B1) -- (B2) -- (B3) -- (B4);

		%Dii'jj'
		\node[blue node] (C1) at (0:4) {};
		\node[blue node] (C2) at (10:3) {};
		\node[blue node] (C3) at (-5:3) {};
		\node[blue node] (C4) at (-5:2) {};
		
		\node[blue node] (D1) at (35:4) {};
		\node[blue node] (D2) at (50:4) {};
		\node[blue node] (D3) at (40:2) {};
		
		\node[blue node] (E1) at (90:4) {};
		\node[blue node] (E2) at (75:4) {};
		\node[blue node] (E3) at (75:2) {};
		
		\node[blue node] (F1) at (120:4) {};
		\node[blue node] (F2) at (130:4) {};
		\node[blue node] (F3) at (135:3) {};
		\node[blue node] (F4) at (120:3) {};
		\node[blue node] (F5) at (125:2) {};
		
		\draw[celestialblue, line width=1.5pt] (C1) -- (C2) -- (C3) -- (C4);
		\draw[celestialblue, line width=1.5pt] (F1) -- (F2) -- (F3) -- (F4) -- (F5);
		\begin{scope}
		\clip (125:2) -- (120:4) -- (60:9) -- (0:4) -- (-5:2) -- cycle;
		\draw[celestialblue, line width =1.5pt] (0,0) circle (2 cm);
		\draw[celestialblue, line width =1.5pt] (0,0) circle (4 cm);
		\end{scope}
		
		\begin{scope}[on background layer]
		\clip (60:9) -- (0:4) -- (10:3) -- (-5:3) -- (-5:2) -- (40:2)-- (75:2) -- (125:2) -- (120:3) -- (135:3) -- (130:4) -- (120:4) -- cycle;
		\filldraw[draw= celestialblue, fill= celestialblue!50!white, line width =1.5pt] (0,0) circle (4 cm);
		\filldraw[draw=white,fill=white, line width =1.5pt] (0,0) circle (2 cm);
		\end{scope}

		\begin{scope}
		\clip  (320:2)  -- (310:3) -- (290:4) -- (295:5) -- (290:6) -- (270:9) -- (235:6) -- (250:5) -- (245:4) -- (250:3) --  (235:2) -- (270:1);
		
		\foreach \x in {2,...,6}{
			\draw[green, line width =1pt] (0,0) circle (\x cm);
		}

		\node[black node] (P16a) at (290:6) {};
		\node[black node] (P16b) at (300:6) {};
		\node[black node] (P26a) at (295:5) {};
		\node[yellow node] (P36a) at (300:4) {};
		\node[yellow node] (P36b) at (290:4) {};
		\node[yellow node] (P46a) at (310:3) {};
		\node[black node] (P56) at (320:2) {};
		\node[black node] (P15) at (235:6) {};
		\node[black node] (P25a) at (240:5) {};
		\node[black node] (P25b) at (250:5) {};
		\node[black node] (P35b) at (245:4) {};
		\node[black node] (P45a) at (250:3) {};
		\node[black node] (P45b) at (240:3) {};
		\node[black node] (P55) at (235:2) {};
		\end{scope}
		%		\node (C1) at (270:7) {$C_{1}$};
		%		\node (C5) at (270:1) {$C_{5}$};
		\end{tikzpicture}}
	
	\caption{An example of a $(5,8)$-railed annulus ${\cal A}$, the set $F_{{\cal A}}$ (depicted in green), and the graphs $L_{2, 5\to 7}$ (depicted in red), $R_{2\to 4, 1}$ (depicted in yellow), and $\Delta_{3,5,2,5}$ (depicted in blue).}
\label{sdfadgdfgagdbgbfb}
\end{figure}
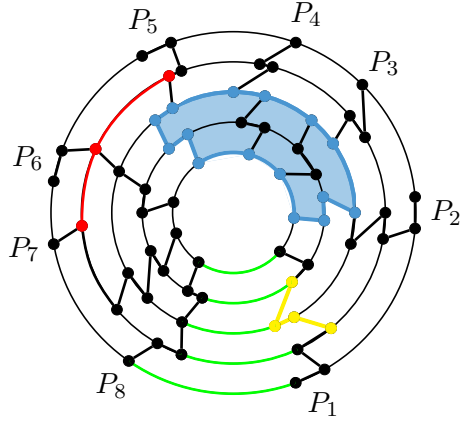

We refer the reader to~\autoref{sdfadgdfgagdbgbfb} for an illustration of the following definitions.
For every $i\in[p]$, we define  $F^{(i)}_{\cal A}$ as the edge set of the unique  
$(P_{i,q},P_{i,1})$-path  that does not contain 
any vertex from $P_{2}$. We also set $F_{\cal A}=\bigcup_{i\in[p]}F^{(i)}$. %\marg{$F_{\cal A}$}
%
%\begin{lemma}
%\label{e9iuf}
%Let $G$ be a \seg, $k\in\Bbb{N}_{\geq 1}$  and $D$ be a closed disk of $\Delta$, whose 
%boundary is a cycle of $G$ whose vertices, in clockwise order are 
%$[n_{1},\ldots,n_{k-1},e_{1},]$
%Let also ${\cal P}_{1}$ and ${\cal P}_{2}$ be two collections of paths in $G$ such that 
%\begin{itemize}
%\item all paths in ${\cal P}_{1}$ are pairwise vertex disjoint and have their endpoints 
%on the boundary of $D$.
%\item all paths in ${\cal P}_{2}$ are pairwise vertex disjoint and have their endpoints 
%on the boundary of $D$, and
%\item if $P_{1}\in{\cal P}_{1}$ and $P_{2}\in{\cal P}_{2}$, then $D\cap(V(P_{1}\cap P_{2}))\neq\emptyset$,
%\end{itemize}
%then $\tw(G)\geq\min\{r,q\}-1$.  
%\end{lemma}
Let $(i,j,j')\in[p]\times[q]^2$ where $j\neq j'$. We denote by $L_{i,j\rightarrow j'}$%\marg{$L_{i,j\rightarrow j'}$}
the shortest  path in $C_{i}$ starting from a vertex of $P_{i,j}$
and finishing to a vertex of $P_{i,j'}$ and that  does not contain any edge from $F_{\cal A}$.
Let $(i,i',j)\in[p]^2\times[q]$ where $i\neq i'$. We denote by $R_{i\rightarrow i',j}$
the shortest  path in $P_{j}$ starting from a vertex of $P_{i,j}$
and finishing to a vertex of $P_{i',j}$.
Let $(i,i',j,j')\in [p]^2\times[q]^2$ such that $i< i'$ and $j< j'$.
We define $\Delta_{i,i',j,j'}$ as the closed disk bounded by the unique cycle in the graph 
%
%\vspace{-3mm}
\begin{eqnarray*}
& P_{i,j}\cup L_{i,j\to j'}\cup P_{i,j'}\cup R_{i\to i',j'}\cup P_{i',j'}\cup L_{i',j'\to j}\cup P_{i',j}\cup R_{i'\to i,j}.
\end{eqnarray*}
%\vspace{-3mm}

%
%We also  set ${\ann}({\cal A})=\ann({\cal C})$.\marginpar{$\ann({\cal A})$.}
%
%
%
%\smallskip\noindent {\bf Linkages.}
%A \emph{linkage} in a graph $G$ is a subgraph $L$ of $G$  whose connected components are all non-trivial paths. The {\em paths} of a linkage are its connected components and we denote them by ${\cal P}(L).$
%The {\em size} of $L$ is the number of its paths and is denoted by $|L|$,
%The \emph{terminals} of a linkage $L$, denoted by $T(L)$, are the endpoints of the paths in ${\cal P}(L)$,
%and
%%\marg{$|L|$, $T(L)$,  ${\cal P}(L)$}
%the \emph{pattern} of $L$ is the set $$\big\{\{s,t\}\mid {\cal P}(L)\mbox{ contains some $(s,t)$-path}\big\}.$$ 
%Two linkages $L_{1},L_{2}$ of $G$  are {\em equivalent} if they have the same pattern and we denote this fact by $L_{1}\equiv L_{2}$.%\marg{$\equiv$}
%%
%We say that a linkage $L$ of a graph $G$ is {\em vital} if $V(L)=V(G)$ and there is no other linkage of $G$ that is equivalent to $L$.

%\vspace{-5mm}

Let $\Delta$ be a closed annulus.
Let $\mathcal{A}=(\mathcal{C},\mathcal{P})$ be a $(p,q)$-railed annulus of a partially $\Delta$-embedded graph $G$.
We set $z=\lfloor\min\{p,q\}/2\rfloor$.
For each $i\in[z]$, we define $${C}^\mathcal{A}_{i}=\bd(\Delta_{i,p-i+1,i,q-i+1}).$$
%\begin{eqnarray*}
%L_{i,i\rightarrow q-i+1}\cup L_{p-i+1,i\rightarrow q-i+1}\cup R_{i\rightarrow p-i+1,i}\cup R_{i\rightarrow p-i+1,q-i+1}.
%\end{eqnarray*}
%%\vspace{-4mm}

We set $\mathcal{C}_\mathcal{A} = [C^\mathcal{A}_{1},\ldots,C^\mathcal{A}_{z}]$.
If $p,q\geq 5$, we use $\Delta_{\mathcal{C}_\mathcal{A}}$ to denote the closed disk $\ann(\mathcal{C}^\mathcal{A},1,z) = \Delta_{1,p,1,q}$.
Notice that if $p,q\geq 5$, then $\mathcal{C}_\mathcal{A}$ is a $\Delta_{\mathcal{C}_\mathcal{A}}$-nested collection of cycles of $G$ (see \autoref{asdfsdfgdhfgsdhffsdhnsfdhgsa}).%\marg{$\mathcal{C}_\mathcal{A}$}

\begin{figure}[H]
	\centering\scalebox{1}{
	\sshow{0}{\begin{tikzpicture}[scale=0.6]
	\draw[black] (0,4.0) to
	 node[tiny black node, pos=0.05] (A1) {}
	 node[tiny black node, pos=0.08] (A2) {}
	 node[tiny black node, pos=0.13] (A3) {}
	 node[tiny black node, pos=0.17] (A4) {}
	 node[tiny black node, pos=0.23] (A5) {}
	 node[tiny black node, pos=0.26] (A6) {}
	 node[tiny black node, pos=0.35] (A7) {}
	 node[tiny black node, pos=0.40] (A8) {}
	 node[tiny black node, pos=0.45] (A9) {}
	 node[tiny black node, pos=0.49] (A10) {}
	 node[tiny black node, pos=0.57] (A11) {}
	 node[tiny black node, pos=0.62] (A12) {}
	 node[tiny black node, pos=0.68] (A13) {}
	 node[tiny black node, pos=0.74] (A14) {}
	 node[tiny black node, pos=0.80] (A15) {}
	 node[tiny black node, pos=0.85] (A16) {}
	 node[tiny black node, pos=0.93] (A17) {}
	 node[tiny black node, pos=0.97] (A18) {}
	   (12,4.0);
	\draw[black] (0,3.5) to 	
	 node[tiny black node, pos=0.10] (B1) {}
	 node[tiny black node, pos=0.18] (B2) {}
	 node[tiny black node, pos=0.30] (B3) {}
	 node[tiny black node, pos=0.37] (B4) {}
	 node[tiny black node, pos=0.43] (B5) {}
	 node[tiny black node, pos=0.47] (B6) {}
	 node[tiny black node, pos=0.60] (B7) {}
	 node[tiny black node, pos=0.65] (B8) {}
	 node[tiny black node, pos=0.73] (B9) {}
	 node[tiny black node, pos=0.77] (B10) {}
	 node[tiny black node, pos=0.86] (B11) {}
	 node[tiny black node, pos=0.92] (B12) {}
	 node[tiny black node, pos=0.98] (B13) {}
	(12,3.5);
	\draw[black] (0,3.0) to 
	 node[tiny black node, pos=0.05] (C1) {}
	 node[tiny black node, pos=0.13] (C2) {}
	 node[tiny black node, pos=0.25] (C3) {}
	 node[tiny black node, pos=0.30] (C4) {}
	 node[tiny black node, pos=0.45] (C5) {}
	 node[tiny black node, pos=0.50] (C6) {}
	 node[tiny black node, pos=0.57] (C7) {}
	 node[tiny black node, pos=0.63] (C8) {}
	 node[tiny black node, pos=0.68] (C9) {}
	 node[tiny black node, pos=0.72] (C10) {}
	 node[tiny black node, pos=0.80] (C11) {}
	 node[tiny black node, pos=0.84] (C12) {}
	 node[tiny black node, pos=0.95] (C13) {}
(12,3.0);
	\draw[black] (0,2.5) to 
	 node[tiny black node, pos=0.05] (D1) {}
	 node[tiny black node, pos=0.08] (D2) {}
	 node[tiny black node, pos=0.13] (D3) {}
	 node[tiny black node, pos=0.17] (D4) {}
	 node[tiny black node, pos=0.23] (D5) {}
	 node[tiny black node, pos=0.26] (D6) {}
	 node[tiny black node, pos=0.35] (D7) {}
	 node[tiny black node, pos=0.40] (D8) {}
	 node[tiny black node, pos=0.45] (D9) {}
	 node[tiny black node, pos=0.49] (D10) {}
	 node[tiny black node, pos=0.57] (D11) {}
	 node[tiny black node, pos=0.62] (D12) {}
	 node[tiny black node, pos=0.68] (D13) {}
	 node[tiny black node, pos=0.74] (D14) {}
	 node[tiny black node, pos=0.80] (D15) {}
	 node[tiny black node, pos=0.85] (D16) {}
	 node[tiny black node, pos=0.90] (D17) {}
	 node[tiny black node, pos=0.97] (D18) {}
	  (12,2.5);
	\draw[black] (0,2.0) to 
	 node[tiny black node, pos=0.05] (E1) {}
	 node[tiny black node, pos=0.13] (E2) {}
	 node[tiny black node, pos=0.25] (E3) {}
	 node[tiny black node, pos=0.36] (E4) {}
	 node[tiny black node, pos=0.42] (E5) {}
	 node[tiny black node, pos=0.50] (E6) {}
	 node[tiny black node, pos=0.57] (E7) {}
	 node[tiny black node, pos=0.63] (E8) {}
	 node[tiny black node, pos=0.68] (E9) {}
	 node[tiny black node, pos=0.72] (E10) {}
	 node[tiny black node, pos=0.80] (E11) {}
	 node[tiny black node, pos=0.84] (E12) {}
	 node[tiny black node, pos=0.93] (E13) {}
	 (12,2.0);
	\draw[black] (0,1.5) to
	 node[tiny black node, pos=0.10] (F1) {}
	 node[tiny black node, pos=0.18] (F2) {}
	 node[tiny black node, pos=0.30] (F3) {}
	 node[tiny black node, pos=0.37] (F4) {}
	 node[tiny black node, pos=0.43] (F5) {}
	 node[tiny black node, pos=0.47] (F6) {}
	 node[tiny black node, pos=0.55] (F7) {}
	 node[tiny black node, pos=0.65] (F8) {}
	 node[tiny black node, pos=0.73] (F9) {}
	 node[tiny black node, pos=0.77] (F10) {}
	 node[tiny black node, pos=0.86] (F11) {}
	 node[tiny black node, pos=0.94] (F12) {}
	 node[tiny black node, pos=0.98] (F13) {}
	 (12,1.5);
	\draw[black] (0,1.0) to 
	 node[tiny black node, pos=0.05] (G1) {}
	 node[tiny black node, pos=0.08] (G2) {}
	 node[tiny black node, pos=0.13] (G3) {}
	 node[tiny black node, pos=0.17] (G4) {}
	 node[tiny black node, pos=0.23] (G5) {}
	 node[tiny black node, pos=0.26] (G6) {}
	 node[tiny black node, pos=0.35] (G7) {}
	 node[tiny black node, pos=0.40] (G8) {}
	 node[tiny black node, pos=0.45] (G9) {}
	 node[tiny black node, pos=0.49] (G10) {}
	 node[tiny black node, pos=0.55] (G11) {}
	 node[tiny black node, pos=0.62] (G12) {}
	 node[tiny black node, pos=0.68] (G13) {}
	 node[tiny black node, pos=0.74] (G14) {}
	 node[tiny black node, pos=0.80] (G15) {}
	 node[tiny black node, pos=0.85] (G16) {}
	 node[tiny black node, pos=0.93] (G17) {}
	 node[tiny black node, pos=0.97] (G18) {}
	(12,1.0);
	\draw[black] (0,0.5) to
	 node[tiny black node, pos=0.10] (H1) {}
	 node[tiny black node, pos=0.18] (H2) {}
	 node[tiny black node, pos=0.27] (H3) {}
	 node[tiny black node, pos=0.37] (H4) {}
	 node[tiny black node, pos=0.43] (H5) {}
	 node[tiny black node, pos=0.47] (H6) {}
	 node[tiny black node, pos=0.60] (H7) {}
	 node[tiny black node, pos=0.65] (H8) {}
	 node[tiny black node, pos=0.73] (H9) {}
	 node[tiny black node, pos=0.77] (H10) {}
	 node[tiny black node, pos=0.86] (H11) {}
	 node[tiny black node, pos=0.91] (H12) {}
	 node[tiny black node, pos=0.98] (H13) {}
	(12,0.5);
	\draw[black] (0,0.0) to 
	 node[tiny black node, pos=0.05] (I1) {}
	 node[tiny black node, pos=0.19] (I2) {}
	 node[tiny black node, pos=0.25] (I3) {}
	 node[tiny black node, pos=0.34] (I4) {}
	 node[tiny black node, pos=0.45] (I5) {}
	 node[tiny black node, pos=0.50] (I6) {}
	 node[tiny black node, pos=0.57] (I7) {}
	 node[tiny black node, pos=0.63] (I8) {}
	 node[tiny black node, pos=0.68] (I9) {}
	 node[tiny black node, pos=0.72] (I10) {}
	 node[tiny black node, pos=0.80] (I11) {}
	 node[tiny black node, pos=0.84] (I12) {}
	 node[tiny black node, pos=0.93] (I13) {}
	 (12,0.0);
	 
	 \draw[red,line width =1pt]
	 (A2) -- (B1) -- (C2) -- (D3) -- (D4) -- (E2) -- (F1) -- (G3) -- (H1) -- (I1);
	\draw[-]
	(A4) -- (B2) -- (C3) -- (D6) -- (E3) -- (F2) -- (G5) -- (G6)-- (H3) -- (H2) -- (I2);
	\draw[red,line width =1pt]
	(B2) -- (C3) -- (D6) -- (E3) -- (F2) -- (G5) -- (G6)-- (H3);
	\draw[-]
	(A6) -- (B3) -- (C4) -- (D7) -- (E4) -- (F4) -- (G7) -- (H4) -- (I4);
	\draw[red,line width =1pt]
	(C4) -- (D7) -- (E4) -- (F4) -- (G7);
	\draw[-]
	(A8) -- (B4) -- (B5) -- (C5) -- (D8) -- (E5) -- (F5) -- (G8) -- (H5) -- (I5);
	\draw[red,line width =1pt]
	(D8) -- (E5) -- (F5);
	\draw[-]
	(A9) -- (A10)-- (B6) -- (C6) -- (D9) -- (D10) -- (E6) -- (F7) -- (G11) -- (H6) -- (I6);
	\draw[-]
	(A11) -- (A12) -- (B7) -- (C8) -- (D11) -- (E7) -- (F8) -- (G12) -- (H8) -- (I7);
	\draw[red,line width =1pt]
	(D11) -- (E7) -- (F8);
	\draw[-]
	(A14) -- (B9) -- (B10) -- (C10) -- (D14) -- (E11) -- (F9) -- (F10) -- (G14) -- (H9) -- (I10);
	\draw[red,line width =1pt]
	(C10) -- (D14) -- (E11) -- (F9) -- (F10) -- (G14);
	\draw[-]
	(A16) -- (B11) -- (C12) -- (D16) -- (E12) -- (F11) -- (G16) -- (H11) -- (I12) -- (I11);
	\draw[red,line width =1pt]
	(B11) -- (C12) -- (D16) -- (E12) -- (F11) -- (G16) -- (H11);
	\draw[red,line width =1pt]
	(A17) -- (B12) -- (C13) -- (D17) -- (E13) -- (F12) -- (G17) -- (H12) -- (I13);
	
	\foreach \i/\j in {2/3,3/4,4/5,5/6,6/7,7/8,8/9,9/10,10/11,11/12,12/13,13/14,14/15,15/16,16/17}{
		\draw[red,line width =1pt] (A\i) to (A\j);
	}
	
	\foreach \i/\j in {2/3,3/4,4/5,5/6,6/7,7/8,8/9,9/10,10/11}{
		\draw[red,line width =1pt] (B\i) to (B\j);
	}
	
	\foreach \i/\j in {4/5,5/6,6/7,7/8,8/9,9/10}{
		\draw[red,line width =1pt] (C\i) to (C\j);
	}
	\foreach \i/\j in {8/9,9/10,10/11}{
		\draw[red,line width =1pt] (D\i) to (D\j);
	}
	
	\foreach \i/\j in {5/6,6/7,7/8}{
		\draw[red,line width =1pt] (F\i) to (F\j);
	}
	
	\foreach \i/\j in {7/8,8/9,9/10,10/11,11/12,12/13,13/14}{
		\draw[red,line width =1pt] (G\i) to (G\j);
	}
	
	\foreach \i/\j in {3/4,4/5,5/6,6/7,7/8,8/9,9/10,10/11}{
		\draw[red,line width =1pt] (H\i) to (H\j);
	}
	
	\foreach \i/\j in {1/2,2/3,3/4,4/5,5/6,6/7,7/8,8/9,9/10,10/11,11/12,12/13}{
		\draw[red,line width =1pt] (I\i) to (I\j);
	}
\end{tikzpicture}}}
\caption{A $(9,9)$-railed annulus $\mathcal{A} = (\mathcal{C}, \mathcal{P})$ and the sequence $C_\mathcal{A}$ (depicted in red).}
\label{asdfsdfgdhfgsdhffsdhnsfdhgsa}
\end{figure}
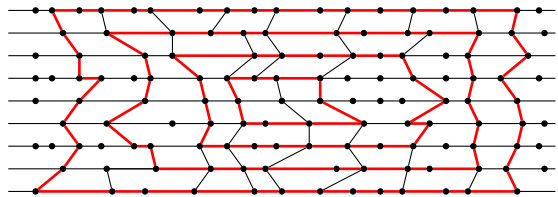

%We are now ready to prove \autoref{afsfsdfdsdsafsadffasdasfd2}.

\begin{proof}[Proof of \autoref{afsfsdfdsdsafsadffasdasfd2}]
We set $m=f_{{\cal G},r}(k)$.
We define $\funref{fun_somelinkage}(m,r) := 3 m^2+6m + 2rm+ 2r+ 2.$
Let also $b=3m/2$, and keep in mind that
$$p\geq \funref{fun_somelinkage}(m,r) +s = 3 m^2+6m + 2rm+ 2r+ 2 + s=2\cdot(m+1)\cdot (b+r)+2+s+2b$$
and that $|I|\geq m\cdot (r+1)$. 

Recall that  
${\cal C}_{\cal A}=[C_{1}',\ldots,C_{z}']$, where $z=\lfloor \min\{p,q\}/2\rfloor$,
is a $\Delta_{{\cal C}_{\cal A}}$-nested collection of cycles of $G$.
For each $i\in[z]$, we denote by ${D}_{i}'$ the closed disk corresponding to $C_{i}'$.
Let also  $D:={D}_{b+1}'$.
Keep in mind that  ${D}_{1}' = \Delta_{1,p,1,q} = \Delta_{{\cal C}_{\cal A}}$
and $D = \Delta_{b+1,p-b,b+1, q-b}$.

Observe now that $L$ is a $\Delta_{{\cal C}_{\cal A}}$-avoiding linkage.
Let $L'$ be a $({\cal C}_{\cal A},\Delta_{{\cal C}_{\cal A}},L)$-minimal $r$-scattered linkage.
Given that, by definition,  ${\cal G}$ is hereditary, we have that since $G\in{\cal G}$ and $L'\cup(\cupall{\cal C}_{\cal A}\setminus \Delta_{{\cal C}_{\cal A}})$
 is a subgraph of $G$, $L'\cup(\cupall{\cal C}_{\cal A}\setminus \Delta_{{\cal C}_{\cal A}})\in {\cal G}$.
 Therefore, by definition of ${\cal G}$ and~\autoref{ap43k9s}, $\tw(L'\cup(\cupall{\cal C}_{\cal A}\setminus \Delta_{{\cal C}_{\cal A}}))\leq m$.
By applying \autoref{aop4icl} on $G$, ${\cal C}_{\cal A}$, $L$, and $ \Delta_{{\cal C}_{\cal A}}$,
we obtain that $L'$ is a $D$-free linkage of $G$.

It is easy to verify that $L'$ is $\Delta$-avoiding,
$p>m$, $D\subseteq\Delta$, and $|L'|=|L|\leq k$.
Consider a $({\cal C},D,L')$-minimal $r$-scattered linkage $L''$ of $G$.
Again, heredity of ${\cal G}$ implies that $L''\cup(\cupall{\cal C}\setminus D)\in {\cal G}$.
 Therefore, by definition of ${\cal G}$ and~\autoref{ap43k9s}, $\tw(L''\cup(\cupall{\cal C}\setminus D))\leq m$.
We may now apply \autoref{ato954jgd} and~\autoref{fskfsl} on $k$, $G$, ${\cal A}$, $D$, and $L'$ and deduce that  
\begin{itemize}
	\item[(i.)] All $({\cal C},D,L'')$-mountains/valleys of $G$ have height/depth at most   $b$.
	\item[(ii.)] $L''$ has at most $m$ ${\cal A}$-rivers of $G$,
\end{itemize}
Let ${\cal Z}=[Z_1,\ldots,Z_d]$  be the $D$-ordering of the  ${\cal A}$-rivers of $L''$ in $G$ and
keep in mind that, from (ii.), $d\leq m$.
Also, since for every $i\in[d]$, $Z_i$ is a subgraph of $L''$ and $L''$ is an $r$-scattered linkage, it holds that for every $i,j\in[d]$ where $i\neq j$, $N_G^{(\leq r)}(V(Z_i))\cap V(Z_j)=\emptyset$.

{Recall that $m\geq r$.}
For every $i\in[d]$, we define $x_{i}^{\rm down}$ (resp. $x_{i}^{\rm up}$) as the vertex 
in the path $C_{(i+1)\cdot b+1}\setminus \inter(D)$ (resp. $C_{p-(i+1)\cdot b}\setminus \inter(D)$)
that belongs in $Z_{i}$ and has the minimum possible distance 
to the vertices of the path $P_{(i+1)\cdot b+1,q-b}$ (resp. $P_{p-(i+1)\cdot b,q-b}$).
We also denote by $Q_{i}^{\rm down}$ (resp. $Q_{i}^{\rm up}$) the path certifying this minimum distance.
Since $m\geq r$ and $b=3m/2$, we have that the union of $Q_{1}^{\rm down}, \ldots,Q_{d}^{\rm down},Q_{1}^{\rm up}, \ldots, Q_{d}^{\rm up}$ is an $r$-scattered linkage.

	%For every $i\in[0,d-1]$, we define $\hat{x}_{i}^{\rm down}$ (resp. $\hat{x}_{i}^{\rm up}$) as the vertex 
	%in the graph $C_{(i+1)\cdot b+1}\setminus \breve{D}$ (resp. $C_{r-(i+1)\cdot b}\setminus \breve{D}$)
	%that belongs in $Z_{i+1}$ and has the minimum possible distance 
	%to ${x}_{i}^{\rm down}$ (resp. ${x}_{i}^{\rm up}$) --- in the case where $i=0$, then use any vertex of $L_{1}^{\rm down}\cap D$ (resp. $L_{1}^{\rm up}\cap D$) instead of ${x}_{i}^{\rm down}$ (resp. ${x}_{i}^{\rm up}$).  As before, we denote by  $\hat{Q}_{i}^{\rm down}$ (resp. $\hat{Q}_{i}^{\rm up}$) the path certifying this minimum distance.
	
	%
	%\gr{Δες ότι τα διαστήματα δεν είναι ισυνεχή και πρέπει ίσως να πάει κάποιος πιο δεξιά για να τα ορίσει!}
		
For $i\in[d]$, let $Z_{i}^{\rm down}$ and  $Z_{i}^{\rm up}$ be the two connected
components of the graph obtained from $Z_{i}$ if we remove 
the edges of its $(x_{i}^{\rm down},x_{i}^{\rm up})$-subpath (see \autoref{asadfdsfdsfsdafsdf}).
We choose  $Z_{i}^{\rm down}$ (resp. $Z_{i}^{\rm up}$) so that 
it intersects $C_{1}$ (resp. $C_{p}$).
%\vspace{-3mm}
	
	\begin{figure}[H]
		\centering\scalebox{0.9}{
		\sshow{0}{\begin{tikzpicture}[scale=0.55]
		\foreach \y in {0,0.5, 2,3,5,6, 7.5,8}{		
			\draw[-] (2,\y) to (19,\y);
		}
	
		\foreach \y in {0.7, 1.8,2.2,2.8,3.2,4.8,5.2,5.8,6.2, 7.3}{		
		\draw[-, opacity=0.1] (2,\y) to (19,\y);
	}
		
		%P1
%		\node[track node 1] (1P1) at (1,8) {};
%		\node[track node 2] (1P2) at (0.5,7.5) {};
%		\node[track node 2] (1P3) at (1.5,7.5) {};
%		\node[track node 2] (1P4) at (1,6) {};
%		\node[track node 2] (1P5) at (1.5,5) {};
%		\node[track node 2] (1P6) at (1,3) {};
%		\node[track node 2] (1P7) at (1.5,2) {};
%		\node[track node 2] (1P8) at (1,2) {};
%		\node[track node 2] (1P9) at (0.5,0.5) {};
%		\node[track node 1] (1P10) at (1,0) {};
%		
%		\foreach \i/\j in {1/2,2/3,3/4,4/5,5/6,6/7,7/8,8/9,9/10}{
%			\draw[opacity=0.3] (1P\i) to (1P\j);
%		}
		
		%P2
		\node[track node 1] (2P1) at (5,8) {};
		\node[track node 1] (2P2) at (6,8) {};
		\node[track node 2] (2P3) at (6,7.5) {};
		\node[track node 2] (2P4) at (5.5,6) {};
		\node[track node 2] (2P5) at (5,5) {};
		\node[track node 2] (2P6) at (4.5,3) {};
		\node[track node 2] (2P7) at (5,2) {};
		\node[track node 2] (2P8) at (5.5,0.5) {};
		\node[track node 2] (2P9) at (6,0.5) {};
		\node[track node 1] (2P10) at (6.5,0) {};
		
		\foreach \i/\j in {1/2,2/3,3/4,4/5,5/6,6/7,7/8,8/9,9/10}{
			\draw[opacity=0.3] (2P\i) to (2P\j);
		}
		
		%P3
		\node[track node 1] (3P1) at (8,8) {};
		\node[track node 2] (3P2) at (7.5,7.5) {};
		\node[track node 2] (3P3) at (7.5,6) {};
		\node[track node 2] (3P4) at (7,6) {};
		\node[track node 2] (3P5) at (7.5,5) {};
		\node[track node 2] (3P6) at (8,3) {};
		\node[track node 2] (3P7) at (8.5,2) {};
		\node[track node 2] (3P8) at (8,2) {};
		\node[track node 2] (3P9) at (9,0.5) {};
		\node[track node 1] (3P10) at (9,0) {};
		
		\foreach \i/\j in {1/2,2/3,3/4,4/5,5/6,6/7,7/8,8/9,9/10}{
			\draw[opacity=0.3] (3P\i) to (3P\j);
		}
		
		%P4
		\node[track node 1] (4P1) at (17,8) {};
		\node[track node 2] (4P2) at (17.5,7.5) {};
		\node[track node 2] (4P3) at (18,6) {};
		\node[track node 2] (4P4) at (18,5) {};
		\node[track node 2] (4P5) at (18.5,5) {};
		\node[track node 2] (4P6) at (18,3) {};
		\node[track node 2] (4P7) at (17.5,2) {};
		\node[track node 2] (4P8) at (17,0.5) {};
		\node[track node 1] (4P9) at (18,0) {};
		\node[track node 1] (4P10) at (18.5,0) {};
		
		\foreach \i/\j in {1/2,2/3,3/4,4/5,5/6,6/7,7/8,8/9,9/10}{
			\draw[opacity=0.3] (4P\i) to (4P\j);
		}
		%1T
		\draw[blue, line width=1pt] plot [smooth, tension=1] coordinates {(11,8) (11.5,7) (10.5,6.5) (11,6)};
		\draw[blue, line width=1pt] plot [smooth, tension=1] coordinates {(11,2) (11.5,1) (12,2.5) (12.5,1.5) (12,1) (12,0)};
		\draw[blue!40!white, line width=1pt] plot [smooth, tension=1] coordinates {(11,2) (11.5,3) (11,4) (12,5) (11.5, 5.3) (11,6)};
		
		%2T
		\draw[blue, line width=1pt] plot [smooth, tension=1] coordinates {(14,8) (15,7.2) (13.5,6.8) (14.5,6.3) (14,6) (14.5,5.5) (14,5) };
		\draw[blue, line width=1pt] plot [smooth, tension=1] coordinates {(15,3) (15.7,2.5) (14.5,1) (15,0.5) (14.5,0)};
		\draw[blue!40!white, line width=1pt] plot [smooth, tension=1] coordinates {(14,5) (13.8,4) (14.5,4.5) (15,4) (14.5,3.5) (15,3)};

		\node[small black node] (1T1) at (11,8) {};
		\node[small black node] (1T2) at (11,6) {};
		\node[small black node] (1T3) at (11,2) {};
		\node[small black node] (1T4) at (12,0) {};
		
		\node[small black node] (2T1) at (14,8) {};
		\node[small black node] (2T2) at (14,5) {};
		\node[small black node] (2T3) at (15,3) {};
		\node[small black node] (2T4) at (14.5,0) {};

		\draw[crimsonglory, line width=1pt] (1T2) -- (3P3);
		\draw[crimsonglory, line width=1pt] (1T3) -- (3P7);
		
		\draw[crimsonglory, line width=1pt] (2T2) -- (3P5);
		\draw[crimsonglory, line width=1pt] (2T3) -- (3P6);
		
		\begin{scope}[on background layer]
		\fill[applegreen!70!white] (5.5,6) -- (5,5) -- (4.5,3) -- (5,2) -- (8.5,2) -- (8,3) -- (7.5,5) -- (7,6) -- cycle;
		\end{scope}

		\node () at (6,4) {$D$};
		\node () at (9,6.5) {$Q_{1}^{\rm up}$};
		\node () at (9,4.5) {$Q_{2}^{\rm up}$};
		\node () at (9.5,3.5) {$Q_{2}^{\rm down}$};
		\node () at (9.6,1.5) {$Q_{1}^{\rm down}$};
		
		\node () at (11.7,6.4) {$x_{1}^{\rm up}$};
		\node () at (10.5,2.5) {$x_{1}^{\rm down}$};
		
		\node () at (13.5,5.5) {$x_{2}^{\rm up}$};
		\node () at (15.7,3.5) {$x_{2}^{\rm down}$};
		
		\node () at (10.5,7) {$Z_{1}^{\rm up}$};
		\node () at (13.5,1) {$Z_{1}^{\rm down}$};
		\node () at (15.5,6.5) {$Z_{2}^{\rm up}$};
		\node () at (16,1) {$Z_{2}^{\rm down}$};
		
		\draw[<->] (4,6) to (4,5);
		\node[anchor=east] () at (4.2,5.5) {$b+r$};
		\draw[<->] (4,3) to (4,2);
		\node[anchor=east] () at (4.2,2.5) {$b+r$};
		
		\node[anchor=north east] () at (2,0) {$C_{1}$};
		\node[anchor=east] () at (2,0.5) {$C_{2}$};
		\node[anchor=east] () at (2,2) {$C_{b+r+1}$};
		\node[anchor=east] () at (2,3) {$C_{2(b+r)+1}$};
		\node[anchor=east] () at (2,5) {$C_{p-2(b+r)}$};
		\node[anchor=east] () at (2,6) {$C_{p-(b+r)}$};
		\node[anchor=east] () at (2,7.5) {$C_{p-1}$};
		\node[anchor=south east] () at (2,8) {$C_{p}$};
		\draw[<->] (4,0.5) to (4,2);
		\node[anchor=east] () at (4.2,1) {$b+r$};
		\draw[<->] (4,6) to (4,7.5);
		\node[anchor=east] () at (4.2,6.5) {$b+r$};
		
%		\node[anchor=north] () at (1,0) {$P_{1}$};
		\node[anchor=north] () at (6.5,0) {$P_{b+1}$};
		\node[anchor=north] () at (9,0) {$P_{q-b}$};
		\node[anchor=north] () at (18.5,0) {$P_{q}$};
		
		\node[anchor=north] () at (12,0) {$Z_{1}$};
		\node[anchor=north] () at (15,0) {$Z_{2}$};
		
		\end{tikzpicture}}}
		\caption{Visualization of an $(p,q)$-railed annulus and the notations introduced above.}
		\label{asadfdsfdsfsdafsdf}
	\end{figure}
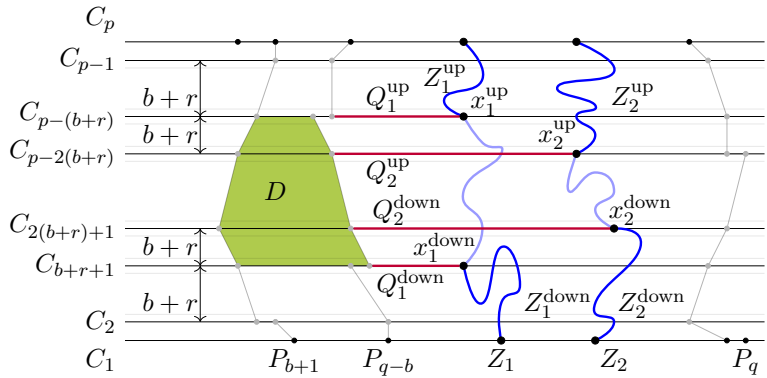

\noindent{\em Claim:} For every $i\in[d]$,
$V(Z^{\rm down}_{i-1})\cap N_G^{(\leq r)}(V(Q_i^{\rm down}))=\emptyset$
and
$V(Z^{\rm up}_{i-1})\cap N_G^{(\leq r)}(V(Q_i^{\rm up}))=\emptyset$
--- where $Z_{0}^{\rm down}$ (resp. $Z_{0}^{\rm up}$)
denotes $Z_{q}$.\smallskip

\noindent{\em Proof of  claim:} 
If $V(Z^{\rm down}_{i-1})\cap N_G^{(\leq r)}(V(Q_i^{\rm down}))\neq\emptyset$
for some $i\in[d]$, then there is a connected component $F$
of $Z_{i-1}^{\rm down}\cap \overline{D}_{(i-1)\cdot (b+r)+1}$ that is a path with endpoints in $V(C_{(i-1)\cdot (b+r)+1})$ and intersects $N_G^{(\leq r)}(V(Q_i^{\rm down}))$.
Observe that $F$ is a $({\cal C},D,L'')$-mountain of $G$ based on $C_{(i-1)\cdot (b+r)+1}$.
Since $V(F)\cap N_G^{(\leq r)}(V(Q_i^{\rm down}))\neq \emptyset$ and $V(Q_i^{\rm down})\subseteq V(C_{i\cdot (b+r)+1})$,
we have that $V(F)\cap V(C_{(i-1)\cdot (b+r)+b+1}))\neq\emptyset$ and therefore $F$ is of height $>b$, a contradiction to (a).
Thus,  for every $i\in[d]$,
$V(Z^{\rm down}_{i-1})\cap N_G^{(\leq r)}(V(Q_i^{\rm down}))=\emptyset$.
Similarly, suppose, towards a contradiction, that $V(Z^{\rm up}_{i-1})\cap N_G^{(\leq r)}(V(Q_i^{\rm up}))\neq\emptyset$
for some $i\in[d]$.
Then there is a connected component $F'$ of $Z_{i-1}^{\rm up}\cap (\Delta\setminus {D}_{p-(i-1)\cdot (b+r)})$ that is a path with endpoints in
$V(C_{p-(i-1)\cdot (b+r)})$ and intersects $N_G^{(\leq r)}(V(Q_i^{\rm up}))$.
Observe that $F'$ is a $({\cal C},D,L'')$-valley of $G$ based on $C_{p-(i-1)\cdot (b+r)}$.
Since $V(F')\cap N_G^{(\leq r)}(V(Q_i^{\rm up}))\neq \emptyset$ and $V(Q_i^{\rm up})\subseteq V(C_{p-i\cdot (b+r)})$,
we have that $V(F')\cap V(C_{p-((i-1)\cdot (b+r)-b)}))\neq\emptyset$ and therefore $F'$ is of depth $>b$, a contradiction to (i.).
Claim follows.\hfill$\diamond$
\medskip
	
Because of the above claim, it follows that the 
paths $Q_{i}^{\rm down}\cup Z_{i}^{\rm down}$ (resp. $Q_{i}^{\rm up}\cup Z_{i}^{\rm up}$),   $i\in[d]$ are $(Z_{i}\cap C_{1},P_{i\cdot (b+r)+1,q-b})$-paths (resp. $(Z_{i}\cap C_{p},P_{p-i\cdot (b+r),q-b})$-paths) in $G$ that do not intersect the open disk $\inter(D)$ and the graph
$\bigcup_{i\in [d]} Q_{i}^{\rm down}\cup Z_{i}^{\rm down}\cup Q_{i}^{\rm up}\cup Z_{i}^{\rm up}$ is an $r$-scattered linkage of $G$.

	Let $w=(m+1)\cdot (b+r)+2$ and $w'=p-(m+1)\cdot(b+r)-1$.
	For $i\in[d]$, we now define $c_i = b+(i-1)\cdot (r+1)+1$ and (see \autoref{aasdfdsfsfdfdsfsfsdfgfdhffjh}) 
	\begin{eqnarray*}
	Y_{i}^{\rm down}& =&\mbox{the}~(P_{q-b}, P_{c_i})\mbox{-path  in}~ L_{i\cdot (b+r)+1,q-b\rightarrow c_i}\cup P_{i\cdot (b+r)+1,c_i}\cup
	  R_{i\cdot (b+r)+1\rightarrow  w,c_i},\\
	Y_{i}^{\rm up} &=&\mbox{the}~(P_{q-b}, P_{c_i})\mbox{-path in}~ L_{p-i\cdot (b+r),q-b\rightarrow c_i}\cup P_{p-i\cdot (b+r),c_i}\cup R_{p-i\cdot (b+r)\rightarrow w',c_i} .
	\end{eqnarray*}

	\begin{figure}[H]
	\centering\scalebox{0.9}{
	\sshow{0}{\begin{tikzpicture}[scale=0.5]
	\foreach \y in {0,1,3,3.5, 5.5, 6, 8,9}{		
		\draw[-] (2,\y) to (20,\y);
	}
	
	\foreach \y in {0.2,0.6, 0.8, 1.2, 1.4, 2.6,2.8,3.7,3.9,5.1,5.3,6.2,6.4,7.6,7.8,8.2,8.4,8.8}{		
		\draw[-, opacity=0.1] (2,\y) to (20,\y);
	}
	
	%P1
	\node[track node 1] (1P1) at (3,9) {};
	\node[track node 2] (1P2) at (3.5,8) {};
	\node[track node 2] (1P3) at (4.5,8) {};
	\node[track node 2] (1P4) at (3,6) {};
	\node[track node 2] (1P5) at (3.5,5.5) {};
	\node[track node 2] (1P6) at (4,3.5) {};
	\node[track node 2] (1P7) at (3.5,3) {};
	\node[track node 2] (1P8) at (3,3) {};
	\node[track node 2] (1P9) at (3.5,1) {};
	\node[track node 1] (1P10) at (3,0) {};
	
	\foreach \i/\j in {1/2,2/3,3/4,4/5,5/6,6/7,7/8,8/9,9/10}{
		\draw[opacity=0.3] (1P\i) to (1P\j);
	}
	
	%P2
	\node[track node 1] (2P1) at (6,9) {};
	\node[track node 1] (2P2) at (7,9) {};
	\node[track node 2] (2P3) at (6,8) {};
	\node[track node 2] (2P4) at (5.5,6) {};
	\node[track node 2] (2P5) at (5,5.5) {};
	\node[track node 2] (2P6) at (5.5,3.5) {};
	\node[track node 2] (2P7) at (6,3.5) {};
	\node[track node 2] (2P8) at (5.5,3) {};
	\node[track node 2] (2P9) at (4.5,1) {};
	\node[track node 1] (2P10) at (5,0) {};
	
	\foreach \i/\j in {1/2,2/3,3/4,4/5,5/6,6/7,7/8,8/9,9/10}{
		\draw[opacity=0.3] (2P\i) to (2P\j);
	}
	
	%P3
	\node[track node 1] (3P1) at (14,9) {};
	\node[track node 2] (3P2) at (14.5,8) {};
	\node[track node 2] (3P3) at (14.5,6) {};
	\node[track node 2] (3P4) at (15,6) {};
	\node[track node 2] (3P5) at (14,5.5) {};
	\node[track node 2] (3P6) at (13,3.5) {};
	\node[track node 2] (3P7) at (13.5,3) {};
	\node[track node 2] (3P8) at (14.5,3) {};
	\node[track node 2] (3P9) at (14,1) {};
	\node[track node 1] (3P10) at (13.5,0) {};
	
	\foreach \i/\j in {1/2,2/3,3/4,4/5,5/6,6/7,7/8,8/9,9/10}{
		\draw[opacity=0.3] (3P\i) to (3P\j);
	}
	
	%P4
	\node[track node 1] (4P1) at (17,9) {};
	\node[track node 2] (4P2) at (17.5,8) {};
	\node[track node 2] (4P3) at (18,6) {};
	\node[track node 2] (4P4) at (18,5.5) {};
	\node[track node 2] (4P5) at (18.5,5.5) {};
	\node[track node 2] (4P6) at (18,3.5) {};
	\node[track node 2] (4P7) at (17.5,3) {};
	\node[track node 2] (4P8) at (17,1) {};
	\node[track node 1] (4P9) at (18,0) {};
	\node[track node 1] (4P10) at (18.5,0) {};
	
	\foreach \i/\j in {1/2,2/3,3/4,4/5,5/6,6/7,7/8,8/9,9/10}{
		\draw[opacity=0.3] (4P\i) to (4P\j);
	}
	
	\draw[crimsonglory, line width=1pt] (4P10)-- (3P10)--(2P10) -- (1P10) -- (1P9) -- (1P8) -- (1P7)-- (1P6);
	\draw[crimsonglory, line width=1pt] (4P8) -- (3P9)-- (2P9) -- (2P8) -- (2P7);
	\draw[crimsonglory, line width=1pt] (4P7) -- (3P8) -- (3P7) -- (3P6);

	\draw[crimsonglory, line width=1pt] (4P1) -- (3P1) -- (2P2)--(2P1) -- (1P1) -- (1P2) -- (1P3) -- (1P4) -- (1P5);
	\draw[crimsonglory, line width=1pt] (4P2) -- (3P2) -- (2P3) -- (2P4) -- (2P5);
	\draw[crimsonglory, line width=1pt] (4P3) -- (3P4) -- (3P5);
	
	\node[anchor=east] () at (2,0) {$C_{b+r+1}$};
	\node[anchor=east] () at (2,1) {$C_{2(b+r)+1}$};
	\node[anchor=east] () at (2,3) {$C_{d\cdot (b+r)+1}$};
	\node[anchor=east] () at (2,3.5) {$C_{w}$};
	\node[anchor=north east] () at (2,5.4) {$C_{w'}$};
	\node[anchor=east] () at (2,6.1) {$C_{p-d\cdot (b+r)}$};
	\node[anchor=east] () at (2,8) {$C_{p-2(b+r)}$};
	\node[anchor=east] () at (2,9) {$C_{p-(b+r)}$};

	\node[anchor=north] () at (1P10) {$P_{c_1}$};
	\node[anchor=north] () at (2P10) {$P_{c_2}$};
	\node[anchor=north] () at (3P10) {$P_{c_d}$};
	\node[anchor=north] () at (4P10) {$P_{q-b}$};
	
	\node[anchor=north] () at (10,0) {$Y_{1}^{\rm down}$};
	\node[anchor=south] () at (10,1) {$Y_{2}^{\rm down}$};
	
	\node[anchor=south] () at (10,9) {$Y_{1}^{\rm up}$};
	\node[anchor=north] () at (10,8) {$Y_{2}^{\rm up}$};
	
	\node[anchor=south] () at (16,6) {$Y_{d}^{\rm up}$};
	\node[anchor=north] () at (16,3) {$Y_{d}^{\rm down}$};
	\end{tikzpicture}}}
	\caption{Visualization of the definition of $Y_{i}^{\rm up}$ and $Y_{i}^{\rm down}$, $i\in[d]$.}
	\label{aasdfdsfsfdfdsfsfsdfgfdhffjh}
\end{figure}
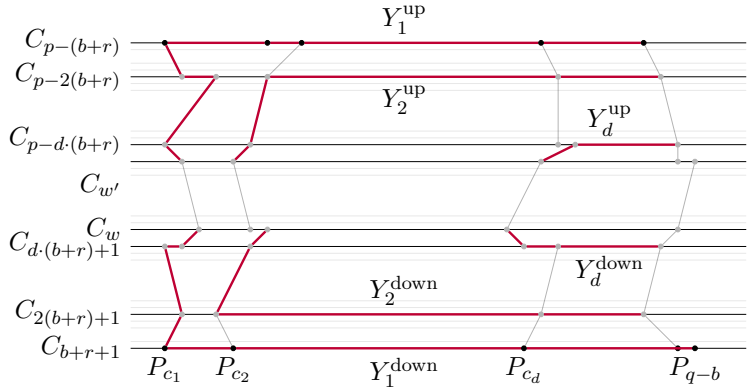
By the definition of  $Y_{i}^{\rm down} $ and $Y_{i}^{\rm up}$, the graphs $X_{i}^{\rm down}=Z_{i}^{\rm down}\cup Q_{i}^{\rm down}\cup  Y_{i}^{\rm down}$ and 
$X_{i}^{\rm up}=Z_{i}^{\rm up}\cup Q_{i}^{\rm up}\cup  Y_{i}^{\rm up},  i\in[d]$, are paths and $\bigcup_{i\in [d]} X_i^{\rm down}\cup X_i^{\rm up}$ is an $r$-scattered linkage.
In particular, we have that 
\begin{eqnarray}
{X}_{i}^{\rm down}  \mbox{ is a  } & \hspace{-3mm}(Z_{i}\cap C_{1},P_{w,c_i}){\rm\mbox{-}path }  &  \mbox{and}\label{ay7}\\
{X}_{i}^{\rm up}  \mbox{ is a  }  & \hspace{-3mm} (Z_{i}\cap C_{p},P_{w',c_i}){\rm\mbox{-}path} & \label{ay9}
\end{eqnarray}
Let $\Omega=\ann({\cal C},w,w')$ and $K'=\bigcup_{i\in [d]} X_i^{\rm down}\cup X_i^{\rm up}$.
Observe  that 
\begin{eqnarray}
K'\cap \Omega & = & \bigcup_{i\in[d]}(V(P_{w,c_i})\cup V(P_{w',c_i})). \label{acongts}
\end{eqnarray}

Let $\bar{\cal A}=(\bar{\cal C},\bar{\cal P})$, where $\bar{\cal C}=[C_{w},\ldots,C_{w'}]$ and $\bar{\cal P}={\cal P}\cap \Omega$. 
Notice that $|\bar{\cal C}|=w'-w+1= p - 2(m+1)\cdot (b+r)-2\geq s+2b$. 
Notice also that $d\leq |I|$ and $I\subseteq[q]$. Finally, 
$b=3/2m$ and $d\leq m$ imply that  $b+d\cdot (r+1)\leq \frac{2r+5}{2}\cdot m\leq q$.
We can now apply \autoref{alo3qx} for $p,q,s,b,d$, $\bar{\cal A}$, and $I$
and obtain a linkage $K$ of $\bar{\cal A}$ satisfying properties (a) and (b) of  \autoref{alo3qx}.
	
From Property (a)  we can write ${\cal P}(K)=[K_{1},\ldots,K_{d}]$ and, using~\eqref{acongts}, we deduce that,
for $i\in[d]$, $K_{i}$ is a $(P_{w,c_i},P_{w',c_i})$-path of $G$. This, together with~\eqref{ay7},~\eqref{ay9}, and~\eqref{acongts}, implies that 
$K\cup K'$ is a linkage of $G$ where $K\cup K'\subseteq\ann({\cal C})$. 
From Property (b), $K$ is $(s,I)$-confined in $\bar{\cal A}$, therefore, from~\eqref{acongts}, we get that $K\cup K'$
is  $(s,I)$-confined in ${\cal A}$. Observe also that each of the $d$ paths of ${\cal P}(K\cup K')$ is a $(Z_{i}\cap C_{1},Z_{i}\cap C_{p})$-path of $G$ 
for some $i\in[d]$. We define
$$\tilde{L}=(L\setminus A')\cup K\cup K'$$
where $A'={\sf int}(\ann({\cal C}))$. By definition $\tilde{L}$ is an $r$-scattered linkage of $G$ where  $\tilde{L}\equiv L$ and $\tilde{L}\setminus \ann({\cal C}) \subseteq L \setminus \ann({\cal C})$.
Finally, as $K\cup K'$ is $(s,I)$-confined in ${\cal A}$, then $\tilde{L}$ $(s,I)$-confined in ${\cal A}$ as well.
\end{proof}

\section{Linkage reducibility of surface embeddable graphs}
\label{sec_irrfriendlybdgenus}
% of \autoref{amultheah}}

In this section, our goal is to prove~\autoref{amultheah}.
Our approach is based on the technique developed by Kawarabayashi and Kobayashi for solving the induced paths problem on planar graphs~\cite{KawarabayashiK12alin}.
We generalize this technique to two directions.
The first (and rather straightforward) task is to adapt their results for general $r$-scattered linkages (keep in mind that induced linkages are $1$-scattered linkages).
The second, and more intricate, task is to lift their technique from graphs embedded on the plane to graphs embedded to a surface of fixed genus.
The main difference is that the number of homotopic loops in a graph embedded in a surface is a linear function of the genus of the surface (see~\autoref{jfopdsn}).
Making use of this, we can enhance the arguments of~\cite{KawarabayashiK12alin} to deal with the presence of crosscups and handles outside a fixed disk of the surface.
To obtain a single-exponential dependency on $k$, we use the result of~\cite{Mazoit13asin} (see~\autoref{maz}) to route linkages in graphs embedded in surfaces. To ease readability, in~\autoref{subsec_intermediate} we start with some defintions, we state an intermediate result (\autoref{reducedlinkage}), and we show how this result implies~\autoref{amultheah}.
Then, in~\autoref{subsec_proofofreducedlinkage}, we present the proof of~\autoref{reducedlinkage}.

\subsection{An intermediate step}\label{subsec_intermediate}

To prove~\autoref{amultheah}, we will show~\autoref{maintheorem}, that intuitively states that given a graph $G$ embedded on a surface and some ``large enough'' collection ${\cal C}$ of nested cycles of $G$ embedded in a disk,
every (scattered) linkage of $G$ can be rerouted ``away'' from vertices inside the disk bounded by the innermost cycle of ${\cal C}$.
We start with some additional definitions.

\paragraph{Surfaces.}
A {\em surface} is a compact connected $2$-manifold without boundary.
It is known
(see e.g.,~\cite{MoharT01grap}) that any surface $Σ$ can be obtained, up to homeomorphism, by adding
${\bf eg}(Σ)$ crosscaps to the sphere, where ${\bf eg}(Σ)$ is called the Euler genus of $Σ$.
%A cycle $C$ of $\Sigma$ is {\em contractible} if at least one of the connected components of $\Sigma\setminus C$ is an open disk of $\Sigma$.

\paragraph{Isolated vertices.}% Let $t\in \mathbb{N}_{\geq 2}$.
%Also, let $\Sigma$ be a surface, $G$ be a graph embedded in $\Sigma$,
%and ${\cal C}=[C_{1},\ldots, C_{t}]$ be a collection of vertex disjoint cycles of $G$.
%We say that the sequence ${\cal C}$ is a {\em nested sequence of cycles} of $G$
%if every $C_{i}$ is the boundary of an open disk $D_{i}$ of $\Sigma$ such that $D_{t}\supseteq \cdots \supseteq D_{1}$. We call $D_{t}$ the {\em outer disk} of ${\cal C}$.
%Given an open disk $\Delta$ of $\Sigma$, we say that the sequence ${\cal C}$ is a {\em $\Delta$-nested sequence of cycles} of $G$
%if its outer disk is $\Delta$.
Let $\Delta$ be a closed disk, let $G$ be a graph, and let ${\cal C}$ be a $\Delta$-nested sequence of cycles of $G$.
Given a vertex set $S\subseteq V(G)$, we say that ${\cal C}$ {\em isolates} $S$ if $S\subseteq \inter(D_{1})$.
Also, given an $\ell\in \mathbb{N}$, an open disk $\Delta$ of $\Sigma$, and a set $S\subseteq V(G)$, we say that $S$ is {\em $(\ell,\Delta)$-isolated} in $G$ if there is a $\Delta$-nested sequence of cycles of $G$ of size $\ell$ that isolates $S$.

Let $t\in\mathbb{N}_{\geq 2}.$
Given a vertex $v\in V(G)$ and an $r\in\mathbb{N}$,
we say that a $\Delta$-nested sequence of cycles ${\cal C}=[C_{1},\ldots, C_{t}]$ of $G$ is {\em $r$-tight around $v$}
if ${\cal C}$ isolates $\{v\}$ and for every $i\in[t]$,
there is no cycle $C_{i}'\neq C_{i}$ of $G$ contained in $D_{i}\setminus D_{i-1}$ such that
$N_G^{(\leq r)}(V(C_{i-1}))\cap V(C_i')=\emptyset$, where $C_{0}=\{v\}$.

\begin{figure}[ht]
\centering
\scalebox{0.9}{
\begin{tikzpicture}[ipe stylesheet]
  \draw[draw=darkgray]
    (280, 604)
     -- (292, 608);
  \draw[draw=darkgray]
    (280, 604)
     -- (284, 596);
  \draw[draw=darkgray]
    (272, 608)
     -- (280, 604);
  \draw[draw=darkgray]
    (228, 616)
     -- (240, 616);
  \draw[draw=darkgray]
    (228, 616)
     -- (240, 616);
  \draw[draw=darkgray]
    (224, 604)
     -- (240, 616);
  \draw[ipe pen fat]
    (184, 656)
     -- (172, 620);
  \draw[darkgray]
    (172, 620)
     -- (192, 588);
  \draw[darkgray]
    (192, 588)
     -- (232, 592);
  \draw[darkgray]
    (232, 592)
     -- (260, 624);
  \draw[darkgray]
    (260, 624)
     -- (244, 660);
  \draw[ipe pen fat]
    (244, 660)
     -- (220, 672);
  \draw[ipe pen fat]
    (220, 672)
     -- (184, 656);
  \draw[darkgray]
    (172, 620)
     -- (168, 644);
  \draw[darkgray]
    (168, 644)
     -- (184, 656);
  \draw[darkgray]
    (184, 656)
     -- (220, 680);
  \draw[darkgray]
    (220, 680)
     -- (220, 672);
  \draw[darkgray]
    (220, 680)
     -- (244, 660);
  \draw[darkgray]
    (172, 620)
     -- (164, 604);
  \draw[darkgray]
    (232, 592)
     -- (256, 604);
  \draw[darkgray]
    (260, 624)
     -- (264, 616);
  \draw[darkgray]
    (244, 660)
     -- (260, 652);
  \draw[darkgray]
    (260, 652)
     -- (268, 652);
  \draw[darkgray]
    (260, 624)
     -- (264, 640);
  \draw[darkgray]
    (264, 640)
     -- (268, 644);
  \draw[ipe pen fat]
    (268, 652)
     -- (256, 676);
  \draw[darkgray]
    (244, 660)
     -- (244, 668);
  \draw[darkgray]
    (244, 668)
     -- (256, 676);
  \draw[darkgray]
    (220, 680)
     -- (216, 684);
  \draw[ipe pen fat]
    (216, 684)
     -- (188, 676);
  \draw[ipe pen fat]
    (188, 676)
     -- (176, 668);
  \draw[darkgray]
    (188, 676)
     -- (188, 668);
  \draw[darkgray]
    (188, 668)
     -- (184, 656);
  \draw[darkgray]
    (184, 656)
     -- (176, 660);
  \draw[darkgray]
    (176, 660)
     -- (176, 668);
  \draw[darkgray]
    (176, 660)
     -- (168, 656);
  \draw[ipe pen fat]
    (168, 656)
     -- (164, 648);
  \draw[darkgray]
    (164, 648)
     -- (168, 644);
  \draw[ipe pen fat]
    (168, 656)
     -- (176, 668);
  \draw[darkgray]
    (164, 648)
     -- (164, 620);
  \draw[darkgray]
    (164, 620)
     -- (164, 604);
  \draw[ipe pen fat]
    (164, 648)
     -- (156, 616);
  \draw[darkgray]
    (156, 616)
     -- (164, 604);
  \draw[ipe pen fat]
    (156, 616)
     -- (164, 584);
  \draw[ipe pen fat]
    (164, 584)
     -- (184, 576);
  \draw[darkgray]
    (164, 584)
     -- (164, 604);
  \draw[darkgray]
    (164, 584)
     -- (180, 588);
  \draw[darkgray]
    (180, 588)
     -- (192, 588);
  \draw[darkgray]
    (180, 588)
     -- (172, 620);
  \draw[ipe pen fat]
    (184, 576)
     -- (216, 576);
  \draw[ipe pen fat]
    (172, 620)
     -- (232, 592);
  \draw[darkgray]
    (184, 576)
     -- (192, 588);
  \draw[darkgray]
    (216, 576)
     -- (240, 584);
  \draw[ipe pen fat]
    (240, 584)
     -- (260, 596);
  \draw[darkgray]
    (256, 604)
     -- (260, 596);
  \draw[darkgray]
    (256, 604)
     -- (264, 616);
  \draw[ipe pen fat]
    (260, 596)
     -- (272, 608);
  \draw[ipe pen fat]
    (216, 576)
     -- (224, 584);
  \draw[ipe pen fat]
    (224, 584)
     -- (240, 584);
  \draw[ipe pen fat]
    (272, 608)
     -- (264, 616);
  \draw[ipe pen fat]
    (264, 616)
     -- (268, 644);
  \draw[darkgray]
    (272, 608)
     -- (268, 644);
  \draw[darkgray]
    (260, 652)
     -- (264, 640);
  \draw[ipe pen fat]
    (268, 644)
     -- (268, 652);
  \draw[ipe pen fat]
    (256, 676)
     -- (216, 684);
  \draw[darkgray]
    (220, 680)
     -- (244, 668);
  \draw[darkgray]
    (164, 584)
     -- (152, 584);
  \draw[darkgray]
    (152, 584)
     -- (148, 592);
  \draw[darkgray]
    (152, 584)
     -- (156, 616);
  \draw[darkgray]
    (184, 576)
     -- (196, 568);
  \draw[darkgray]
    (196, 568)
     -- (188, 564);
  \draw[darkgray]
    (196, 568)
     -- (232, 568);
  \draw[darkgray]
    (232, 568)
     -- (240, 584);
  \draw[darkgray]
    (232, 568)
     -- (244, 568);
  \draw[ipe pen fat]
    (188, 564)
     -- (232, 560);
  \draw[ipe pen fat]
    (232, 560)
     -- (244, 568);
  \draw[darkgray]
    (244, 568)
     -- (268, 592);
  \draw[darkgray]
    (268, 592)
     -- (240, 584)
     -- (240, 584);
  \draw[ipe pen fat]
    (244, 568)
     -- (284, 596);
  \draw[ipe pen fat]
    (292, 608)
     -- (288, 648);
  \draw[ipe pen fat]
    (288, 648)
     -- (272, 688);
  \draw[ipe pen fat]
    (272, 688)
     -- (220, 700);
  \draw[ipe pen fat]
    (220, 700)
     -- (156, 692);
  \draw[ipe pen fat]
    (156, 692)
     -- (148, 636);
  \draw[darkgray]
    (148, 636)
     -- (136, 588);
  \draw[ipe pen fat]
    (136, 588)
     -- (172, 560);
  \draw[ipe pen fat]
    (172, 560)
     -- (188, 564);
  \draw[darkgray]
    (188, 564)
     -- (168, 572);
  \draw[darkgray]
    (168, 572)
     -- (152, 584);
  \draw[ipe pen fat]
    (148, 592)
     -- (136, 588);
  \draw[ipe pen fat]
    (148, 592)
     -- (148, 636);
  \draw[darkgray]
    (164, 648)
     -- (160, 664);
  \draw[darkgray]
    (160, 664)
     -- (156, 692);
  \draw[darkgray]
    (156, 692)
     -- (168, 676);
  \draw[darkgray]
    (168, 676)
     -- (168, 656);
  \draw[darkgray]
    (168, 676)
     -- (188, 676);
  \draw[darkgray]
    (188, 676)
     -- (204, 692);
  \draw[darkgray]
    (204, 692)
     -- (216, 684);
  \draw[darkgray]
    (216, 684)
     -- (232, 688);
  \draw[darkgray]
    (232, 688)
     -- (220, 700);
  \draw[darkgray]
    (204, 692)
     -- (220, 700);
  \draw[darkgray]
    (232, 688)
     -- (204, 692);
  \draw[darkgray]
    (232, 688)
     -- (256, 676);
  \draw[darkgray]
    (168, 676)
     -- (204, 692);
  \draw[darkgray]
    (268, 592)
     -- (284, 596);
  \draw[darkgray]
    (268, 644)
     -- (276, 644);
  \draw[darkgray]
    (276, 644)
     -- (288, 648);
  \draw[darkgray]
    (256, 676)
     -- (268, 672);
  \draw[darkgray]
    (268, 672)
     -- (276, 644);
  \draw[darkgray]
    (268, 672)
     -- (272, 688);
  \draw[darkgray]
    (184, 656)
     -- (188, 636);
  \draw[darkgray]
    (188, 636)
     -- (188, 624);
  \draw[ipe pen fat]
    (188, 624)
     -- (212, 612);
  \draw[ipe pen fat]
    (212, 612)
     -- (220, 616);
  \draw[darkgray]
    (220, 616)
     -- (224, 604);
  \draw[darkgray]
    (224, 604)
     -- (232, 592);
  \draw[darkgray]
    (220, 616)
     -- (228, 616);
  \draw[darkgray]
    (240, 632)
     -- (244, 660)
     -- (244, 660);
  \draw[darkgray]
    (240, 632)
     -- (232, 656);
  \draw[darkgray]
    (232, 656)
     -- (220, 672);
  \draw[ipe pen fat]
    (220, 616)
     -- (224, 656);
  \draw[darkgray]
    (224, 656)
     -- (232, 656);
  \draw[ipe pen fat]
    (224, 656)
     -- (212, 664);
  \draw[ipe pen fat]
    (212, 664)
     -- (200, 652);
  \draw[ipe pen fat]
    (200, 652)
     -- (188, 624);
  \draw[darkgray]
    (188, 636)
     -- (200, 652);
  \draw[darkgray]
    (172, 620)
     -- (180, 628);
  \draw[darkgray]
    (180, 628)
     -- (188, 624);
  \draw[darkgray]
    (172, 620)
     -- (188, 616);
  \draw[darkgray]
    (188, 616)
     -- (188, 624);
  \draw[darkgray]
    (200, 652)
     -- (208, 640);
  \draw[darkgray]
    (208, 640)
     -- (212, 632);
  \draw[darkgray]
    (212, 632)
     -- (208, 628);
  \draw[darkgray]
    (208, 628)
     -- (188, 624);
  \draw[darkgray]
    (208, 628)
     -- (220, 616);
  \draw[darkgray]
    (208, 640)
     -- (224, 656);
  \pic[red]
     at (212, 632) {ipe disk};
  \pic
     at (136, 588) {ipe disk};
  \pic
     at (148, 592) {ipe disk};
  \pic
     at (152, 584) {ipe disk};
  \pic
     at (156, 616) {ipe disk};
  \pic
     at (164, 604) {ipe disk};
  \pic
     at (164, 584) {ipe disk};
  \pic
     at (148, 636) {ipe disk};
  \pic
     at (156, 692) {ipe disk};
  \pic
     at (168, 676) {ipe disk};
  \pic
     at (168, 656) {ipe disk};
  \pic
     at (176, 668) {ipe disk};
  \pic
     at (176, 660) {ipe disk};
  \pic
     at (164, 648) {ipe disk};
  \pic
     at (168, 644) {ipe disk};
  \pic
     at (184, 656) {ipe disk};
  \pic
     at (188, 676) {ipe disk};
  \pic
     at (188, 668) {ipe disk};
  \pic
     at (204, 692) {ipe disk};
  \pic
     at (216, 684) {ipe disk};
  \pic
     at (232, 688) {ipe disk};
  \pic
     at (220, 700) {ipe disk};
  \pic
     at (272, 688) {ipe disk};
  \pic
     at (268, 672) {ipe disk};
  \pic
     at (256, 676) {ipe disk};
  \pic
     at (244, 668) {ipe disk};
  \pic
     at (220, 680) {ipe disk};
  \pic
     at (220, 672) {ipe disk};
  \pic
     at (212, 664) {ipe disk};
  \pic
     at (200, 652) {ipe disk};
  \pic
     at (188, 636) {ipe disk};
  \pic
     at (188, 624) {ipe disk};
  \pic
     at (180, 628) {ipe disk};
  \pic
     at (188, 616) {ipe disk};
  \pic
     at (212, 612) {ipe disk};
  \pic
     at (220, 616) {ipe disk};
  \pic
     at (224, 656) {ipe disk};
  \pic
     at (232, 656) {ipe disk};
  \pic
     at (228, 616) {ipe disk};
  \pic
     at (244, 660) {ipe disk};
  \pic
     at (260, 652) {ipe disk};
  \pic
     at (264, 640) {ipe disk};
  \pic
     at (268, 644) {ipe disk};
  \pic
     at (268, 652) {ipe disk};
  \pic
     at (276, 644) {ipe disk};
  \pic
     at (288, 648) {ipe disk};
  \pic
     at (272, 608) {ipe disk};
  \pic
     at (284, 596) {ipe disk};
  \pic
     at (268, 592) {ipe disk};
  \pic
     at (260, 596) {ipe disk};
  \pic
     at (256, 604) {ipe disk};
  \pic
     at (264, 616) {ipe disk};
  \pic
     at (260, 624) {ipe disk};
  \pic
     at (232, 592) {ipe disk};
  \pic
     at (224, 604) {ipe disk};
  \pic
     at (192, 588) {ipe disk};
  \pic
     at (180, 588) {ipe disk};
  \pic
     at (172, 620) {ipe disk};
  \pic
     at (168, 572) {ipe disk};
  \pic
     at (172, 560) {ipe disk};
  \pic
     at (188, 564) {ipe disk};
  \pic
     at (196, 568) {ipe disk};
  \pic
     at (184, 576) {ipe disk};
  \pic
     at (216, 576) {ipe disk};
  \pic
     at (224, 584) {ipe disk};
  \pic
     at (240, 584) {ipe disk};
  \pic
     at (232, 568) {ipe disk};
  \pic
     at (244, 568) {ipe disk};
  \pic
     at (232, 560) {ipe disk};
  \pic
     at (208, 628) {ipe disk};
  \pic
     at (208, 640) {ipe disk};
  \pic
     at (160, 664) {ipe disk};
  \pic
     at (164, 620) {ipe disk};
  \node[ipe node]
     at (212, 636) {$v$};
  \node[ipe node]
     at (292, 636) {$C_4$};
  \node[ipe node]
     at (224, 624) {$C_1$};
  \node[ipe node]
     at (244, 628) {$C_2$};
  \node[ipe node]
     at (272, 632) {$C_3$};
  \filldraw[ipe pen fat]
    (280, 628)
     arc[start angle=90, end angle=90, radius=4];
  \pic
     at (292, 608) {ipe disk};
  \pic
     at (240, 632) {ipe disk};
  \pic
     at (240, 616) {ipe disk};
  \filldraw[ipe pen fat]
    (240, 632)
     -- (240, 616)
     -- cycle;
  \filldraw[ipe pen fat]
    (240, 616)
     -- (232, 592)
     -- cycle;
  \filldraw[ipe pen fat]
    (240, 632)
     -- (244, 660)
     -- cycle;
  \pic
     at (280, 604) {ipe disk};
  \filldraw[ipe pen fat]
    (284, 596)
     -- (292, 608);
\end{tikzpicture}}
\caption{A graph $G$, a $\Delta$-nested sequence of cycles ${\cal C}=[C_1,\ldots, C_4]$ that is $1$-tight around the vertex $v$ (depicted in red).}
\end{figure}
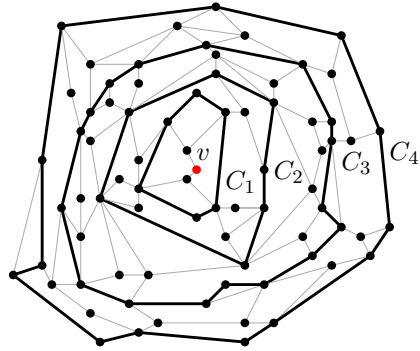

\paragraph{Bridges.}
Let $\Sigma$ be a surface and $\Delta$ be an open disk of $\Sigma$.
Let $G$ be a graph embedded in $\Sigma$ and $L$ be a $\Delta$-avoiding linkage of $G$.
We call a  connected component $B$ of $L\setminus\Delta$ a {\em $\Delta$-bridge} of $L$
if $B\cap T(L)=\emptyset$ and $B\subsetneq \bd(\Delta)$ (see~\autoref{fig_bridgescrossings} for an illustration).
Observe that every $\Delta$-bridge of $L$ is a subpath of a path of $L$.
The {\em endpoints} of a $\Delta$-bridge $B$ of $L$ are its (two) vertices that are incident to exactly one edge in $B$.
We denote by ${\cal B}_{\Delta}(L)$ the set of all $\Delta$-bridges of $L$ and use ${\sf bridges}_{\Delta}(L)$ to denote $|{\cal B}_{\Delta}(L)|$.

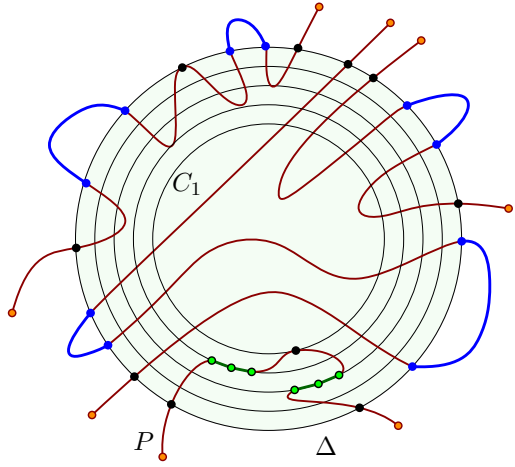
\begin{figure}[ht]
\centering
\scalebox{0.9}{
\begin{tikzpicture}[ipe stylesheet]
  \filldraw[fill=lightgreen, ipe opacity 10]
    (256, 640) circle[radius=80];
   \draw
    (256, 640) circle[radius=80];
  \draw
    (256, 640) circle[radius=64];
  \draw
    (256, 640) circle[radius=72];
  \draw[darkred, ipe pen heavier]
    (266.804, 576.918)
     .. controls (253.14, 568.057) and (291.422, 577.316) .. (309.756, 561.965);
  \draw[darkred, ipe pen heavier]
    (212.142, 548.942)
     .. controls (210.278, 567.132) and (223.023, 589.402) .. (232.473, 589.182);
  \draw[darkred, ipe pen heavier]
    (149.741, 609.038)
     .. controls (156.636, 634.926) and (166.437, 634.383) .. (176.295, 636.193)
     .. controls (186.153, 638.003) and (196.068, 642.165) .. (196.622, 647.126)
     .. controls (197.176, 652.086) and (188.369, 657.845) .. (180.41, 663.252);
  \draw[blue, ipe pen fat]
    (180.41, 663.252)
     .. controls (172.451, 668.658) and (165.34, 673.712) .. (165.195, 679.093)
     .. controls (165.049, 684.473) and (171.869, 690.18) .. (177.656, 693.298)
     .. controls (183.443, 696.415) and (188.197, 696.943) .. (196.54, 693.521);
  \draw[darkred, ipe pen heavier]
    (196.481, 693.455)
     .. controls (197.507, 693.125) and (202.062, 688.779) .. (206.512, 684.511)
     .. controls (210.963, 680.243) and (215.309, 676.053) .. (216.836, 680.172)
     .. controls (218.362, 684.292) and (217.068, 696.722) .. (217.194, 703.353)
     .. controls (217.32, 709.983) and (218.864, 710.815) .. (220.26, 711.521)
     .. controls (221.655, 712.227) and (222.902, 712.806) .. (226.363, 709.77)
     .. controls (229.824, 706.733) and (235.5, 700.081) .. (239.972, 697.285)
     .. controls (244.444, 694.489) and (247.712, 695.55) .. (246.804, 700.897)
     .. controls (245.897, 706.244) and (240.815, 715.876) .. (240.013, 718.386);
  \draw[blue, ipe pen fat]
    (240.013, 718.386)
     .. controls (237.582, 728.367) and (239.432, 731.225) .. (242.852, 731.482)
     .. controls (246.273, 731.738) and (251.265, 729.393) .. (254.758, 720.459);
  \draw[darkred, ipe pen heavier]
    (254.758, 720.459)
     .. controls (255.535, 721.171) and (254.814, 715.294) .. (255.17, 710.353)
     .. controls (255.525, 705.411) and (256.957, 701.404) .. (259.996, 704.837)
     .. controls (263.034, 708.27) and (267.679, 719.143) .. (276.984, 736.885);
  \draw[darkred, ipe pen heavier]
    (319.274, 722.825)
     .. controls (266.285, 682.749) and (261.433, 667.986) .. (260.405, 661.218)
     .. controls (259.378, 654.45) and (262.174, 655.678) .. (270.164, 661.706)
     .. controls (278.153, 667.734) and (291.336, 678.562) .. (299.224, 685.217)
     .. controls (307.111, 691.873) and (309.704, 694.357) .. (314.004, 695.891);
  \draw[blue, ipe pen fat]
    (314.004, 695.891)
     .. controls (322.887, 699.28) and (333.479, 701.72) .. (337.275, 699.017)
     .. controls (341.071, 696.313) and (338.071, 688.467) .. (325.629, 679.393);
  \draw[darkred, ipe pen heavier]
    (325.629, 679.393)
     .. controls (318.4257, 675.993) and (312.2063, 672.62) .. (306.971, 669.274)
     .. controls (299.119, 664.255) and (293.538, 657.382) .. (293.321, 653.34)
     .. controls (293.105, 649.298) and (298.254, 648.087) .. (306.391, 650.174)
     .. controls (314.527, 652.262) and (325.652, 657.647) .. (355.481, 652.663);
  \draw[darkred, ipe pen heavier]
    (306.512, 730.209)
     .. controls (223.537, 648.357) and (192.746, 620.823) .. (182.237, 609.032);
  \draw[blue, ipe pen fat]
    (182.237, 609.032)
     .. controls (172.034, 596.815) and (169.442, 588.651) .. (179.22, 591.496)
     .. controls (183.2213, 592.5967) and (186.6363, 593.958) .. (189.465, 595.58);
  \draw[darkred, ipe pen heavier]
    (189.465, 595.58)
     .. controls (215.141, 617.183) and (220.986, 627.589) .. (229.844, 633.726)
     .. controls (238.703, 639.863) and (250.575, 641.731) .. (260.815, 637.321)
     .. controls (271.055, 632.91) and (279.663, 622.219) .. (293.379, 623.419)
     .. controls (307.094, 624.618) and (325.916, 637.706) .. (336.046, 638.979);
  \draw[blue, ipe pen fat]
    (336.046, 638.979)
     .. controls (346.175, 640.251) and (347.612, 629.708) .. (347.691, 621.009)
     .. controls (347.771, 612.31) and (346.493, 605.456) .. (343.333, 598.806)
     .. controls (340.172, 592.156) and (335.129, 585.71) .. (315.618, 586.655);
  \draw[darkred, ipe pen heavier]
    (315.618, 586.655)
     .. controls (311.3413, 588.5963) and (305.8457, 591.7897) .. (299.131, 596.235)
     .. controls (289.058, 602.904) and (279.355, 611.089) .. (267.447, 615.746)
     .. controls (255.54, 620.403) and (241.429, 621.533) .. (182.846, 566.434);
  \pic[draw=darkred, fill=darkorange]
     at (149.741, 609.038) {ipe fdisk};
  \pic[draw=darkred, fill=darkorange]
     at (276.984, 736.885) {ipe fdisk};
  \pic[draw=darkred, fill=darkorange]
     at (306.512, 730.209) {ipe fdisk};
  \pic[draw=darkred, fill=darkorange]
     at (319.274, 722.825) {ipe fdisk};
  \pic[draw=darkred, fill=darkorange]
     at (182.846, 566.434) {ipe fdisk};
  \pic[draw=darkred, fill=darkorange]
     at (212.142, 548.942) {ipe fdisk};
  \pic[draw=darkred, fill=darkorange]
     at (309.756, 561.965) {ipe fdisk};
  \pic[draw=darkred, fill=darkorange]
     at (355.481, 652.663) {ipe fdisk};
  \pic[blue]
     at (180.41, 663.252) {ipe disk};
  \pic[blue]
     at (196.54, 693.521) {ipe disk};
  \pic[blue]
     at (240.013, 718.386) {ipe disk};
  \pic
     at (220.26, 711.521) {ipe disk};
  \pic
     at (176.295, 636.193) {ipe disk};
  \pic[blue]
     at (254.758, 720.459) {ipe disk};
  \pic
     at (268.214, 719.638) {ipe disk};
  \pic
     at (288.93, 712.908) {ipe disk};
  \pic
     at (299.428, 707.186) {ipe disk};
  \pic[blue]
     at (313.442, 695.681) {ipe disk};
  \pic[blue]
     at (325.629, 679.393) {ipe disk};
  \pic
     at (334.636, 654.709) {ipe disk};
  \pic[blue]
     at (336.046, 638.979) {ipe disk};
  \pic[blue]
     at (315.618, 586.655) {ipe disk};
  \pic
     at (200.415, 582.465) {ipe disk};
  \pic[blue]
     at (189.465, 595.58) {ipe disk};
  \pic[blue]
     at (182.237, 609.032) {ipe disk};
  \node[ipe node]
     at (275.382, 549.91) {$\Delta$};
  \draw
    (256, 640) circle[radius=56];
  \draw
    (256, 640) circle[radius=48];
  \draw[darkgreen, ipe pen fat]
    (232.473, 589.182)
     -- (240.554, 586.172)
     -- (248.992, 584.44);
  \draw[darkred, ipe pen heavier]
    (248.992, 584.44)
     .. controls (261.95, 584.317) and (258.746, 592.079) .. (267.325, 593.355)
     .. controls (282.101, 595.123) and (291.092, 586.478) .. (285.22, 583.104);
  \draw[darkgreen, ipe pen fat]
    (285.22, 583.104)
     -- (276.673, 579.431)
     -- (266.804, 576.918);
  \pic
     at (215.714, 570.884) {ipe disk};
  \pic
     at (293.72, 569.451) {ipe disk};
  \pic
     at (267.325, 593.355) {ipe disk};
  \pic[fill=green]
     at (232.473, 589.182) {ipe fdisk};
  \pic[fill=green]
     at (240.554, 586.172) {ipe fdisk};
  \pic[fill=green]
     at (248.992, 584.44) {ipe fdisk};
  \pic[fill=green]
     at (266.804, 576.918) {ipe fdisk};
  \pic[fill=green]
     at (276.673, 579.431) {ipe fdisk};
  \pic[fill=green]
     at (285.22, 583.104) {ipe fdisk};
  \node[ipe node]
     at (216, 660) {$C_1$};
  \node[ipe node]
     at (200, 552) {$P$};
\end{tikzpicture}}
\caption{A closed disk $\Delta$ of a surface $\Sigma$, a $\Delta$-nested sequence of cycles $\mathcal{C}=[C_1,\ldots, C_5]$,
and a $\inter(\Delta)$-avoiding linkage $L$. The $\inter(\Delta)$-bridges of $L$ are depicted as blue segments of the linkage $L$ while the crossings of the path $P\in \mathcal{P}(L)$ with $C_2$ and $C_3$ are depicted in green. Note that $P$ does not cross $C_1$.}
\label{fig_bridgescrossings}
\end{figure}

 \paragraph{Crossings.}
Let ${\cal C}=[C_{1},\ldots,C_{t}], t\geq 2$ be a $\Delta$-nested sequence of cycles  of $G$.
For $i\in[t]$, we say that a path $P=[v_{0}e_{1}v_{1}\cdots v_{\ell}]$ of $L$ {\em crosses} $C_{i}$
if there exist integers $q,r$ with $0<q\leq r<\ell$ such that the subpath $P'=[v_{q}\ldots v_{r}]$ of $P$
is contained in $C_{i}$, $e_{q}$ and $e_{r+1}$ are not in $C_{i}$, and exactly one of $e_{q}$ and $e_{r+1}$ is in $D_{i}$.
In this case, we say that $P$ crosses $C_{i}$ at $P'$.
See~\autoref{fig_bridgescrossings} for an illustration.
We also define ${\sf cross}_{\cal C}(L)$ to be the total number of crossings of $L$ with ${\cal C}$.
More formally,
$${\sf cross}_{\cal C}(L)=|\{P'\mid \exists P\in {\cal P}(L) \mbox{ and }\exists C\in{\cal C}\mbox{ such that } P \mbox{ crosses $C$ at }P'\}|.$$

\paragraph{BC-minimal linkages.}
Let $G$ be a graph embedded on $\Sigma$,  let $\Delta$ be an open disk of $\Sigma$,
let $v$ be a vertex in $\Delta\cap V(G)$, and
let ${\cal C}$ be a $\Delta$-nested sequence of cycles of $G$ that is $r$-tight around $v$.
We say that a $\Delta$-avoiding $r$-scattered linkage $L$ of $G$ is \emph{BC-minimal around $v$}, if
$v\in V(L)$ and for every $\Delta$-avoiding $r$-scattered linkage $L'$ of $G$ such that $v\in V(L')$ and $L\equiv L'$, it holds that ${\sf cross}_{\cal C}(L)\leq {\sf cross}_{\cal C}(L')$ and ${\sf bridges}_{\Delta}(L)\leq {\sf bridges}_{\Delta}(L')$.\medskip

We now state the following results that intuitively says that given a graph $G$ that is embedded on a fixed surface, a ``big enough'' nested sequence ${\cal C}$ of cycles of $G$, and an $r$-scattered linkage $L$ that contains a vertex $v$ that is isolated from ${\cal C}$ and has minimal number of bridges and crossings, there is an $r$-scattered linkage $L'$ of $G\setminus v$ that is equivalent to $L$.

\begin{lemma}\label{reducedlinkage}
There exists a function $\newfun{reducedfun}:\mathbb{N}^{3}\to \mathbb{N}$ such that for every $r,k,g\in\mathbb{N}$,
if $\Sigma$ is a surface of genus $g$, 
$G$ is a graph embedded on $\Sigma$,
$\Delta$ is an open disk of $\Sigma$,
$v$ is a vertex in $\Delta\cap V(G)$,
${\cal C}$ is a $\Delta$-nested sequence of  cycles  of $G$, where $|{\cal C}|\geq\funref{reducedfun}(r,k,g)$, that is $r$-tight around $v$, and
$L$ is a $\Delta$-avoiding $r$-scattered linkage of $G$ of size $k$ that is BC-minimal around $v$,
then there is a $\Delta$-avoiding $r$-scattered linkage $L'$ of $G\setminus v$ such that $L\equiv L'$,  $L\setminus \Delta=L' \setminus \Delta$, and $L'\subseteq L\cup\cupall {\cal C}$.
\end{lemma}

The proof of \autoref{reducedlinkage} is postponed to~\autoref{subsec_proofofreducedlinkage}.
We now show how to use  \autoref{reducedlinkage} to prove the following result.

\begin{lemma}\label{maintheorem}
	There is a function $\newfun{mainfun}:\mathbb{N}^{3}\to \mathbb{N}$ such that for every $r,k,g\in \mathbb{N}$ if $\Sigma$ is a surface of genus $g$, $G$ is a graph embedded on $\Sigma$, $\Delta$ is an open disk of $\Sigma$, $L$ is a $\Delta$-avoiding $r$-scattered linkage of $G$ of size at most $k$, ${\cal C}$ is a $\Delta$-nested sequence of cycles of size $\funref{mainfun}(r,k,g)$,  $v$ is a vertex of $G$ that is isolated in $G$ by ${\cal C}$, then there is a $\Delta$-avoiding $r$-scattered linkage $L'$ of $G\setminus v$ that is equivalent to $L$ such that $L'\subseteq L\cup\cupall {\cal C}$. Moreover, it holds that $\funref{mainfun}(r,k,g)=r\cdot 2^{{\cal O}(k+g)}$.
\end{lemma}

\begin{proof}
	Let $r,k,g\in \mathbb{N}$. We set $\funref{mainfun}(r,k,g)=r\cdot \funref{reducedfun}(r,k,g)$.
	Let $\Sigma$ be a surface of genus $g$, $G$ be a graph embedded in $\Sigma$, $\Delta$ be an open disk of $\Sigma$, $L$ be a $\Delta$-avoiding $r$-scattered linkage of $G$ of size at most $k$, and $v$ be a vertex of $G$ that is $(\funref{mainfun}(r,k,g),\Delta)$-isolated in $G$.
	%We prove that there is a $\Delta$-avoiding $r$-scattered linkage $L'$ of $G\setminus v$ that is equivalent to $L$.
	We assume that $v\in V(L)$, since otherwise the theorem holds trivially.

	Since $v$ is $(\funref{mainfun}(r,k,g),\Delta)$-isolated in $G$ and $\funref{mainfun}(r,k,g)=r\cdot \funref{reducedfun}(r,k,g)$, there is a nested sequence ${\cal C}$ of cycles of $G$ of size at least $\funref{reducedfun}(r,k,g)$ that is $r$-tight around $v$ and whose outer disk $\Delta'$ is a subset of $\Delta$.
	Among all $\Delta$-avoiding $r$-scattered linkages of $G$ that are equivalent to $L$ and contain $v$, let $\tilde{L}$ be the one that minimizes the quantities ${\sf cross}_{\cal C}(\tilde{L})$ and ${\sf bridges}_{\Delta}(\tilde{L})$.
	Observe that, since $\Delta'\subseteq \Delta$, $\tilde{L}$ is also $\Delta'$-avoiding.
	By \autoref{reducedlinkage}, there is a $\Delta'$-avoiding $r$-scattered linkage $\tilde{L}'$ of $G\setminus v$ that is equivalent to $\tilde{L}$ and moreover $\tilde{L}\setminus \Delta'=\tilde{L}' \setminus \Delta'$.
	Notice that since $\tilde{L}$ is $\Delta$-avoiding, $\Delta'\subseteq \Delta$, and $\tilde{L}\setminus \Delta'=\tilde{L}' \setminus \Delta'$, it follows that $\tilde{L}'$ is also $\Delta$-avoiding. Therefore, $L':=\tilde{L}'$ is the claimed linkage.
\end{proof}

We next show how to prove~\autoref{amultheah} using~\autoref{maintheorem}.
Before this, we introduce some additional definitions concerning {\sl walls}.

\paragraph{Walls.}
Let  $k,r\in\mathbb{N}.$ The
\emph{$(k\times r)$-grid} is the
graph whose vertex set is $[k]\times[r]$ and two vertices $(i,j)$ and $(i',j')$ are adjacent if and only if $|i-i'|+|j-j'|=1.$
An  \emph{elementary $r$-wall}, for some odd integer $r\geq 3,$ is the graph obtained from a
$(2 r\times r)$-grid
with vertices $(x,y)
	\in[2r]\times[r],$
after the removal of the
``vertical'' edges $\{(x,y),(x,y+1)\}$ for odd $x+y,$ and then the removal of
all vertices of degree one.
Notice that, as $r\geq 3,$  an elementary $r$-wall is a planar graph
that has a unique (up to topological isomorphism) embedding in the plane
such that all its finite faces are incident to exactly six
edges.
The {\em perimeter} of an elementary $r$-wall is the cycle bounding its infinite face.

An {\em $r$-wall} is any graph $W$ obtained from an elementary $r$-wall $\bar{W}$
after subdividing 
edges\footnote{
Given an edge $e=\{u,v\}\in E(G),$ we define the {\em subdivision} of $e$ to be the operation of deleting $e,$ adding a new vertex $w$ and making it adjacent to $u$ and $v.$}.
The {\em perimeter} of $W$, denoted by $\perim(W)$, is the cycle of $W$ whose non-subdivision vertices are the vertices of the perimeter of $\overline{W}$. 

Given an elementary $r$-wall $\bar{W},$ some odd $i\in \{1,3,\ldots,2r-1\},$ and $i'=(i+1)/2,$
the {\em $i'$-th  vertical path} of $\bar{W}$  is the one whose
vertices, in order of appearance, are $(i,1),(i,2),(i+1,2),(i+1,3),
	(i,3),(i,4),(i+1,4),(i+1,5),
	(i,5),\ldots,(i,r-2),(i,r-1),(i+1,r-1),(i+1,r).$
Also, given some $j\in[2,r-1]$ the {\em $j$-th horizontal path} of $\bar{W}$
is the one whose
vertices, in order of appearance, are $(1,j),(2,j),\ldots,(2r,j).$
A \emph{vertical} (resp. \emph{horizontal}) path of $W$ is one
that is a subdivision of a  vertical (resp. horizontal) path of $\bar{W}.$
Notice that the perimeter of an $r$-wall $W$
is uniquely defined regardless of the choice of the elementary $r$-wall $\bar{W}.$
A {\em subwall} of $W$ is any subgraph $W'$ of  $W$
that is an $r'$-wall, with $r' \leq r,$ and such the vertical (resp. horizontal) paths of $W'$ are subpaths of the
	{vertical} (resp. {horizontal}) paths of $W.$

Let an odd integer $r\geq 3.$
Let $W$ be an $r$-wall of a graph $G$ and $K'$ be the connected component of $G\setminus \perim(W)$ that contains $W\setminus \perim(W)$.
The {\em compass} of $W$, denoted by ${\sf Compass}(W)$, is the graph $G[V(K')\cup V(\perim(W))]$. Observe that $W$ is a subgraph of ${\sf Compass}(W)$ and ${\sf Compass}(W)$ is connected.

The {\em layers} of an $r$-wall $W$  are recursively defined as follows.
The first layer of $W$ is its perimeter. For $i=2,\ldots,(r-1)/2,$ the $i$-th layer of $W$ is the $(i-1)$-th layer of the subwall $W'$ obtained from $W$ after removing from $W$ its perimeter and all occurring vertices of degree one. Notice that each $(2r+1)$-wall has $r$ layers.
The {\em central  vertices} of $W$ are the two branch vertices of $W$ that do not belong to any of its layers and that are connected by a path of $W$ that does not intersect any layer.
\medskip

We will use the following relation between the treewidth of a graph $G$ embedded on a surface of fixed genus and a wall of $G$, derived from \cite[Theorem 4.12]{DemaineFHT05subex}.

\begin{proposition}\label{fnsdjkfnsdjk}
There is a function $\newfun{fun_treewidth}:\mathbb{N}^2\to\mathbb{N}$ such that for every $r,g\in\mathbb{N}$,
if $G$ is a graph embedded on a surface $\Sigma$ of genus $g$ and $\tw(G)> \funref{fun_treewidth}(r,g)$, then $G$ contains an $r$-wall as a subgraph.
Moreover, $\funref{fun_treewidth}(r,g) = \mathcal{O}(r\cdot g)$.
\end{proposition}

\begin{proof}[Proof of~\autoref{amultheah}]
We set $m=\sqrt{2k+1}\cdot (2\cdot \funref{mainfun}(r,k,g) +1)$ and $\funref{wtfun}=\funref{fun_treewidth}(g,m)$.
Let $G$ be a graph embedded on a surface $\Sigma$ of genus $g$ and let 
$L$ be an $r$-scattered linkage of $G$ of size at most $k$.
Also, suppose that $\tw(G)\geq \funref{wtfun}(r,k,g)$.
Since $\tw(G)\geq \funref{wtfun}(r,k,g)$ and $\funref{wtfun}(r,k,g) =\funref{fun_treewidth}(g,m)$,
by~\autoref{fnsdjkfnsdjk} we have that $G$ contains an $m$-wall $W$ as a subgraph.
Also, since $L$ is an $r$-scattered linkage of $G$ of size at most $k$ and $m=\sqrt{2k+1}\cdot (2\cdot \funref{mainfun}(r,k,g) +1)$, there is a subwall $W'$ of $W$ of height $2\cdot \funref{mainfun}(r,k,g) +1$ such that $W$ is embedded in a closed disk $\Delta$ of $\Sigma$ and ${\sf Compass}(W')$ does not contain any terminal of $L$.
Therefore, $L$ is $\Delta$-avoiding.
Let ${\cal C}$ be the collection of the layers of $W'$, let $v$ be a central vertex of $W'$, and observe that $|{\cal C}|=\funref{mainfun}(r,k,g)$ and $v$ is isolated in $G$ by ${\cal C}$.
By~\autoref{maintheorem}, there is a $\Delta$-avoiding $r$-scattered linkage $L'$ of $G\setminus v$ that is equivalent to $L$.
Moreover, since $\funref{mainfun}(r,k,g)=r\cdot 2^{{\cal O}(k+g)}$, we also have that $\funref{wtfun}(r,k,g)=r\cdot 2^{{\cal O}(k+g)}$.
\end{proof}

\subsection{Proof of~\autoref{reducedlinkage}}\label{subsec_proofofreducedlinkage}

In order to show~\autoref{reducedlinkage}, we first observe that if the given linkage has few bridges, then
we can find an equivalent linkage that avoids an isolated vertex $v$.
This is formulated in the following lemma
that can be derived from the proof of \cite[Theorem 7]{KawarabayashiK12alin}.
We prove it here for completeness.

\begin{lemma}\label{fewbridges}
Let $\ell, r, k\in\mathbb{N}$.
Let $\Delta$ be a closed annulus.
If $G$ is a partially $\Delta$-embedded graph,
$v$ is a vertex of $G$,
${\cal C}$ is a $\Delta$-nested sequence of  cycles  of $G$ of size at least $\ell$ that is $r$-tight around $v$, and
$L$ is a $\Delta$-avoiding $r$-scattered linkage of $G$ of size at most $k$ that is BC-minimal around $v$ and ${\sf bridges}_{\Delta}(L)\leq\ell -\lfloor k/2\rfloor-1$,
then there is a $\Delta$-avoiding $r$-scattered linkage $L'$ of $G\setminus v$ that is equivalent to $L$,
$L\setminus \Delta=L' \setminus \Delta$, and $L'\subseteq L\cup\cupall {\cal C}$.
\end{lemma}
 
\begin{proof}
Since $L$ has size at most $k$, $\mathcal{P}(L)$ has size at most $\lfloor k/2\rfloor$.
Also, since $L$ is $\Delta$-avoiding, for every
$P\in\mathcal{P}(L)$, the number of components in $P\cap\inter(\Delta)$ is equal to ${\sf bridges}_{\Delta}(P)+1$.
Therefore, the number of components in $L\cap \inter(\Delta)$ is equal to ${\sf bridges}_{\Delta}(L)+|\mathcal{P}(L)|$.
Since ${\sf bridges}_{\Delta}(L)\leq\ell -\lfloor k/2\rfloor-1$ and $|\mathcal{P}(L)|\leq \lfloor k/2\rfloor$,
we have that there exist at most $\ell-1$ components in $L\cap {\sf int}(\Delta)$.
Now, it is easy to see that the fact that ${\cal C}$ is a $\Delta$-nested sequence of cycles of $G$ of size at least $\ell$ implies that there is a $\Delta$-avoiding $r$-scattered linkage $L'$ of $G\setminus v$ that is equivalent to $L$ and $L\setminus \Delta=L' \setminus \Delta$ (intuitively, $L'$ can be obtained by shortcuting every component of
$L\cap \inter(\Delta)$ to pass through the cycles of ${\cal C}$).
\end{proof}

In the rest of this section, our goal is to argue that one can always reduce the number of bridges.
We start with some additional definitions.

\paragraph{Rainbows.}
Let $\Sigma$ be a surface and let $\Delta$ be an open disk 
of $\Sigma$. We set $\overline{\Delta}=\Sigma\setminus\Delta$. 
Let also $G$ be a $\overline{\Delta}$-embedded  1-regular graph
such that $V(G)\subseteq {\bf bd}(\Delta)$ and $\cupall E(G)\cap {\bf bd}(\Delta)=\emptyset$.
We call $G$ a {\em $\Delta$-outer matching}.
Let $C$ be the embedded cycle whose vertices are the vertices 
of $G$ and whose  edges are the connected components of $ {\bf bd}(\Delta)\setminus V(G)$.
We denote $G^+=G\cup C$ and observe that $G^+$ is a $\Sigma$-embedded 
3-regular multigraph.

The {\em facets} of $G$ are the connected components 
of $\Sigma\setminus G^+$ that are different than $\Delta$.
%We call a facet of $G$ {\em simplicial}  if its boundary consists of the two edges of a multiedge
%of $G^+$ of multiplicity 2 among with their endpoints (one of these edges should be an edge of $G$ and the other should be an edge of $C$). We denote by $S(G)$ the set of all simplicial facets of $G$.
We call a facet of $G$ a {\em bar} if it is homeomorphic to 
an open disk $D$ whose boundary is the union of two edges of $G$, two edges of $C$,
and the 4 endpoints of those edges.  We denote by ${\sf Bars}(G)$ the set of all bars of $G$.
Given a $b\in {\sf Bars}(G)$, we denote by ${\sf bd\text{-}edges}(b)$ the set of edges of $G$ that are contained in the boundary of $b$ and observe that for every $b\in {\sf Bars}(G)$, $|{\sf bd\text{-}edges}(b)|=2$.

%Given two bars $b_{1},b_{2}\in B(G)$, we say that they are {\em neighboring}, if 
%%there is an edge of $G$ that is contained to both boundaries of $b_{1}$ and $b_{2}$.
%${\sf bd\text{-}edges}(b_{1})\cap {\sf bd\text{-}edges}(b_{2})\neq \emptyset$.
%We consider the transitive closure of 
%the neighboring relation  and let ${\cal B}(G)=\{B_{1},\ldots,B_{r}\}$ be the partition of ${B}(G)$ defined be the resulting 
%equivalence relation.
%Notice that for every $B_i\in {\cal B}(G)$, it holds that 
%$\bigcup_{f\in B_{i}}{\bf cov}(f)$ is homeomorphic to a closed disk of $\overline{\Delta}$.
%\blue{
%	We say that a set $E\subseteq E(G)$ is a {\em $\Delta$-rainbow} of $G$ if there is a  $B\in {\cal B}(G)$ such that 
%	$E=\bigcup_{f\in B}{\sf bd\text{-}edges}(f)$.
%	%\{e\mid \exists b,b'\in B, \ b\neq b'\mbox{ such that }e=\bd(b)\cap \bd(b')\}.$$
%	%\bigcup_{f\in B_{i}}\bd(f)\cap \cupall E(G)$.
%}

Given two edges $e_{1}, e_{2}\in E(G)$, we say that they are {\em neighboring}, if $e_{1}=e_{2}$ or there is a $b\in {\sf Bars}(G)$ such that ${\sf bd\text{-}edges}(b)=\{e_{1},e_{2}\}$.
We consider the transitive closure of 
the neighboring relation and let ${\cal E}(G)$ be the partition of $E(G)$ defined be the resulting 
equivalence relation. We call ${\cal E}(G)$ the {\em homotopy-partition} of $E(G)$ and each $E\in {\cal E}(G)$ a {\em $\Delta$-rainbow} of $G$.
We say that a $\Delta$-rainbow $E$ of $G$ is {\em trivial} if $|E|=1$.
See~\autoref{figure_rainbows} for an illustration of the above notions.

\begin{figure}[ht]
\centering
\scalebox{0.55}{
\begin{tikzpicture}[ipe stylesheet]
  \filldraw[ipe dash dashed, fill=firebrick1, ipe opacity 10]
    (334.478, 606.911)
     .. controls (316.907, 619.278) and (305.113, 623.9045) .. (294.7623, 627.3873)
     .. controls (284.4117, 630.87) and (275.5043, 633.209) .. (263.7945, 635.2577)
     .. controls (252.0847, 637.3063) and (237.5723, 639.0647) .. (223.4757, 639.1613)
     .. controls (209.379, 639.258) and (195.698, 637.693) .. (178.0868, 633.2813)
     .. controls (160.4755, 628.8695) and (138.934, 621.611) .. (115.824, 602.251)
     .. controls (112.064, 598.153) and (104.416, 590.601) .. (99.2271, 578.98)
     .. controls (107.07, 592.081) and (116.3685, 598.9995) .. (125.7084, 604.7706)
     .. controls (135.0483, 610.5417) and (144.4297, 615.1653) .. (154.53, 618.7835)
     .. controls (164.6303, 622.4017) and (175.4497, 625.0143) .. (187.9758, 627.0637)
     .. controls (200.502, 629.113) and (214.735, 630.599) .. (229.3545, 630.2527)
     .. controls (243.974, 629.9063) and (258.98, 627.7277) .. (270.0745, 625.7035)
     .. controls (281.169, 623.6793) and (288.352, 621.8097) .. (294.2178, 619.986)
     .. controls (300.0837, 618.1623) and (304.6323, 616.3847) .. (310.1767, 614.1723)
     .. controls (315.721, 611.96) and (322.261, 609.313) .. (328.0707, 605.973)
     .. controls (333.8803, 602.633) and (338.9597, 598.6) .. (343.1618, 594.705)
     .. controls (347.364, 590.81) and (350.689, 587.053) .. (354.485, 581.442)
     .. controls (354.485, 581.442) and (351.4685, 587.885) .. (348.2635, 592.6853)
     .. controls (345.0585, 597.4855) and (341.665, 600.643) .. (334.478, 606.911)
     -- cycle;
     \draw[draw=black!30!white,ipe dash dashed]
    (334.478, 606.911)
     .. controls (316.907, 619.278) and (305.113, 623.9045) .. (294.7623, 627.3873)
     .. controls (284.4117, 630.87) and (275.5043, 633.209) .. (263.7945, 635.2577)
     .. controls (252.0847, 637.3063) and (237.5723, 639.0647) .. (223.4757, 639.1613)
     .. controls (209.379, 639.258) and (195.698, 637.693) .. (178.0868, 633.2813)
     .. controls (160.4755, 628.8695) and (138.934, 621.611) .. (115.824, 602.251)
     .. controls (112.064, 598.153) and (104.416, 590.601) .. (99.2271, 578.98)
     .. controls (107.07, 592.081) and (116.3685, 598.9995) .. (125.7084, 604.7706)
     .. controls (135.0483, 610.5417) and (144.4297, 615.1653) .. (154.53, 618.7835)
     .. controls (164.6303, 622.4017) and (175.4497, 625.0143) .. (187.9758, 627.0637)
     .. controls (200.502, 629.113) and (214.735, 630.599) .. (229.3545, 630.2527)
     .. controls (243.974, 629.9063) and (258.98, 627.7277) .. (270.0745, 625.7035)
     .. controls (281.169, 623.6793) and (288.352, 621.8097) .. (294.2178, 619.986)
     .. controls (300.0837, 618.1623) and (304.6323, 616.3847) .. (310.1767, 614.1723)
     .. controls (315.721, 611.96) and (322.261, 609.313) .. (328.0707, 605.973)
     .. controls (333.8803, 602.633) and (338.9597, 598.6) .. (343.1618, 594.705)
     .. controls (347.364, 590.81) and (350.689, 587.053) .. (354.485, 581.442)
     .. controls (354.485, 581.442) and (351.4685, 587.885) .. (348.2635, 592.6853)
     .. controls (345.0585, 597.4855) and (341.665, 600.643) .. (334.478, 606.911)
     -- cycle;
      \filldraw[ipe dash dashed, fill=deepskyblue, ipe opacity 10]
    (93.9309, 559.803)
     .. controls (98.0465, 568.275) and (100.9678, 571.19) .. (104.0362, 574.3855)
     .. controls (107.1047, 577.581) and (110.3203, 581.057) .. (113.4665, 583.8377)
     .. controls (116.6127, 586.6183) and (119.6893, 588.7037) .. (122.6805, 590.6032)
     .. controls (125.6717, 592.5027) and (128.5773, 594.2163) .. (131.892, 595.963)
     .. controls (135.2067, 597.7097) and (138.9303, 599.4893) .. (142.4587, 601.0242)
     .. controls (145.987, 602.559) and (149.32, 603.849) .. (153.5662, 605.2173)
     .. controls (157.8123, 606.5857) and (162.9717, 608.0323) .. (169.4878, 609.6005)
     .. controls (176.004, 611.1687) and (183.877, 612.8583) .. (192.8782, 614.232)
     .. controls (201.8793, 615.6057) and (212.0087, 616.6633) .. (221.1595, 616.841)
     .. controls (230.3103, 617.0187) and (238.4827, 616.3163) .. (244.7787, 615.7247)
     .. controls (251.0747, 615.133) and (255.4943, 614.652) .. (261.8657, 613.5995)
     .. controls (268.237, 612.547) and (276.56, 610.923) .. (283.0567, 609.3967)
     .. controls (289.5533, 607.8703) and (294.2237, 606.4417) .. (299.6227, 604.6325)
     .. controls (305.0217, 602.8233) and (311.1493, 600.6337) .. (317.0272, 597.9683)
     .. controls (322.905, 595.303) and (328.533, 592.162) .. (332.301, 589.8357)
     .. controls (336.069, 587.5093) and (337.977, 585.9977) .. (341.3208, 582.9202)
     .. controls (344.6647, 579.8427) and (349.4443, 575.1993) .. (352.7054, 570.7369)
     .. controls (355.9665, 566.2745) and (357.709, 561.993) .. (359.043, 548.473)
     .. controls (360.542, 556.739) and (359.561, 568.182) .. (354.485, 581.442)
     .. controls (350.689, 587.053) and (347.364, 590.81) .. (343.1618, 594.705)
     .. controls (338.9597, 598.6) and (333.8803, 602.633) .. (328.0707, 605.973)
     .. controls (322.261, 609.313) and (315.721, 611.96) .. (310.1767, 614.1723)
     .. controls (304.6323, 616.3847) and (300.0837, 618.1623) .. (294.2178, 619.986)
     .. controls (288.352, 621.8097) and (281.169, 623.6793) .. (270.0745, 625.7035)
     .. controls (258.98, 627.7277) and (243.974, 629.9063) .. (229.3545, 630.2527)
     .. controls (214.735, 630.599) and (200.502, 629.113) .. (187.9758, 627.0637)
     .. controls (175.4497, 625.0143) and (164.6303, 622.4017) .. (154.53, 618.7835)
     .. controls (144.4297, 615.1653) and (135.0483, 610.5417) .. (125.7084, 604.7706)
     .. controls (116.3685, 598.9995) and (107.07, 592.081) .. (99.2274, 578.98)
     .. controls (95.5675, 571.656) and (94.4119, 563.085) .. (93.9309, 559.803)
     -- cycle;
  \draw[draw=black!30!white,ipe dash dashed]
    (93.9309, 559.803)
     .. controls (98.0465, 568.275) and (100.9678, 571.19) .. (104.0362, 574.3855)
     .. controls (107.1047, 577.581) and (110.3203, 581.057) .. (113.4665, 583.8377)
     .. controls (116.6127, 586.6183) and (119.6893, 588.7037) .. (122.6805, 590.6032)
     .. controls (125.6717, 592.5027) and (128.5773, 594.2163) .. (131.892, 595.963)
     .. controls (135.2067, 597.7097) and (138.9303, 599.4893) .. (142.4587, 601.0242)
     .. controls (145.987, 602.559) and (149.32, 603.849) .. (153.5662, 605.2173)
     .. controls (157.8123, 606.5857) and (162.9717, 608.0323) .. (169.4878, 609.6005)
     .. controls (176.004, 611.1687) and (183.877, 612.8583) .. (192.8782, 614.232)
     .. controls (201.8793, 615.6057) and (212.0087, 616.6633) .. (221.1595, 616.841)
     .. controls (230.3103, 617.0187) and (238.4827, 616.3163) .. (244.7787, 615.7247)
     .. controls (251.0747, 615.133) and (255.4943, 614.652) .. (261.8657, 613.5995)
     .. controls (268.237, 612.547) and (276.56, 610.923) .. (283.0567, 609.3967)
     .. controls (289.5533, 607.8703) and (294.2237, 606.4417) .. (299.6227, 604.6325)
     .. controls (305.0217, 602.8233) and (311.1493, 600.6337) .. (317.0272, 597.9683)
     .. controls (322.905, 595.303) and (328.533, 592.162) .. (332.301, 589.8357)
     .. controls (336.069, 587.5093) and (337.977, 585.9977) .. (341.3208, 582.9202)
     .. controls (344.6647, 579.8427) and (349.4443, 575.1993) .. (352.7054, 570.7369)
     .. controls (355.9665, 566.2745) and (357.709, 561.993) .. (359.043, 548.473)
     .. controls (360.542, 556.739) and (359.561, 568.182) .. (354.485, 581.442)
     .. controls (350.689, 587.053) and (347.364, 590.81) .. (343.1618, 594.705)
     .. controls (338.9597, 598.6) and (333.8803, 602.633) .. (328.0707, 605.973)
     .. controls (322.261, 609.313) and (315.721, 611.96) .. (310.1767, 614.1723)
     .. controls (304.6323, 616.3847) and (300.0837, 618.1623) .. (294.2178, 619.986)
     .. controls (288.352, 621.8097) and (281.169, 623.6793) .. (270.0745, 625.7035)
     .. controls (258.98, 627.7277) and (243.974, 629.9063) .. (229.3545, 630.2527)
     .. controls (214.735, 630.599) and (200.502, 629.113) .. (187.9758, 627.0637)
     .. controls (175.4497, 625.0143) and (164.6303, 622.4017) .. (154.53, 618.7835)
     .. controls (144.4297, 615.1653) and (135.0483, 610.5417) .. (125.7084, 604.7706)
     .. controls (116.3685, 598.9995) and (107.07, 592.081) .. (99.2274, 578.98)
     .. controls (95.5675, 571.656) and (94.4119, 563.085) .. (93.9309, 559.803)
     -- cycle;
  \filldraw[draw=darkgray, fill=papayawhip, ipe opacity 30]
    (316.3498, 618.5817)
     .. controls (293.0807, 630.8217) and (260.7333, 639.8123) .. (227.7943, 640.2692)
     .. controls (194.8553, 640.726) and (161.3247, 632.649) .. (136.4742, 617.824)
     .. controls (111.6237, 602.999) and (95.4534, 581.426) .. (94.0334, 559.3673)
     .. controls (92.6134, 537.3087) and (105.9437, 514.7643) .. (129.7674, 498.7132)
     .. controls (153.591, 482.662) and (187.908, 473.104) .. (226.6032, 473.5157)
     .. controls (265.2983, 473.9273) and (308.3717, 484.3087) .. (332.6677, 502.9448)
     .. controls (356.9637, 521.581) and (362.4823, 548.472) .. (358.1462, 569.6622)
     .. controls (353.81, 590.8523) and (339.619, 606.3417) .. cycle;
  \draw[darkgray]
    (156.081, 558.31)
     .. controls (203.467, 533.038) and (250.0397, 533.3397) .. (295.799, 559.215);
  \filldraw[draw=darkgray, fill=white]
    (161.967, 555.96)
     .. controls (205.7203, 578.4427) and (248.888, 578.7547) .. (291.47, 556.896)
     .. controls (265.083, 540.657) and (216.614, 528.702) .. (161.691, 555.471)
     -- cycle;
  \filldraw[draw=darkgreen, fill=palegreen]
    (200.1697, 502.6265)
     .. controls (200.0917, 515.1593) and (211.8903, 527.2457) .. (224.4978, 527.635)
     .. controls (237.1053, 528.0243) and (250.5217, 516.7167) .. (250.5997, 504.1838)
     .. controls (250.6777, 491.651) and (237.4173, 477.893) .. (224.8098, 477.5037)
     .. controls (212.2023, 477.1143) and (200.2477, 490.0937) .. cycle;
  \draw
    (217.811, 526.345)
     .. controls (218.268, 534.92) and (218.988, 537.249) .. (219.564, 538.4483)
     .. controls (220.14, 539.6475) and (220.572, 539.717) .. (220.572, 539.788);
  \draw[draw=black!30!white,ipe dash dashed]
    (220.572, 539.788)
     .. controls (221.062, 536.963) and (221.202, 534.5105) .. (221.2508, 530.2956)
     .. controls (221.2997, 526.0807) and (221.2573, 520.1033) .. (221.1592, 515.3373)
     .. controls (221.061, 510.5713) and (220.907, 507.0167) .. (220.7785, 503.939)
     .. controls (220.65, 500.8613) and (220.547, 498.2607) .. (220.2578, 493.9643)
     .. controls (219.9685, 489.668) and (219.493, 483.676) .. (217.495, 473.557);
  \draw
    (217.516, 473.697)
     .. controls (216.9287, 474.8497) and (216.64, 476.6383) .. (216.65, 479.063);
  \draw[line cap=round]
    (99.2271, 578.98)
     -- (99.2271, 578.98);
  \filldraw[fill=firebrick1, ipe opacity 50]
    (241.806, 521.135)
     .. controls (302.468, 532.313) and (317.4165, 541.0365) .. (326.8361, 548.1121)
     .. controls (336.2557, 555.1877) and (340.1463, 560.6153) .. (342.5013, 565.506)
     .. controls (344.8563, 570.3967) and (345.6757, 574.7503) .. (345.8877, 578.7795)
     .. controls (346.0997, 582.8087) and (345.7043, 586.5133) .. (344.2274, 590.9794)
     .. controls (342.7505, 595.4455) and (340.192, 600.673) .. (334.477, 606.912)
     .. controls (346.849, 598.07) and (353.185, 584.268) .. (354.485, 581.442)
     .. controls (356.456, 574.543) and (356.5165, 570.3625) .. (355.4136, 565.3901)
     .. controls (354.3107, 560.4177) and (352.0443, 554.6533) .. (348.7903, 549.4985)
     .. controls (345.5363, 544.3437) and (341.2947, 539.7983) .. (332.2651, 533.2037)
     .. controls (323.2355, 526.609) and (309.418, 517.965) .. (250.397, 506.76)
     .. controls (248.855, 512.835) and (245.656, 517.648) .. (241.678, 521.102);
  \filldraw[fill=firebrick1, ipe opacity 50]
    (208.021, 520.131)
     .. controls (160.925, 524.591) and (138.0765, 537.583) .. (124.9386, 548.7037)
     .. controls (111.8007, 559.8243) and (108.3733, 569.0737) .. (107.9439, 577.6841)
     .. controls (107.5145, 586.2945) and (110.083, 594.266) .. (115.824, 602.251)
     .. controls (112.064, 598.153) and (104.416, 590.601) .. (99.2275, 578.98)
     .. controls (97.6636, 569.767) and (97.7865, 565.203) .. (99.1078, 560.2918)
     .. controls (100.4292, 555.3807) and (102.9491, 550.1223) .. (106.0227, 545.7595)
     .. controls (109.0963, 541.3967) and (112.7237, 537.9293) .. (117.9295, 533.9622)
     .. controls (123.1353, 529.995) and (129.9197, 525.528) .. (140.8066, 520.495)
     .. controls (151.6935, 515.462) and (166.683, 509.863) .. (200.299, 505.098)
     .. controls (200.8837, 511.1647) and (203.4577, 516.1757) .. (208.021, 520.131)
     -- cycle;
  \filldraw[fill=deepskyblue, ipe opacity 50]
    (246.182, 490.901)
     .. controls (294.988, 498.813) and (308.9975, 504.5255) .. (319.4029, 509.6293)
     .. controls (329.8083, 514.733) and (336.6097, 519.228) .. (342.1933, 524.1027)
     .. controls (347.777, 528.9773) and (352.143, 534.2317) .. (354.8973, 538.6313)
     .. controls (357.6515, 543.031) and (358.794, 546.576) .. (359.043, 548.473)
     .. controls (360.542, 556.739) and (359.561, 568.182) .. (354.485, 581.442)
     .. controls (356.456, 574.543) and (356.5165, 570.3625) .. (355.4136, 565.3901)
     .. controls (354.3107, 560.4177) and (352.0443, 554.6533) .. (348.7903, 549.4985)
     .. controls (345.5363, 544.3437) and (341.2947, 539.7983) .. (332.2651, 533.2037)
     .. controls (323.2355, 526.609) and (309.418, 517.965) .. (250.397, 506.76)
     .. controls (251.5977, 502.0067) and (250.1927, 496.7203) .. (246.182, 490.901)
     -- cycle;
  \filldraw[fill=deepskyblue, ipe opacity 50]
    (204.423, 489.536)
     .. controls (138.831, 497.936) and (116.3515, 514.689) .. (105.0887, 528.3432)
     .. controls (93.826, 541.9975) and (93.78, 552.553) .. (93.9309, 559.803)
     .. controls (94.4119, 563.085) and (95.5675, 571.656) .. (98.7363, 577.959)
     .. controls (97.369, 569.933) and (97.4919, 565.369) .. (98.8132, 560.4578)
     .. controls (100.1345, 555.5467) and (102.6542, 550.2883) .. (105.7278, 545.9255)
     .. controls (108.8013, 541.5627) and (112.4287, 538.0953) .. (117.6345, 534.1282)
     .. controls (122.8403, 530.161) and (129.6247, 525.694) .. (140.5116, 520.661)
     .. controls (151.3985, 515.628) and (166.388, 510.029) .. (200.004, 505.264)
     .. controls (200.308, 500.588) and (200.229, 495.453) .. (204.423, 489.536)
     -- cycle;
  \node[ipe node, font=\LARGE]
     at (228.149, 495.218) {$\Delta$};
  \pic
     at (208.021, 520.131) {ipe disk};
  \pic
     at (200.299, 505.098) {ipe disk};
  \pic
     at (204.423, 489.536) {ipe disk};
  \pic
     at (241.678, 521.102) {ipe disk};
  \pic
     at (250.397, 506.76) {ipe disk};
  \pic
     at (246.182, 490.901) {ipe disk};
  \pic
     at (216.65, 479.063) {ipe disk};
  \pic
     at (217.811, 526.345) {ipe disk};
  \node[ipe node, font=\LARGE]
     at (302.918, 548.249) {$G$};
  \node[ipe node, font=\Large]
     at (204.217, 531.665) {$e_1$};
  \node[ipe node, font=\Large]
     at (254.476, 529.673) {$e_2$};
  \node[ipe node, font=\Large]
     at (274.49, 518.04) {$e_3$};
  \node[ipe node, font=\Large]
     at (279.492, 488.655) {$e_4$};
\end{tikzpicture}}
\caption{A surface $\Sigma$ of Euler genus one, an open disk $\Delta$ of $\Sigma$ (depicted in green), and a $\Delta$-outer matching $G$.
The bars of $G$ are depicted in red and blue.
The edges $e_2$ and $e_3$ of $G$ and the edges $e_3$ and $e_4$ are neighboring.
The set $\{e_2,e_3,e_4\}$ is 
a $\Delta$-rainbow of $G$. The singleton $\{e_1\}$ is a  trivial $\Delta$-rainbow of $G$.
}
\label{figure_rainbows}
\end{figure}
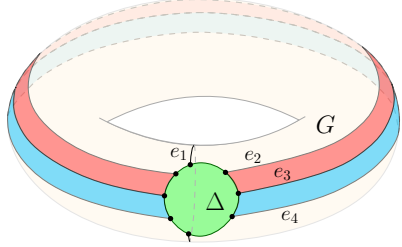

Given a non-trivial $\Delta$-rainbow $E$ of $G$, we define the {\em span} of $E$ in $\Sigma$, denoted by ${\sf Span}_{\Sigma}(E)$, to be the set $$\cupall\{b\cup \bd(b)\mid b\in {\sf Bars}(G) \text{  and }  {\sf bd\text{-}edges}(b)\subseteq E\}.$$
Notice that the span of every non-trivial $\Delta$-rainbow of $G$ is homeomorphic to a closed disk of $\overline{\Delta}$.
%Notice that for every $E\in {\cal E}(G)$, it holds that $\bigcup_{e\in E}\{{\bf cov(b)}\mid b\in B(G)\text{  and } e\in {\sf bd\text{-}edges}(b)\}$ is homeomorphic to a closed disk of $\overline{\Delta}$.
We call an edge of $E$ {\em peripheral} if it is a subset of the boundary of ${\sf Span}_{\Sigma}(E)$.
Observe that every non-trivial $\Delta$-rainbow of $G$ has exactly two peripheral edges.
%The following result is \cite[Proposition 4.2.7]{Mohar}.
%
%\begin{proposition}
%Let $G$ be a $\Pi$-embedded graph and $a,b$ vertices of $G$ (possibly $a=b$).
%If $P_{0}, P_{1},\ldots, P_{k}$ are pairwise internally disjoint paths (or cycles) from $a$ to $b$ such that no two of them are $\Pi$-homotopic, then
%$$k\leq \begin{cases}
%g, & \text{ if }g\leq 1\\
%3g-3, & \text{ if }g\geq 2.
%\end{cases}$$
%\end{proposition}
%Let $\Sigma$ be a surface and $G$ be a $\Sigma$-embedded graph. Also, let $u,v$ be distinct vertices of $G$ and $P, P'$ be two internally disjoint paths from $v$ to $u$ in $G$. Then, $P$ and $P'$ are called {\em homotopic} if the cycle $P\cup P'$ is contractible.

The following result is \cite[Proposition 4.2.7]{MoharT01grap}.

\begin{proposition}\label{jfopdsn}
	Let $\Sigma$ be a surface, $G$ be a graph embedded in $\Sigma$, and $u,v$ be vertices of $G$ (possibly $u=v$).
	If $P_{0}, P_{1},\ldots, P_{k}$ are pairwise internally disjoint paths (or cycles) from $u$ to $v$ such that no two of them are homotopic, then
	$$k\leq \begin{cases}
	{\bf eg}(\Sigma), & \text{ if }{\bf eg}(\Sigma)\leq 1\\
	3{\bf eg}(\Sigma)-3, & \text{ if }{\bf eg}(\Sigma)\geq 2.
	\end{cases}$$
\end{proposition}

We now use \autoref{jfopdsn} in order to prove the following result, which provides a bound on the size of the homotopy-partition of a given $\Delta$-outer matching. This gives an upper bound on the number of different $\Delta$-rainbows of an open disk $\Delta$ of a surface $\Sigma$ as a function of ${\bf eg}(\Sigma)$.

\begin{lemma}\label{rainbow}
	Let $g\in \mathbb{N}$, $\Sigma$ be a surface of genus $g$, $\Delta$ be an open disk 
	of $\Sigma$, $M$ be a  $\Delta$-outer matching, and ${\cal E}(M)$ be the homotopy-partition of $M$. It holds that 
	$|{\cal E}(M)|\leq 3g-2$.
\end{lemma}
%\begin{lemma}\label{rainbow}
%	There is a function $\newfun{moharfun}:\mathbb{N}^{2}\to \mathbb{N}$ such that for every $g\in \mathbb{N}$, if $\Sigma$ is a surface of genus $g$, $\Delta$ is an open disk 
%	of $\Sigma$, and $G$ is a  $\Delta$-outer $\Sigma$-embedded 1-regular graph, then 
%	$|{\cal E}(G)|\leq \funref{moharfun} (g)$.
%\end{lemma}

\begin{proof}
%Let $\funref{moharfun} (x)= 3x$ and let $\Sigma$ be a surface of genus $g$, $\Delta$ be an open disk 
%of $\Sigma$, and $G$ is a  $\Delta$-outer $\Sigma$-embedded 1-regular graph.
Let ${\cal E}(M)=\{E_{1}, \ldots, E_{r}\}$ be the homotopy-partition of $E(M)$. We will prove that $r\leq 3g-3$.
For every $i\in [r]$, let $e_{i}$ be an edge in $E_{i}$ and $E=\bigcup_{i\in [r]}e_{i}$.
Let $x$ be a point of $\Delta$ and $f$ be a homomorphism from $\Sigma$ to $\Sigma$ that maps every point of $\Delta\cup \bd(\Delta)$ to $x$ and leaves everything else untouched.
Observe that since $\Delta$ is an open disk, then ${\bf eg}(f(\Sigma))={\bf eg}(\Sigma)=g$ and if $M'$ is
the graph obtained from $M$ after identifying all vertices of $M$, then $M'$ is  a $f(\Sigma)$-embedded graph (whose edges are all loops).
Since $M$ is a $\Delta$-outer matching, $f(E)$ is a set of pairwise non-crossing loops that are incident to $x$.
Observe that for every $e,e'\in E$, there is an $i\in[r]$ such that $e,e'\in E_i$ if and only if the loops $f(e)$ and $f(e')$ are  non-homotopic.
This means that $r$ is equal to the number of non-homotopic elements of $f(E)$, that, by \autoref{jfopdsn}, are at most $3g-2$.
Therefore, $r\leq 3g-2$.
\end{proof}

Next, we define rainbows in linkages.
Here, in this context, the $\Delta$-outer matching will correspond to the set of $\Delta$-bridges of a given linkage $L$ (in fact, to the contraction of bridges to single edges).

\paragraph{Rainbows in linkages.}
Let $L$ be a $\Delta$-avoiding linkage of a graph $G$ embedded in a surface $\Sigma$, where $\Delta$ is an open disk of $\Sigma$.
%We now extend the definition of a $\Delta$-rainbow of a  $\Delta$-outer $\Sigma$-embedded 1-regular graph to bridges
We define the {\em bridge representative} $H$ of $L$ to be the graph obtained from $\cupall {\cal B}_{\Delta}(L)$ after dissolving every internal vertex of every $\Delta$-bridge of $L$.
%
%if $B$ is a bridge in ${\cal B}_{\Delta}(L)$ with endpoints $x$ and $y$,
%dissolve all vertices in $V(B)\setminus\{x,y\}$ (together with their incident edges) and add the edge $\{x,y\}$, which we denote by $e_{B}$.
Observe that $H$ is a $\Delta$-outer matching.
This observation will allow us to refer to  {\sl homotopy-partitions} and {\sl $\Delta$-rainbows} of the bridge representative of a linkage.

Let ${\cal B}$ be a subset of ${\cal B}_{\Delta}(L)$ and $H$ be the bridge representative of $L$.
We set $E_{{\cal B}}$ to be the set of edges of $H$ that correspond to the $\Delta$-bridges in ${\cal B}$.
We say that ${\cal B}$ is a {\em $\Delta$-rainbow of $L$} if $E_{\cal B}$ is a non-trivial $\Delta$-rainbow of $H$.
Moreover, if $E_{\cal B}$ is a non-trivial $\Delta$-rainbow of $H$ and $e_{1},e_{2}$ are the peripheral edges of $E_{\cal B}$, then we call the corresponding bridges $B_{1},B_{2}$ {\em peripheral bridges} of the $\Delta$-rainbow ${\cal B}$ of $L$.
If ${\cal B}$ is a $\Delta$-rainbow  of $L$ such that $|{\cal B}|\geq 3$ and $B_1,B_2$ are its peripheral bridges, we denote by $\Delta_{\cal B}$ the (unique) connected component of $\Sigma\setminus (\Delta \cup B_{1}\cup B_{2})$ that intersects ${\cal B}$.
For example, in~\autoref{figure_rainbows}, if the ``yellow-pink'' $\Delta$-rainbow $E$ is equal to $E_{\mathcal{B}}$, for some $\mathcal{B}\subseteq \mathcal{B}_\Delta (L)$ of some linkage $L$,
then $\Delta_\mathcal{B}$ corresponds to the open disk ``cropped'' by the peripheral edges of $E$, i.e., the union of the yellow and the pink open disk together with the blue edge that is incident to both of them. 
We say that ${\cal B}$ is {\em clear} if $\Delta_{\cal B}\cap T(L)=\emptyset$.
\smallskip

Let $G$ be a graph, $H$ be a subgraph of $G$, and $F\subseteq E(H)$. Given a graph $J\subseteq G/F$, we say that $H$ is an {\em $F$-expansion} of $J$ if $J$ is obtained from $H$ by contracting all edges in $F$.

%
%\paragraph{Graph contractions.}
%Let $G$ and $H$ be graphs and let $\rho : V(G)\rightarrow V(H)$ be a surjective mapping such that:
%\begin{enumerate}
%	\item for every vertex $v\in V(H)$, its codomain $\rho^{-1}(v)$ induces a connected graph $G[\rho^{-1}(v)]$,
%	\item for every edge $\{u,v\}\in E(H)$, the graph $G[\rho^{-1}(u)\cup \rho^{-1}(v)]$ is connected, and
%	\item for every edge $\{u,v\}\in E(G)$, either $\rho(u)=\rho(v)$ or $\{\rho(u), \rho(v)\}\in E(H)$.
%\end{enumerate}
%We say that {\em $H$ is a contraction of $G$ (via $\rho$)}. % and for a vertex $v\in V(H)$ we call the codomain $\rho^{-1}(v)$ the {\em model of $v$} in $G$.
%Given an edge set $F$ of $G$, we say that $H$ is an {\em $F$-contraction of $G$} if for every edge $\{u,v\}\in E(G)$ it holds that $\rho(u)=\rho(v)$ if and only if $\{u,v\}\in F$.
\medskip

The following result can be derived from the proof of \cite[Theorem 7]{KawarabayashiK12alin}.
It intuitively states that in the presence of a large enough $\Delta$-rainbow of $L$,
one can either find an equivalent linkage with less $\Delta$-bridges or a minor of $G$ that has the following properties: 1) it contains a sequence of nested cycles that isolate a bridge of $L$ and 2) every linkage of this minor can be ``expanded'' to an $r$-scattered linkage of $G$.
The two latter properties will allow us to ``shift'' from $r$-scattered linkages to $0$-scattered linkages and apply the result of Mazoit~\cite{Mazoit13asin} (see~\autoref{maz}) to reroute any given linkage away from the isolated bridge of $L$.
This rerouting allows us to obtain again a linkage with less bridges.

\begin{proposition}\label{japanese}
There exists a function $\newfun{bridgeredfun}:\mathbb{N}^{2}\to \mathbb{N}$ such that for every $r,\ell\in\mathbb{N}$,
if $G$ is a partially $\Delta$-embedded graph,
$v$ is a vertex of $G$,
${\cal C}$ is a $\Delta$-nested sequence of  cycles  of $G$ of size at least $\funref{bridgeredfun}(r,\ell)$ that is $r$-tight around $v$, and
$L$ is a $\Delta$-avoiding $r$-scattered linkage of $G$ such that 
\begin{itemize}
	\item $v\in V(L)$,		
	\item for every $\Delta$-avoiding $r$-scattered linkage $L'$ of $G$ such that $v\in V(L')$ and $L\equiv L'$, it holds that ${\sf cross}_{\cal C}(L)\leq {\sf cross}_{\cal C}(L')$, and
	\item there is a clear $\Delta$-rainbow of $L$ of size  at least $\funref{bridgeredfun}(r,\ell)$,
\end{itemize}
then 	
\begin{enumerate}
	\item either there is a $\Delta$-avoiding $r$-scattered linkage $L'$ of $G$ such that $v\in V(L')$, $L\equiv L'$, and ${\sf bridges}_{\Delta}(L')<{\sf bridges}_{\Delta}(L)$, or
	\item there exist
	\begin{itemize}
		\item a graph $H\subseteq \cupall {\cal C}$ and an edge-set $F\subseteq E(L)\cap \inter(\Delta)$ such that
				% then  $\tilde{G}'$, $\tilde{\cal C}'$, and $\tilde{L}$ is the result of the contraction in $G'$, $C'$, and $L$ of the edges in $F$
				if $\tilde{L}$ is a linkage of $(L\cup H)/F$ %that is equivalent to $L/F$
				then the $F$-expansion $L'$ of $\tilde{L}$ is an $r$-scattered linkage of $G$, and
%				then every linkage $L^*$ of ${G}'/F$ %that is equivalent to $L/F$
%				corresponds to an $r$-scattered linkage $L^\bullet$ in $G$
%				%that is equivalent to $L$
%				, and
		\item a nested sequence of cycles ${\cal C}'$ of $(L\cup H)/F$ of size $\ell$ that  isolates a $\Delta$-bridge of $L$.
	\end{itemize}
\end{enumerate}
Moreover, it holds that $\funref{bridgeredfun}(r,\ell)={\cal O}(r\cdot \ell)$.%\red{$=3r\ell+3$}.
\end{proposition}

Before presenting the proof of~\autoref{reducedlinkage}, we state the main result of~\cite{Mazoit13asin}.

\begin{proposition}\label{maz}
There is a function $\newfun{mazoitfun}:\mathbb{N}^{2}\to \mathbb{N}$ such that for every $k,g\in \mathbb{N}$ if $\Sigma$ is a surface of genus $g$, $G$ is a graph embedded on $\Sigma$, $L$ is a linkage of $G$ of size at most $k$, $v$ is a vertex of $G$, and ${\cal C}$ is a nested sequence of cycles of $G$ of size $\funref{mazoitfun}(k,g)$ that isolates $\{v\}$, then there is a linkage $L'$ of $G\setminus v$ that is equivalent to $L$. Moreover, it holds that $\funref{mazoitfun}(k,g)=2^{{\cal O}(k+g)}$.
\end{proposition}

We conclude this section with the proof of~\autoref{reducedlinkage}.

\begin{proof}[Proof of \autoref{reducedlinkage}]
	Let $r,k,g\in \mathbb{N}$.
	We set

	$$m= (k+1)\cdot \funref{bridgeredfun}(r,\funref{mazoitfun}(k,g)),$$
	$$\funref{reducedfun}(r,k,g)=3g\cdot m+\lfloor k/2\rfloor+1.$$
	
	Let $\Sigma$ be a surface of genus $g$, $\Delta$ be an open disk of $\Sigma$, $G$ be a $\Sigma$-embedded graph, and $v$ be a vertex in $\Delta\cap V(G)$.
	Also, let	${\cal C}$ be a $\Delta$-nested sequence of  cycles  of $G$ of size at least $\funref{reducedfun}(r,k,g)$ that is $r$-tight around $v$, and
	$L$ be a $\Delta$-avoiding $r$-scattered linkage of $G$ of size $k$ such that 
	\begin{itemize}
		\item[(a)] for every $\Delta$-avoiding $r$-scattered linkage $L'$ of $G$ such that $v\in V(L')$ and $L\equiv L'$, it holds that ${\sf cross}_{\cal C}(L)\leq {\sf cross}_{\cal C}(L')$ and
		\item[(b)] for every $\Delta$-avoiding $r$-scattered linkage $L'$ of $G$ such that $v\in V(L')$ and $L\equiv L'$, it holds that ${\sf bridges}_{\Delta}(L)\leq {\sf bridges}_{\Delta}(L')$.
	\end{itemize}
%We will prove that there is a $\Delta$-avoiding $r$-scattered linkage $L' $ of $G\setminus v$ such that $L'\equiv L$ and $L\setminus \Delta=L' \setminus \Delta$.
We assume that $v\in V(L)$, since otherwise the lemma holds for $L'=L$.

We aim to prove that ${\sf bridges}_{\Delta}(L)\leq3g m= \funref{reducedfun}(r,k,g)-\lfloor k/2\rfloor-1$,
since in this case, by \autoref{fewbridges}, we deduce that there is a $\Delta$-avoiding $r$-scattered linkage $L' $ of $G\setminus v$
such that $L'\equiv L$ and $L\setminus \Delta=L' \setminus \Delta$.
\medskip

Suppose, towards a contradiction, that ${\sf bridges}_{\Delta}(L)> 3gm$.
Let $J$ be the bridge representative of $L$ and recall that $J$ is a $\Delta$-outer matching.
Therefore, by \autoref{rainbow}, it follows that $|{\cal E}(J)|\leq 3g$ and therefore there exists an $E\in {\cal E}(J)$ such that $|E|\geq m$.
This implies the existence of a $\Delta$-rainbow of $L$ of size at least $m$.
Notice that since $m=(k+1)\cdot \funref{bridgeredfun}(r,\funref{mazoitfun}(k,g))$,
there exists a clear $\Delta$-rainbow of $L$ of size at least $\funref{bridgeredfun}(r,\funref{mazoitfun}(k,g))$.

Now, by \autoref{japanese}, one of the following holds:
\begin{enumerate}
	\item[(i)] either there is a $\Delta$-avoiding $r$-scattered linkage $L'$ of $G$ such that $v\in V(L')$, $L\equiv L'$, and ${\sf bridges}_{\Delta}(L')<{\sf bridges}_{\Delta}(L)$, or
	\item[(ii)] there exist
	\begin{itemize}
		\item a graph $H\subseteq \cupall {\cal C}$ and an edge-set $F\subseteq E(L)\cap \inter(\Delta)$ such that
		if $\tilde{L}$ is a linkge of $(L\cup H)/F$
		then the $F$-expansion $L'$ of $\tilde{L}$ is an $r$-scattered linkage of $G$, and
		\item a nested sequence of cycles ${\cal C}'$ of $(L\cup H)/F$ of size $\funref{mazoitfun}(k,g)$ that isolates a $\Delta$-bridge of $L$.
	\end{itemize}
\end{enumerate}
Observe that if (i) holds, then we arrive to a contradiction to property (b) of $L$.
In the case that (ii) holds, let $B$ be a $\Delta$-bridge of $L$ that ${\cal C}'$ isolates.
Notice that $B$ is also a $\Delta$-bridge of the linkage $L/F$.
Let $v_{B}$ be the vertex obtained by contracting every edge of $B$ and
let $L^{*}$ be the graph obtained from $L/F$ after applying the same contractions.
Observe that $L^{*}$ is equivalent to $L$.
Now, by \autoref{maz} for $L^{*}\cup H$, $L^{*}$, $v_{B}$, and ${\cal C}'$, we deduce the existence of a linkage $\tilde{L}$ of $(L^{*}\cup H)\setminus v_{B}$ that is equivalent to $L^{*}$.
Notice that $\tilde{L}$ is also a linkage of $(L\cup H)/F$ that does not intersect $B$ and is equivalent to $L$.
Consider now the $F$-expansion $L'$ of $\tilde{L}$ and observe that, since $\tilde{L}$ is equivalent to $L$, the same holds for $L'$.
Moreover, $L'$ is an $r$-scattered linkage in $G$.
The fact that $L'$ is an $r$-scattered linkage in $G$ that is equivalent to $L$, contains $v$ and does not intersect $B$, contradicts property (b) of $L$.
Therefore,
we have that
${\sf bridges}_{\Delta}(L)\leq3g m$ and this concludes the proof of the lemma.
\end{proof}

 \newcommand{\bibremark}[1]{\marginpar{\tiny\bf#1}}
  \newcommand{\biburl}[1]{\url{#1}}

%
%
%\bibliographystyle{plainurl}
%\bibliography{Combinglemma_bibliography,bibliography_logic,complete_irrelog,irrelo_biblio}

\begin{thebibliography}{}

\end{thebibliography}


\begin{thebibliography}{10}

\bibitem{AdlerGK08comp}
Isolde Adler, Martin Grohe, and Stephan Kreutzer.
\newblock Computing excluded minors.
\newblock In {\em Proc. of the 19th Annual ACM-SIAM Symposium on Discrete
  Algorithms (SODA)}, pages 641--650, 2008.
\newblock URL: \url{http://portal.acm.org/citation.cfm?id=1347082.1347153}.

\bibitem{AdlerKKLST17irre}
Isolde Adler, Stavros~G. Kolliopoulos, Philipp~Klaus Krause, Daniel Lokshtanov,
  Saket Saurabh, and Dimitrios~M. Thilikos.
\newblock Irrelevant vertices for the planar disjoint paths problem.
\newblock {\em Journal of Combinatorial Theory, Series B}, 122:815--843, 2017.
\newblock \href {https://doi.org/10.1016/j.jctb.2016.10.001}
  {\path{doi:10.1016/j.jctb.2016.10.001}}.

\bibitem{AgrawalKLPRSZ22dele}
Akanksha Agrawal, Lawqueen Kanesh, Daniel Lokshtanov, Fahad Panolan, M.~S.
  Ramanujan, Saket Saurabh, and Meirav Zehavi.
\newblock Deleting, eliminating and decomposing to hereditary classes are all
  fpt-equivalent.
\newblock In {\em Proc. of the 2022 {ACM-SIAM} Symposium on Discrete Algorithms
  (SODA)}, pages 1976--2004. {SIAM}, 2022.
\newblock \href {https://doi.org/10.1137/1.9781611977073.79}
  {\path{doi:10.1137/1.9781611977073.79}}.

\bibitem{BasteST20acomp}
Julien Baste, Ignasi Sau, and Dimitrios~M. Thilikos.
\newblock A complexity dichotomy for hitting connected minors on bounded
  treewidth graphs: the chair and the banner draw the boundary.
\newblock In {\em Proc. of the 31st Annual {ACM-SIAM} Symposium on Discrete
  Algorithms (SODA)}, pages 951--970, 2020.
\newblock \href {https://doi.org/10.1137/1.9781611975994.57}
  {\path{doi:10.1137/1.9781611975994.57}}.

\bibitem{CattellDDFL00onco}
Kevin Cattell, Michael~J. Dinneen, Rodney~G. Downey, Michael~R. Fellows, and
  Michael~A. Langston.
\newblock On computing graph minor obstruction sets.
\newblock {\em Theoretical Computer Science}, 233:107--127, 2000.
\newblock \href {https://doi.org/10.1016/S0304-3975(97)00300-9}
  {\path{doi:10.1016/S0304-3975(97)00300-9}}.

\bibitem{CyganMPP13thep}
Marek Cygan, D{\'{a}}niel Marx, Marcin Pilipczuk, and Michal Pilipczuk.
\newblock The planar directed k-vertex-disjoint paths problem is
  fixed-parameter tractable.
\newblock In {\em Proc. of the 54th Annual {IEEE} Symposium on Foundations of
  Computer Science (FOCS)}, pages 197--206, 2013.
\newblock \href {https://doi.org/10.1109/FOCS.2013.29}
  {\path{doi:10.1109/FOCS.2013.29}}.

\bibitem{DemaineFHT05subex}
Erik~D. Demaine, Fedor~V. Fomin, Mohammad~Taghi Hajiaghayi, and Dimitrios~M.
  Thilikos.
\newblock Subexponential parameterized algorithms on bounded-genus graphs and
  \emph{H}-minor-free graphs.
\newblock {\em Journal of the {ACM}}, 52(6):866--893, 2005.
\newblock \href {https://doi.org/10.1145/1101821.1101823}
  {\path{doi:10.1145/1101821.1101823}}.

\bibitem{FominGSST21acomp}
Fedor~V. Fomin, Petr~A. Golovach, Ignasi Sau, Giannos Stamoulis, and
  Dimitrios~M. Thilikos.
\newblock A compound logic for modification problems: Big kingdoms fall from
  within, 2021.
\newblock \href {http://arxiv.org/abs/2111.02755} {\path{arXiv:2111.02755}}.

\bibitem{FominGST20analgo}
Fedor~V. Fomin, Petr~A. Golovach, Giannos Stamoulis, and Dimitrios~M. Thilikos.
\newblock An algorithmic meta-theorem for graph modification to planarity and
  {FOL}.
\newblock In {\em Proc. of the 28th Annual European Symposium on Algorithms
  (ESA)}, volume 173 of {\em LIPIcs}, pages 51:1--51:17, 2020.
\newblock \href {https://doi.org/10.4230/LIPIcs.ESA.2020.51}
  {\path{doi:10.4230/LIPIcs.ESA.2020.51}}.

\bibitem{FominGT19modif}
Fedor~V. Fomin, Petr~A. Golovach, and Dimitrios~M. Thilikos.
\newblock {Modification to Planarity is Fixed Parameter Tractable}.
\newblock In {\em Proc. of the 36th International Symposium on Theoretical
  Aspects of Computer Science (STACS)}, volume 126 of {\em Leibniz
  International Proceedings in Informatics (LIPIcs)}, pages 28:1--28:17,
  Dagstuhl, Germany, 2019. Schloss Dagstuhl--Leibniz-Zentrum fuer Informatik.
\newblock \href {https://doi.org/10.4230/LIPIcs.STACS.2019.28}
  {\path{doi:10.4230/LIPIcs.STACS.2019.28}}.

\bibitem{FominLP0Z20hittin}
Fedor~V. Fomin, Daniel Lokshtanov, Fahad Panolan, Saket Saurabh, and Meirav
  Zehavi.
\newblock Hitting topological minors is {FPT}.
\newblock In {\em Proc. of the 52nd Annual {ACM} {SIGACT} Symposium on Theory
  of Computing (STOC)}, pages 1317--1326. {ACM}, 2020.
\newblock \href {https://doi.org/10.1145/3357713.3384318}
  {\path{doi:10.1145/3357713.3384318}}.

\bibitem{FominLRS11sube}
Fedor~V. Fomin, Daniel Lokshtanov, Venkatesh Raman, and Saket Saurabh.
\newblock Subexponential algorithms for partial cover problems.
\newblock {\em Information Processing Letters}, 111(16):814--818, 2011.
\newblock \href {https://doi.org/10.1016/j.ipl.2011.05.016}
  {\path{doi:10.1016/j.ipl.2011.05.016}}.

\bibitem{FominLST12line}
Fedor~V. Fomin, Daniel Lokshtanov, Saket Saurabh, and Dimitrios~M. Thilikos.
\newblock Linear kernels for (connected) dominating set on \emph{H}-minor-free
  graphs.
\newblock In Yuval Rabani, editor, {\em Proceedings of the Twenty-Third Annual
  {ACM-SIAM} Symposium on Discrete Algorithms, {SODA} 2012, Kyoto, Japan,
  January 17-19, 2012}, pages 82--93. {SIAM}, 2012.
\newblock \href {https://doi.org/10.1137/1.9781611973099.7}
  {\path{doi:10.1137/1.9781611973099.7}}.

\bibitem{Frank90pack}
Andr\'as Frank.
\newblock Packing paths, cuts and circuits – a survey.
\newblock In B.~Korte, L.~Lovász, H.J. Prömel, and A.~Schrijver, editors,
  {\em Paths, Flows,VLSI-Layout}, pages 49--100. Springer-Verlag, Berlin, 1990.
\newblock \href {https://doi.org/10.1007/978-1-4419-6045-0_14}
  {\path{doi:10.1007/978-1-4419-6045-0_14}}.

\bibitem{GolovachKPT09indu}
Petr~A. Golovach, M.~Kami{\'n}ski, D.~Paulusma, and D.~M. Thilikos.
\newblock Induced packing of odd cycles in a planar graph.
\newblock In {\em 20th International Symposium on Algorithms and Computation
  (ISAAC 2009)}, volume 5878 of {\em LNCS}, pages 514--523. Springer, Berlin,
  2009.

\bibitem{GolovachKMT17thep}
Petr~A. Golovach, Marcin Kaminski, Spyridon Maniatis, and Dimitrios~M.
  Thilikos.
\newblock The parameterized complexity of graph cyclability.
\newblock {\em SIAM Journal on Discrete Mathematics}, 31(1):511--541, 2017.
\newblock \href {https://doi.org/10.1137/141000014}
  {\path{doi:10.1137/141000014}}.

\bibitem{GolovachKP13}
Petr~A. Golovach, Dieter Kratsch, and Dani{\"{e}}l Paulusma.
\newblock Detecting induced minors in at-free graphs.
\newblock {\em Theoretical Computer Science}, 482:20--32, 2013.
\newblock \href {https://doi.org/10.1016/j.tcs.2013.02.029}
  {\path{doi:10.1016/j.tcs.2013.02.029}}.

\bibitem{GolovachST20hitti}
Petr~A. Golovach, Giannos Stamoulis, and Dimitrios~M. Thilikos.
\newblock Hitting topological minor models in planar graphs is fixed parameter
  tractable.
\newblock In {\em Proc. of the 31st Annual {ACM-SIAM} Symposium on Discrete
  Algorithms (SODA)}, pages 931--950, 2020.
\newblock \href {https://doi.org/10.1137/1.9781611975994.56}
  {\path{doi:10.1137/1.9781611975994.56}}.

\bibitem{GroheKMW11find}
Martin Grohe, Ken{-}ichi Kawarabayashi, D{\'{a}}niel Marx, and Paul Wollan.
\newblock Finding topological subgraphs is fixed-parameter tractable.
\newblock In {\em Proc. of the 43rd {ACM} Symposium on Theory of Computing
  (STOC)}, pages 479--488. {ACM}, 2011.
\newblock \href {https://doi.org/10.1145/1993636.1993700}
  {\path{doi:10.1145/1993636.1993700}}.

\bibitem{HeggernesHLP11obta}
Pinar Heggernes, Pim van~'t Hof, Daniel Lokshtanov, and Christophe Paul.
\newblock Obtaining a bipartite graph by contracting few edges.
\newblock In {\em IARCS Annual Conference on Foundations of Software Technology
  and Theoretical Computer Science, (FSTTCS 2011)}, pages 217--228, 2011.

\bibitem{TakehiroKPT09para}
Takehiro Ito, Marcin Kami{\'n}ski, Dani{\"e}l Paulusma, and Dimitrios~M.
  Thilikos.
\newblock Parameterizing cut sets in a graph by the number of their components.
\newblock {\em Theor. Comput. Sci.}, 412(45):6340--6350, 2011.

\bibitem{JansenK021verte}
Bart M.~P. Jansen, Jari J.~H. de~Kroon, and Michal W{\l}odarczyk.
\newblock Vertex deletion parameterized by elimination distance and even less.
\newblock In {\em Proc. of the 53rd Annual {ACM} Symposium on Theory of
  Computing (STOC)}, pages 1757--1769, 2021.
\newblock \href {https://doi.org/10.1145/3406325.3451068}
  {\path{doi:10.1145/3406325.3451068}}.

\bibitem{KaminskiN12find}
Marcin Kami{\'n}ski and Naomi Nishimura.
\newblock Finding an induced path of given parity in planar graphs in
  polynomial time.
\newblock In {\em Proceedings of the Twenty-Second Annual ACM-SIAM Symposium on
  Discrete Algorithms, (SODA 2011)}, pages 656--670. ACM, 2012.

\bibitem{KaminskiT12cont}
Marcin Kami{\'n}ski and Dimitrios~M. Thilikos.
\newblock Contraction checking in graphs on surfaces.
\newblock In {\em 29th International Symposium on Theoretical Aspects of
  Computer Science, (STACS)}, pages 182--193, 2012.

\bibitem{Karp72redu}
Richard Karp.
\newblock Reducibility among combinatorial problems.
\newblock volume~40, pages 85--103, 01 1972.
\newblock \href {https://doi.org/10.1007/978-3-540-68279-0_8}
  {\path{doi:10.1007/978-3-540-68279-0_8}}.

\bibitem{Kawarabayashi07half}
{Ken-ichi} Kawarabayashi.
\newblock Half integral packing, {E}rd{\H o}s-{P}{\'o}sa-property and graph
  minors.
\newblock In {\em Proceedings of the eighteenth annual ACM-SIAM symposium on
  Discrete algorithms}, SODA '07, pages 1187--1196, Philadelphia, PA, USA,
  2007. Society for Industrial and Applied Mathematics.
\newblock URL: \url{http://dl.acm.org/citation.cfm?id=1283383.1283511}.

\bibitem{Kawarabayashi09plan}
{Ken-ichi} Kawarabayashi.
\newblock Planarity allowing few error vertices in linear time.
\newblock In {\em 50th Annual IEEE Symposium on Foundations of Computer
  Science}, FOCS 2009, pages 639--648, 2009.

\bibitem{KawarabayashiK08thei}
Ken{-}ichi Kawarabayashi and Yusuke Kobayashi.
\newblock The induced disjoint paths problem.
\newblock In Andrea Lodi, Alessandro Panconesi, and Giovanni Rinaldi, editors,
  {\em Integer Programming and Combinatorial Optimization, 13th International
  Conference, {IPCO} 2008, Bertinoro, Italy, May 26-28, 2008, Proceedings},
  volume 5035 of {\em Lecture Notes in Computer Science}, pages 47--61.
  Springer, 2008.
\newblock \href {https://doi.org/10.1007/978-3-540-68891-4\_4}
  {\path{doi:10.1007/978-3-540-68891-4\_4}}.

\bibitem{Kawarabayashi09algo}
Ken-ichi Kawarabayashi and Yusuke Kobayashi.
\newblock Algorithms for finding an induced cycle in planar graphs and bounded
  genus graphs.
\newblock In {\em Proceedings of the 20th Annual ACM-SIAM Symposium on Discrete
  Algorithms}, SODA '09, page 1146–1155, USA, 2009. Society for Industrial
  and Applied Mathematics.
\newblock \href {https://doi.org/10.5555/1496770.1496894}
  {\path{doi:10.5555/1496770.1496894}}.

\bibitem{KawarabayashiK10impr}
{Ken-ichi} Kawarabayashi and Yusuke Kobayashi.
\newblock An improved algorithm for the half-disjoint paths problem.
\newblock {\em SIAM J. Discrete Math.}, 25(3):1322--1330, 2011.

\bibitem{KawarabayashiK12alin}
Ken{-}ichi Kawarabayashi and Yusuke Kobayashi.
\newblock A linear time algorithm for the induced disjoint paths problem in
  planar graphs.
\newblock {\em Journal of Computer and System Sciences}, 78(2):670--680, 2012.
\newblock \href {https://doi.org/10.1016/j.jcss.2011.10.004}
  {\path{doi:10.1016/j.jcss.2011.10.004}}.

\bibitem{KawarabayashiKR12thedis}
Ken{-}ichi Kawarabayashi, Yusuke Kobayashi, and Bruce~A. Reed.
\newblock The disjoint paths problem in quadratic time.
\newblock {\em Journal of Combinatorial Theory, Series {B}}, 102(2):424--435,
  2012.
\newblock \href {https://doi.org/10.1016/j.jctb.2011.07.004}
  {\path{doi:10.1016/j.jctb.2011.07.004}}.

\bibitem{KawarabayashiKM10link}
{Ken-ichi} Kawarabayashi, Stephan Kreutzer, and Bojan Mohar.
\newblock Linkless and flat embeddings in 3-space and the unknot problem.
\newblock In {\em Proceedings of the 2010 annual symposium on Computational
  geometry}, SoCG '10, pages 97--106, New York, NY, USA, 2010. ACM.

\bibitem{KawarabayashiLR10reco}
{Ken-ichi} Kawarabayashi, Zhentao Li, and Bruce~A. Reed.
\newblock Recognizing a totally odd ${K}_{4}$-subdivision, parity 2-disjoint
  rooted paths and a parity cycle through specified elements.
\newblock In {\em Proceedings of the Twenty-First Annual ACM-SIAM Symposium on
  Discrete Algorithms, SODA 201}, pages 318--328, 2010.

\bibitem{KawarabayashiMR08asim}
{Ken-ichi} Kawarabayashi, Bojan Mohar, and Bruce~A. Reed.
\newblock A simpler linear time algorithm for embedding graphs into an
  arbitrary surface and the genus of graphs of bounded tree-width.
\newblock In {\em 49th Annual IEEE Symposium on Foundations of Computer
  Science, FOCS 2008}, pages 771--780, 2008.

\bibitem{KawarabayashiR09anea}
Ken-ichi Kawarabayashi and Bruce Reed.
\newblock A nearly linear time algorithm for the half integral parity disjoint
  paths packing problem.
\newblock In {\em Proc. of the 20th Annual ACM-SIAM Symposium on Discrete
  Algorithms}, SODA '09, page 1183–1192, USA, 2009. Society for Industrial
  and Applied Mathematics.

\bibitem{KawarabayashiR09hadw}
{Ken-ichi} Kawarabayashi and Bruce~A. Reed.
\newblock Hadwiger's conjecture is decidable.
\newblock In {\em 41st Annual ACM Symposium on Theory of Computing, (STOC
  2009)}, pages 445--454, 2009.

\bibitem{KawarabayashiR10oddc}
{Ken-ichi} Kawarabayashi and Bruce~A. Reed.
\newblock Odd cycle packing.
\newblock In {\em Proceedings of the 42nd ACM Symposium on Theory of Computing,
  STOC 2010}, pages 695--704, 2010.

\bibitem{KawarabayashiW2010asho}
Ken{-}ichi Kawarabayashi and Paul Wollan.
\newblock A shorter proof of the graph minor algorithm: the unique linkage
  theorem.
\newblock In {\em 42nd {ACM} Symposium on Theory of Computing, STOC 2010},
  pages 687--694. ACM, 2010.

\bibitem{KawarabayashiW10asho}
Ken{-}ichi Kawarabayashi and Paul Wollan.
\newblock A shorter proof of the graph minor algorithm: the unique linkage
  theorem.
\newblock In {\em Proc. of the 42nd {ACM} Symposium on Theory of Computing
  {(STOC)}}, pages 687--694. {ACM}, 2010.
\newblock \href {https://doi.org/10.1145/1806689.1806784}
  {\path{doi:10.1145/1806689.1806784}}.

\bibitem{Kleinberg98deci}
Jon~M. Kleinberg.
\newblock Decision algorithms for unsplittable flow and the half-disjoint paths
  problem.
\newblock In {\em Proc. of the 30th Annual ACM Symposium on Theory of
  Computing}, STOC '98, page 530–539, New York, NY, USA, 1998. Association
  for Computing Machinery.
\newblock \href {https://doi.org/10.1145/276698.276867}
  {\path{doi:10.1145/276698.276867}}.

\bibitem{KramerL84thec}
Mark~R. Kramer and Jan van Leeuwen.
\newblock The complexity of wire-routing and finding minimum area layouts for
  arbitrary {VLSI} circuits.
\newblock {\em Advances in Comp. Research}, 2:129--146, 1984.

\bibitem{DLindermayrSV20elimi}
Alexander Lindermayr, Sebastian Siebertz, and Alexandre Vigny.
\newblock Elimination distance to bounded degree on planar graphs.
\newblock In {\em Proc. of the 45th International Symposium on Mathematical
  Foundations of Computer Science (MFCS)}, volume 170 of {\em LIPIcs}, pages
  65:1--65:12, 2020.
\newblock \href {https://doi.org/10.4230/LIPIcs.MFCS.2020.65}
  {\path{doi:10.4230/LIPIcs.MFCS.2020.65}}.

\bibitem{Lynch75thee}
James~F. Lynch.
\newblock The equivalence of theorem proving and the interconnection problem.
\newblock {\em SIGDA Newsletter}, 5(3):31–36, sep 1975.
\newblock \href {https://doi.org/10.1145/1061425.1061430}
  {\path{doi:10.1145/1061425.1061430}}.

\bibitem{MarxS07obta}
D{\'{a}}niel Marx and Ildik{\'{o}} Schlotter.
\newblock Obtaining a planar graph by vertex deletion.
\newblock {\em Algorithmica}, 62(3-4):807--822, 2012.
\newblock \href {https://doi.org/10.1007/s00453-010-9484-z}
  {\path{doi:10.1007/s00453-010-9484-z}}.

\bibitem{Mazoit13asin}
Frédéric Mazoit.
\newblock A single exponential bound for the redundant vertex theorem on
  surfaces, 2013.
\newblock \href {http://arxiv.org/abs/1309.7820} {\path{arXiv:1309.7820}}.

\bibitem{MiddendorfP93onth}
Matthias Middendorf and Frank Pfeiffer.
\newblock On the complexity of the disjoint paths problem.
\newblock {\em Combinatorica}, 13(1):97--107, 1993.
\newblock \href {https://doi.org/10.1007/BF01202792}
  {\path{doi:10.1007/BF01202792}}.

\bibitem{MoharT01grap}
Bojan Mohar and Carsten Thomassen.
\newblock {\em Graphs on Surfaces}.
\newblock Johns Hopkins series in the mathematical sciences. Johns Hopkins
  University Press, 2001.
\newblock URL: \url{https://jhupbooks.press.jhu.edu/title/graphs-surfaces}.

\bibitem{NavesS09mult}
Guyslain Naves and Andr{\'a}s Seb{\H{o}}.
\newblock {\em Multiflow Feasibility: An Annotated Tableau}, pages 261--283.
\newblock Springer Berlin Heidelberg, Berlin, Heidelberg, 2009.
\newblock \href {https://doi.org/10.1007/978-3-540-76796-1_12}
  {\path{doi:10.1007/978-3-540-76796-1_12}}.

\bibitem{Reed95root}
Bruce~A. Reed.
\newblock Rooted routing in the plane.
\newblock {\em Discrete Applied Mathematics}, 57(2-3):213--227, 1995.
\newblock \href {https://doi.org/10.1016/0166-218X(94)00104-L}
  {\path{doi:10.1016/0166-218X(94)00104-L}}.

\bibitem{ReedRSS91find}
Bruce~A. Reed, Neil Robertson, Alexander Schrijver, and Paul~D. Seymour.
\newblock Finding disjoint trees in planar graphs in linear time.
\newblock In Neil Robertson and Paul~D. Seymour, editors, {\em Graph Structure
  Theory, Proceedings of a {AMS-IMS-SIAM} Joint Summer Research Conference on
  Graph Minors held June 22 to July 5, 1991, at the University of Washington,
  Seattle, {USA}}, volume 147 of {\em Contemporary Mathematics}, pages
  295--301. American Mathematical Society, 1991.

\bibitem{RoSe90}
Neil Robertson and Paul~D. Seymour.
\newblock An outline of a disjoint paths algorithm.
\newblock {\em Paths, Flows and VLSI Design, Algorithms and Combinatorics},
  9:267--292, 1990.

\bibitem{RobertsonS95GMXIII}
Neil Robertson and Paul~D. Seymour.
\newblock Graph minors . {XIII.} {The} disjoint paths problem.
\newblock {\em Journal of Combinatorial Theory, Series B}, 63(1):65--110, 1995.
\newblock \href {https://doi.org/10.1006/jctb.1995.1006}
  {\path{doi:10.1006/jctb.1995.1006}}.

\bibitem{RobertsonS09XXI}
Neil Robertson and Paul~D. Seymour.
\newblock Graph minors. {XXI.} graphs with unique linkages.
\newblock {\em Journal of Combinatorial Theory, Series B}, 99(3):583--616,
  2009.
\newblock \href {https://doi.org/10.1016/j.jctb.2008.08.003}
  {\path{doi:10.1016/j.jctb.2008.08.003}}.

\bibitem{RobertsonS12XXII}
Neil Robertson and Paul~D. Seymour.
\newblock Graph minors. {XXII.} {I}rrelevant vertices in linkage problems.
\newblock {\em Journal of Combinatorial Theory, Series B}, 102(2):530--563,
  2012.
\newblock URL: \url{http://dx.doi.org/10.1016/j.jctb.2007.12.007}, \href
  {https://doi.org/10.1016/j.jctb.2007.12.007}
  {\path{doi:10.1016/j.jctb.2007.12.007}}.

\bibitem{SauST20anftp}
Ignasi Sau, Giannos Stamoulis, and Dimitrios~M. Thilikos.
\newblock {An FPT-Algorithm for Recognizing $k$-Apices of Minor-Closed Graph
  Classes}.
\newblock In {\em Proc. of the 47th International Colloquium on Automata,
  Languages, and Programming (ICALP)}, volume 168 of {\em LIPIcs}, pages
  95:1--95:20, 2020.
\newblock \href {https://doi.org/10.4230/LIPIcs.ICALP.2020.95}
  {\path{doi:10.4230/LIPIcs.ICALP.2020.95}}.

\bibitem{SauST21amor}
Ignasi Sau, Giannos Stamoulis, and Dimitrios~M. Thilikos.
\newblock {A more accurate view of the Flat Wall Theorem}, 2021.
\newblock \href {http://arxiv.org/abs/2102.06463} {\path{arXiv:2102.06463}}.

\bibitem{SauST21kapiI}
Ignasi Sau, Giannos Stamoulis, and Dimitrios~M. Thilikos.
\newblock $k$-apices of minor-closed graph classes. {I.} {Bounding the
  obstructions}, 2021.
\newblock \href {http://arxiv.org/abs/2103.00882} {\path{arXiv:2103.00882}}.

\bibitem{SauST21kapiII}
Ignasi Sau, Giannos Stamoulis, and Dimitrios~M. Thilikos.
\newblock $k$-apices of minor-closed graph classes. {II.} {Parameterized
  algorithms}.
\newblock {\em ACM Transactions on Algorithms}, 2022.
\newblock To appear, available online.
\newblock \href {https://doi.org/10.1145/3519028} {\path{doi:10.1145/3519028}}.

\bibitem{Schrijver03comb}
Alexander Schrijver.
\newblock {\em Combinatorial Optimization. Polyhedra and Efficiency}, volume~A
  of {\em Algorithms and Combinatorics}.
\newblock Springer Berlin, Heidelberg.

\bibitem{SeymourT93grap}
Paul~D. Seymour and Robin Thomas.
\newblock Graph searching and a min-max theorem for tree-width.
\newblock {\em Journal of Combinatorial Theory, Series B}, 58(1):22--33, 1993.
\newblock \href {https://doi.org/10.1006/jctb.1993.1027}
  {\path{doi:10.1006/jctb.1993.1027}}.

\bibitem{Thilikos12grap}
Dimitrios~M. Thilikos.
\newblock Graph minors and parameterized algorithm design.
\newblock In {\em The Multivariate Algorithmic Revolution and Beyond - Essays
  Dedicated to Michael R. Fellows on the Occasion of His 60th Birthday}, pages
  228--256, 2012.

\bibitem{Vygen95npco}
Jens Vygen.
\newblock Np-completeness of some edge-disjoint paths problems.
\newblock {\em Discrete Applied Mathematics}, 61(1):83--90, 1995.
\newblock URL:
  \url{https://www.sciencedirect.com/science/article/pii/0166218X93E0177Z},
  \href {https://doi.org/https://doi.org/10.1016/0166-218X(93)E0177-Z}
  {\path{doi:https://doi.org/10.1016/0166-218X(93)E0177-Z}}.

\end{thebibliography}

\end{document}